\documentclass[fceqn,runningheads,11.85pt]{llncs}
\usepackage{geometry}
\usepackage{wrapfig,booktabs}
\usepackage{amssymb}
\usepackage{amsmath}
\usepackage{graphicx}
\usepackage{subfigmat}
\usepackage{url}
\usepackage{amsfonts}
\usepackage{latexsym} 
\usepackage{wasysym}
\usepackage{newlfont}
\usepackage{leqno}
\usepackage{etoolbox}

\usepackage{lineno}
\usepackage[usenames]{color}
\usepackage{fancybox}
\usepackage{sidecap}
\usepackage{xcolor, soul}
\definecolor{airforceblue}{rgb}{0.90, 0.44, 0.86}
\definecolor{ashgrey}{rgb}{0.7, 0.75, 0.71}
\sethlcolor{airforceblue} 
\usepackage{multirow}
\usepackage[printonlyused]{acronym}
\newcommand*{\BLACKBOX}{\hfill\ensuremath{\blacksquare}}

\usepackage{stmaryrd}

\usepackage{xifthen}
\newcommand{\sMP}[1][]{\ifthenelse{\isempty{#1}}{^{(i)}}{^{(#1)}}}

\newcommand{\MatTwo}[4]
  {\begin{bmatrix}
    #1 & #2\\
    #3 & #4
\end{bmatrix}}
  
\newcommand{\VecTwo}[2]
  {\begin{bmatrix}
    #1 \\
    #2
\end{bmatrix}}
\everymath{\displaystyle}
{\begin{Sbox}\begin{minipage}}%
{\end{minipage}\end{Sbox}\fbox{\TheSbox}}

\newtheorem{assume}{Assumption}

\numberwithin{equation}{section}
    \urldef{\mailsa}\path|ulrich.langer@ricam.oeaw.ac.at|
    \urldef{\mailsd}\path|ioannis.toulopoulos@ricam.oeaw.ac.at|
    \urldef{\mailse}\path|christoph.hofer@jku.at|
    \newcommand{\keywords}[1]{\par\addvspace\baselineskip  
     \noindent\keywordname\enspace\ignorespaces#1}

\begin{document}
\mainmatter  
\title{ Discontinuous Galerkin  Isogeometric Analysis
         on Non-matching Segmentation: Error Estimates and Efficient Solvers }

\titlerunning{dG IgA on non-matching segmantations: error estimates and solvers}


\author{
Christoph Hofer$^1$
\and     Ulrich Langer$^{1,2}$ 
\and Ioannis Toulopoulos$^2$
}
\authorrunning{C. Hofer,  U. Langer, I. Toulopoulos }

\institute{ $1$ Johannes Kepler University (JKU),\\
				Altenbergerstr. 69, A-4040 Linz, Austria,\\[1ex]
	    $2$ Johann Radon Institute for Computational and Applied Mathematics (RICAM),\\
				Austrian Academy of Sciences\\
				Altenbergerstr. 69, A-4040 Linz, Austria,\\[1ex]
\mailse\\
\mailsa \\
\mailsd  
 }

%

\noindent
\maketitle

\begin{abstract}
The Isogeometric Analysis (IgA) of boundary value problems in complex domains
often requires a decomposition of the computational domain into patches
such that each of which can be parametrized by the so-called geometrical mapping.
In this paper, we develop discontinuous Galerkin (dG) IgA techniques  for solving elliptic diffusion problems on  decompositions that can include non-matching parametrizations of the interfaces,
  i.e., the interfaces of the adjacent patches may be not identical. 
  The lack of the exact parametrization of the patches leads to the creation of gap and overlapping regions between the patches.  
  This does not allow the immediate use of the classical numerical fluxes that are known in the literature.
  The unknown normal fluxes of the solution on the non-matching interfaces are approximated by Taylor expansions 
  using the values of the solution computed on the boundary of the patches. 
  These approximations 
  are used in order to build up the numerical fluxes of the final dG IgA scheme and to couple the local patch-wise discrete problems. 
    The resulting linear systems are solved by using efficient domain decomposition methods based on the tearing and interconnecting technology.
  We present numerical results of a series of test problems that validate the theoretical estimates presented.
\end{abstract}

\section{Introduction}
\label{intro}
During last decade, there has been an increasing interest in solving elliptic boundary value problems 
in complicated domains using the Isogeometric Analysis (IgA) methodologies.
The core idea of IgA is to use  the same smooth and high order superior finite dimensional spaces, e.g., B-splines, NURBS, 
for parametrizing  the computational domain and for approximating the  solution of the 
Partial Differential Equation (PDE) model of interest.
The IgA approach for discretizing PDEs and its benefits  have  been highlighted and  discussed in many 
 publications,  see, e.g., the monograph   \cite{LT:Hughes_IGAbook_2009} and the survey paper \cite{VeigaBuffaSangalli_2014}.
 For example, the   simple  and easily materialized algorithm for the construction  of the B-spline basis functions with possible high  smoothness 
 helps extremely in the production of high order approximate solutions. Furthermore IgA offers a particular suitable frame
 for developing $hp$ (here $p$ is the B-spline degree) adaptivity methods with a possible change of the inter-element smoothness, 
 \cite{LT:Hughes_IGAbook_2009}.
 \par
 In most realistic applications, the computational domain $\Omega$ is decomposed into
a union of non-overlapping subdomains   $\overline{\Omega}=\cup_{i=1}^N \overline{\Omega_i}$, 
where each $\Omega_i$ is referred to as a patch, and then it is  viewed as an image of a parametrization mapping.
These mappings are linear combinations of the  B-splines, NURBS, etc, basis functions.
The vector valued coefficients of the parametrizations  describe the shape of the patch. They are called control points.
There have been presented several  segmentation  techniques   
and procedures for splitting complex domains into simpler subdomains and defining their control nets, see, e.g., 
\cite{HLT:JuettlerKaplNguyenPanPauley:2014a},
 \cite{HLT:PauleyNguyenMayerSpehWeegerJuettler:2015a},  
 \cite{Hoschek_Lasser_CAD_book_1993} for a more comprehensive
analysis. 
Usually, one obtains  compatible parametrizations of the patches in the sense that 
the  parameterizations of adjacent patches lead to identical interfaces. 
However, some serious difficulties can arise, especially, when the {patches}
 differ topologically a lot  from a cube. 
 In particular, the control points  related to an interface  may not have been  appropriately defined,
which results in parametrizations of patches that create gap and overlapping regions. 
As a direct consequence, the use of these paramatrization mappings
causes  additional errors in the IgA discretization process of the PDE. 
We speak about segmentation crimes that causes additional errors in the IgA simulation.
\par
In this paper, we are interested in  investigating the solution of  elliptic diffusion
problems using IgA on multipatch decompositions  that include  gaps and overlaps. 
 The missing capability of the IgA patch parametrizations to represent exactly the physical interfaces
  prevent us from using directly 
  the classical interface conditions,
  of the solution in order to  construct the numerical fluxes in the dG scheme.
  We need to derive new    conditions on the interior faces of the 
  gap/overllaping regions and in order to construct the numerical fluxes and to ensure communication of the discrete 
  problems patch-wise problems. 
 We extent the ideas presented in \cite{HoferLangerToulopoulos_2015a,HoferToulopoulos_IGA_Gaps_2015a},
 where discontinuous Galerkin (dG) methods for  IgA 
 have been developed on decompositions including only gaps. In particular, 
  Taylor expansions are also used 
 on the interior faces of the overlapping regions, 
 to approximate the normal fluxes on the overlapping boundaries. 
   Finally, these Taylor  expansions are used   to build  appropriate numerical fluxes
  for the final dG IgA scheme, which in turn help on coupling the local patch-wise discrete problems.
In \cite{LT:LangerToulopoulos:2014a}, dG IgA methods have been analyzed for matching interface parametrization. 
In the present paper,  we use the same results 
in order to express the error bounds related to the approximation properties of the B-splines. Following 
the analysis in \cite{HoferLangerToulopoulos_2015a,HoferToulopoulos_IGA_Gaps_2015a}, 
we investigate  the effect of the approximation of the fluxes on the gap/overlapping 
regions  on the  accuracy of the proposed method. 
 The model problem that we are going to study is a linear diffusion problem with a discontinuous diffusion coefficient.
 For simplicity, we treat the two cases of gap and overlapping regions separately
using decompositions formed by only two patches. 
Then, we proceed and express the final dG IgA scheme in the general case of multipatch decompositions. 
Regarding the case of having overlapping patches, our analysis takes into account the possibility of co-existence of  different diffusion coefficients, which leads to the discretization
of two different  perturbed problems on the overlapping region. 
This means that we analyze the error coming from the B-spline approximations,
 the normal flux approximations and the consistency error, i.e., we estimate the distance between the solutions of the peturbated  problems and the solution of the physical problems.    
We show a priori error estimates in the classical  dG-norm, which depends on the accuracy of the normal flux approximation and 
is expressed in terms of the mesh size $h$ and the parameter $d_{M}=\max(d_g,d_o)$, where $d_g$ 
 is the maximum distance between the  diametrically opposite points on the gap boundaries and 
$d_o$ is  the maximum distance of the opposite points on  the  overlapping boundaries, respectively. 
In particular, we  show that,  if the IgA space defined on the patches  has the approximation power 
$h^p$ and $d_M$  is $\mathcal{O}(h^{p+\frac{1}{2}})$, 
then we obtain optimal convergence rate.
\par
The second contribution of this paper is the presentation of efficient solvers based on the Dual-Primal Isogeometric Tearing and Interconnecting  (IETI-DP) method \cite{HLT:HoferLanger:2016b}. More precisely, we use the dG-IETI-DP method which can handle formulations based on dG, see \cite{HLT:HoferLanger:2016a}. The IETI-DP method provides a quasi optimal condition number bound on the condition number of the preconditioned 
linear system with respect to the mesh size $h$ and robustness with respect to jumps in the diffusion coefficient across interfaces, see \cite{HLT:HoferLanger:2016b,HLT:BeiraoChoPavarinoScacchi:2013a}. For versions using the more sophisticated \emph{deluxe scaling}, see \cite{HLT:BeiraoPavarinoScacchiWidlundZampini:2014a}.  Numerical examples for the dG-IETI-DP methods also indicate the same behaviour, see \cite{HLT:HoferLanger:2016a}. For the proof of the finite element version, we refer to
\cite{HLT:DryjaGalvisSarkis:2013a,HLT:DryjaSarkis:2014a,HLT:DryjaGalvisSarkis:2007a}.\\
Finally, we investigate the accuracy of the proposed method and validate the theoretical estimates by
solving a series of test problems on decompositions with gap and overlapping regions. The first three examples 
are considered in two-dimensional domains with known exact solutions. 
In the last two examples, we consider complicated three-dimensional domains. 
Through the investigation of the numerical convergence rates, 
we have found that in the cases where $d_M$ is of order $\mathcal{O}(h^{p+\frac{1}{2}})$, the rates
are optimal, as it is predicted by the theoretical analysis.
 \\
 Lastly, we mention that several   techniques have recently  been investigated   for coupling 
 non-matching (or non-conforming) subdomain parametrizations in some weak sense. 
 In \cite{Ruess_NitcCoplPtac_2015} and \cite{Nguyen_Nitche_3Dcoupli_2014}, 
Nitsche's method have  been applied to enforce  weak coupling conditions along trimmed B-spline patches. 
 In \cite{Apostolatos_Schmidt_Wuchner_Bletzinger_IJNUMEng_2014}, the most common techniques for 
 imposing   weakly the continuity of the solution on the interfaces have been  applied and tested on nonlinear 
 elasticity problems. The numerical tests have been performed on non-matching grid parametrizations. 
 {
 Recently,  overlapping domain decomposition methods 
 combined  with local non-matching mesh refinement techniques have been investigated 
 for solving diffusion problems in 
 complex domains in \cite{bercovier2015overlapping}. }
Furthermore, mortar
 methods have been developed in the IgA  context utilizing different
 B-spline degrees for the Lagrange multiplier in \cite{Brivadis_IgAMortar_2015}. 
 The method has been applied to
decompositions with non-matching interface parametrizations and studied numerically.  
\\
The paper is organized as follows. In Section 2, some notations, the weak 
form of the problem and  the definition of the  B-spline spaces are presented.  
Furthermore, we give a description of the gap region. 
In Section 3,  we derive  the problem in $\Omega\setminus\overline{\Omega}_g$, 
the approximation of the normal fluxes on the $\partial \Omega_g$, and the 
dG IgA scheme. In the last part of this section, we estimate the  remainder terms in the Taylor expansion,
and derive a priori error estimates.
Section 4 is devoted to efficient solution strategies. 
 Finally, in Section 5, we present numerical tests for validating the theoretical results 
 on two- and three-dimensional test problems. The paper closes with some conclusions in Section 6. 
%
%
%

\section{The model problem}
\subsection{Preliminaries}
Let $\Omega$ be a bounded Lipschitz domain in $\mathbb{R}^d,\,d=2,3$, 
and let $\alpha=(\alpha_1,\ldots,\alpha_d)$ be a multi-index of non-negative integers 
$\alpha_1,\ldots,\alpha_d$
with degree $|\alpha| = \alpha_1+\cdots+\alpha_d$. 
For any  $\alpha$,  we define the differential operator
$D^\alpha=D_1^{\alpha_1}\ldots D_d^{\alpha_d}$,
with $D_j = \partial / \partial x_j$, $j=1,\ldots,d$, and $D^{(0,\ldots,0)}\phi=\phi$.
For a  non-negative integer $m$, let $C^{m}(\Omega)$ denote the space
of all functions $\phi:\Omega \rightarrow \mathbb{R}$, whose 
partial derivatives
$D^\alpha \phi$ of all orders $|\alpha| \leq m$ are continuous in $\Omega$. 
Let  $\ell$ be a non-negative integer. 
 As usual, 
 $L^2(\Omega)$ denotes the  space of square integrable functions  endowed with the norm
 $\|\phi\|_{L^2(\Omega)} = \big(\int_{\Omega}|\phi(x)|^2\,dx\big)^{\frac{1}{2}}$, 
 and  $L^{\infty}(\Omega)$ denotes the functions that are essentially bounded. Also 
 $
 	H^{\ell}(\Omega)=\{\phi\in L^2(\Omega): D^\alpha \phi \in L^{2}(\Omega),\,\text{for all}\,|\alpha| \leq \ell \},
$
denote the standard Sobolev spaces endowed with the following  norms 
 \begin{equation*}
 \|\phi\|_{H^{\ell}(\Omega)} = \big(\sum_{0\leq |\alpha| \leq \ell} \|D^{\alpha}\phi\|_{L^2(\Omega)}^p\big)^{\frac{1}{2}}.
 \end{equation*}
 $H^{\frac{1}{2}}(\partial \Omega)$ denotes the trace space of $H^{1}(\Omega)$.
 We identify $L^2$ and $H^{0}$ and also define  
$
 	H^{1}_0(\Omega) =\{\phi\in H^{1}(\Omega):\phi=0\,\text{on}\, \partial \Omega \}.
$
We recall H\"older's and Young's inequalities 

\begin{equation}\label{HolderYoung}
  \left|  \int_{\Omega}u v\,dx \right| \leq  \|u\|_{L^2(\Omega)}\|v\|_{L^2(\Omega)}
\quad \mbox{and} \quad
\left|  \int_{\Omega}u v\,dx \right|\leq \frac{\epsilon}{2}\|u\|^2_{L^2(\Omega)}+  \frac{1}{2\epsilon}\|v\|^2_{L^2(\Omega)},
\end{equation}
that hold for all $u\in L^2(\Omega)$ and $v\in L^2(\Omega)$ 
and for any fixed $\epsilon \in (0,\infty)$.
\par
For the derivation of the dG IgA scheme,  we will need 
appropriate  approximations of the solution and its normal fluxes
{\color{black}
on the boundary of the gap and overlapping regions. 
}
We will derive such approximations by means of  Taylor's theorem. 
We recall 
Taylor's formula with integral remainder
\begin{flalign}\label{7}
 f(1)=f(0)+\sum_{j=1}^{m-1}\frac{1}{j!}f^{(j)}(0)+\frac{1}{(m-1)!}\int_0^1s^{m-1}f^{(m)}(1-s)\,d s.
\end{flalign}
that holds for all $f\in C^m[0,1]$.
Now, let  $u\in C^m(\Omega)$ with $m\geq 2$ and
 $x,y\in \Omega$. 
We define $f(s)=u(x+s(y-x))$ and by the chain rule we obtain
\begin{flalign}\label{7_a}
 f^{(j)}(s)=\sum_{|\alpha|=j}\frac{j!}{\alpha !}D^\alpha  u(x +s(y-x))(y-x)^\alpha,
\end{flalign}
where $\alpha!=\alpha_1!\ldots\alpha_d!$ and $ (y-x)^\alpha=(y_{1}-x_{1})^{\alpha_1}\ldots
(y_{d}-x_{d})^{\alpha_d}$.
Combining (\ref{7}) and (\ref{7_a}), we can easily  obtain
\begin{subequations}\label{7_b}
\begin{align}\label{7_b_1}
u(y)=&u(x) +\nabla u(x)\cdot(y-x) + R^2u(y +s(x-y)),\\
 u(x)=&u(y) -\nabla u(y)\cdot (y-x)  + R^2 u(x +s(y-x)),
\end{align}
\end{subequations}
where $R^2 u(y +s(x-y))$ and $R^2 u(x +s(y-x))$ are the second order remainder terms defined by
\begin{subequations}\label{7_b_2}
\begin{align}
    R^2 u(y +s(x-y))=   \sum_{|\alpha|=2}(y-x)^\alpha\frac{2}{\alpha !}
        \int_{0}^1sD^\alpha  u(y +s(x-y))\,d s, \\
    R^2 u(x +s(y-x))=   \sum_{|\alpha|=2}(x-y)^\alpha\frac{2}{\alpha !}
        \int_{0}^1sD^\alpha  u(x +s(y-x))\,d s.   
\end{align}
\end{subequations}
By (\ref{7_b}) it follows that  
\begin{subequations}\label{7_1c}
\begin{align}
\label{7_1c_a}
	\nabla u(y) \cdot (y-x)  = &\nabla u(x) \cdot (y-x)  +\big(  R^2 u(x +s(y-x))+ R^2u(y +s(x-y)) \big ), \\
\label{7_1c_b}
	 -\big( u(x)-u(y)\big) = &\nabla u(x)\cdot (y-x)  + R^2u(y +s(x-y)).
\end{align}
\end{subequations}
%
%
%
\subsection{The elliptic diffusion problem}
We shall consider the following elliptic 
Dirichlet boundary value problem
\begin{align}\label{0}
-\mathrm{div}(\rho \nabla u) = f \; \text{in} \; \Omega
\quad \mbox{and}	\quad 
u  = u_D   \; \text{on} \; 
\partial \Omega 
\end{align}
as model problem. 
The weak formulation of the boundary value problem (\ref{0}) reads as follows:
for given source function $f \in L^2(\Omega)$ and Dirichlet data 
$u_D \in H^{1/2}(\partial \Omega)$, the trace space of $H^1(\Omega)$,
find a function $u\in H^1(\Omega)$ such that $u=u_D$ on $\partial \Omega$
and the variational identity
\begin{equation}
\label{4a}
a(u,\phi)=l_f(\phi),\;  \forall \phi \in 
H^{1}_{0}(\Omega), 
\end{equation}
is satisfied, where the bilinear form $a(\cdot,\cdot)$ and the linear form $l_f(\cdot)$
are defined by 
\begin{equation}
\label{4b}
a(u,\phi)=\int_{\Omega}\rho\nabla u\nabla\phi\,dx
\quad \mbox{and}\quad
l_f(\phi)=\int_{\Omega}f\phi\,dx,
\end{equation}
respectively. 
{\color{black}
The given diffusion coefficient $\rho \in L^{\infty}(\Omega)$ is assumed to be
uniformly positive and piecewise (patchwise, see below) constant. 
These assumptions ensure existence and uniqueness of the solution
due to Lax-Milgram's lemma.
}
For simplicity, we only consider  pure Dirichlet boundary conditions on $\partial \Omega$.
However, the analysis presented in our paper can easily be generalized to 
other constellations of boundary conditions which ensure existence and uniqueness
such as Robin or mixed boundary conditions.
\par
In what follows, positive constants $c$ and $C$ appearing in  inequalities are 
generic constants which do not depend on the mesh-size $h$. In many cases,  
we will indicate on what may the constants depend for an easier understanding of the proofs. 
Frequently, we will write $a\sim b$ meaning that $c \, a\leq b \leq C \, a$. 
\subsection{Decomposition into patches}
In many practical situations, the computational domain $\Omega$ has a 
multipatch representation, i.e., it is decomposed into $N$ non-overlapping patches 
$\Omega_1, \Omega_2,\ldots, \Omega_N$, also called subdomains:
\begin{align}\label{DecmbOmega}
\overline{\Omega}=\bigcup_{i=1}^N \overline{\Omega}_i, \quad \text{with}{\ } {\Omega}_i\cap{\Omega}_j=\emptyset,\,
\text{for}\, i\neq j.
\end{align}
We use the notation 
$\cal{T}_{H}(\Omega):=\{\Omega_1, \Omega_2,\ldots, \Omega_N\}$ 
for the decomposition  (\ref{DecmbOmega}), 
and we  denote the common interfaces by  
 $F_{ij}=\partial \Omega_i\cap \partial \Omega_j$, for $1 \leq i\neq j \leq N$. 
Essentially, the decomposition (\ref{DecmbOmega}) helps us to consider $N$ local problems posed on each patch, 
where interface conditions are used for coupling these local problems. 
Typically, the interface conditions across each $F_{ij}$ are derived by a theoretical study  of the elliptic 
problem (\ref{0}) and concern continuity requirements of the solution, e.g.,
\begin{align}\label{Interf_Cond}
	\llbracket u \rrbracket :=u_i-u_j=0\text{ on }F_{ij},\quad\text{and}\quad 
        \llbracket \rho \nabla u\rrbracket\cdot n_{F_{ij}} :=(\rho_i\nabla u_i-\rho_j\nabla u_j)\cdot n_{F_{ij}} = 0\text{ on }F_{ij},
\end{align}
 where $n_{F_{ij}}$ is the unit normal  vector on $F_{ij}$ with direction towards $\Omega_j$,
{\color{black}
and $u_i$ denote the restriction of $u$ to $\Omega_i$.
}
 Numerical schemes based on dG usually utilize  the interface conditions (\ref{Interf_Cond}) in order
to devise  numerical fluxes for coupling the local  problems, 
see, e.g., \cite{Maximiliam_Dryga_DG_DD,ERN_DGbook,Rivierebook}.
 Let 
 $\ell\geq 2$ be  an integer,  and let us define the broken Sobolev space
\begin{align}\label{01c_0}
	H^{\ell}(\cal{T}_{H}(\Omega))=\{u\in L^{2}(\Omega): u_i=u|_{\Omega_i}\in H^{\ell}(\Omega_i),\, \text{for}\,i=1,\ldots,N\}.
\end{align}

\begin{assume}\label{Assumption1}
 We assume that the  solution $u$ of (\ref{4a}) belongs to  $V= H^{1}(\Omega)\cap H^{\ell}(\cal{T}_H(\Omega))$ with some $\ell \geq  2$.
\end{assume}
\begin{remark}\label{Uregularity}
	For cases with high discontinuities of $\rho$, the solution $u$ of (\ref{4b})  does not generally have the regularity
	properties of Assumption \ref{Assumption1}. We study dG IgA methods for these 
	problems in \cite{LT:LangerToulopoulos:2014a}.  
\end{remark}

Using  $\cal{T}_{H}(\Omega)$ of (\ref{DecmbOmega}) and  the interface conditions (\ref{Interf_Cond}), the 
variational equation $(\ref{4a})$  can be rewritten as 
\begin{align}\label{Orig_weak_form_Decomp}
	\sum_{i=1}^N \int_{\Omega_i}\rho_i(x)\nabla u\nabla \phi\,dx
       -\sum_{F_{ij}}\int_{F_{ij}} \llbracket \rho \nabla u \phi\rrbracket \cdot n_{F_{ij}}\,d\sigma
       =\sum_{i=1}^N \int_{\Omega_i} f\phi\,dx,\quad \text{for}{\ } \phi\in H^{1}_0(\Omega).
\end{align}
%
\subsection{Non-matching parametrized interfaces}
In order to get a decomposition of the computational domain into subdomains 
{\color{black}
in the IgA context,
}
we first apply a segmentation procedure to the boundary representation of the domain. 
This gives us subdomains, which are topologically equivalent to a cube.
After that we define the control net and  we try to construct  parametrizations of 
the subdomains using superior finite dimensional spaces, e.g.,
 B-splines, NURBS, etc., see \cite{LT:Hughes_IGAbook_2009}.
In the  
{\color{black} ideal}
case, one obtains compatible parametrizations for the common interfaces of adjacent subdomains. The control
points on a face are appropriately matched with the control points of the adjoining face, 
which give an identical interface for the adjacent subdomains. 
\par
However, the resulting subdomain parametrizations  may not produce identical interfaces for adjacent patches.
{\color{black}
We refer to this phenomena as non-matching interface parametrizations
or segmentation crimes. 
}
This can happen in cases where  after the segmentation procedure 
{\color{black}
the control points defining a face are not in an appropriate  correlation  with the 
corresponding control points defining the face of the adjacent patch.
}
The result is a decomposition with the appearance of gap 
or/and overlapping regions between the adjacent patches.  
As a consequence, we cannot directly use the interface conditions (\ref{Interf_Cond}) if we want to derive a  numerical scheme on such domains, cf. \cite{LT:LangerToulopoulos:2014a}.
The interface conditions have to be appropriately modified in order to couple the 
local problems either separated  by gap regions or defined on overlapping regions. 
In \cite{HoferLangerToulopoulos_2015a} and  \cite{HoferToulopoulos_IGA_Gaps_2015a}, 
we have developed and thoroughly studied dG IgA schemes on decompositions including gap regions only. 
In this paper, we  focus mainly on the presentation of dG IgA methods on decompositions 
which include  overlapping regions.  
\subsection{B-spline spaces}
\label{Bsplinespace}
In this section, we briefly present the B-spline spaces and the form of the B-spline parametrizations
for  the physical  subdomains. 
We refer to \cite{LT:Hughes_IGAbook_2009}, 
\cite{CarlDeBoor_Splines_2001} and \cite{LT:Shumaker_Bspline_book} 
for a more detailed presentation. 
\par
Let us consider the unit cube $\widehat{\Omega}=(0,1)^d\subset \mathbb{R}^d$, which we will refer to as the parametric domain, and
let $\Omega_i$, $i=1,\ldots,N$,
be a decomposition of $\Omega$ as given in (\ref{DecmbOmega}). 
Let the integers $p$  and $n_k$  denote  the
 given  B-spline degree  and  the number of basis functions
 of the B-spline space that will be constructed in $x_k$-direction with $k=1,\ldots,d$.  
We introduce the $d-$dimensional  vector of knots 
$\mathbf{\Xi}^d_i=(\Xi_i^1,\ldots,\Xi_i^{k},\ldots,\Xi_i^d),$\, $k = 1,\ldots,d$, 
 with the particular components given by 
 $\Xi_i^{k}=\{0=\xi^{k}_1 \leq \xi^{k}_2 \leq \ldots\leq \xi^{k}_{n_k+p+1}=1\}$. 
 The components $\Xi_i^{k}$  of $\mathbf{\Xi}^d_i$  form 
a mesh  $T^{(i)}_{h_i,\widehat{\Omega}}=\{\hat{E}_m\}_{m=1}^{M_i}$ in $\widehat{\Omega}$,
where $\hat{E}_m$ are the micro elements and $h_i$ is the mesh size, which is
defined as follows. Given a micro element $\hat{E}_m\in T^{(i)}_{h_i,\widehat{\Omega}} $, 
we set $h_{\hat{E}_m}=\text{diam}(\hat{E}_m)=\max_{x_1,x_2 \in \overline{\hat{E}}_m}\|x_1-x_2\|_d $, where $\|.\|_d$
is the Euclidean norm in $\mathbb{R}^d$.
The subdomain mesh size $h_i$ is defined to be  $h_i= \max\{h_{\hat{E}_m}\}$.
We set $h=\max_{i=1,\ldots,N}\{h_i\}$.
We refer the reader to \cite{LT:Hughes_IGAbook_2009} for more information 
about the meaning of the knot vectors in CAD and IgA.

\begin{assume}\label{Assumption2}
	The  meshes $T^{(i)}_{h_i,\widehat{\Omega}}$ are quasi-uniform, i.e.,
	there exist a constant $\theta \geq 1$ such that 
	$\theta^{-1} \leq {h_{\hat{E}_m}}/{h_{\hat{E}_{m+1}}} \leq \theta$.
	Also, we assume that $h_i \sim h_j$ for $1 \leq i\neq j \leq N$.
\end{assume}
\vskip 0.0005cm
 Given the knot vector $\Xi_i^{k}$ in every direction $k=1,\ldots,d$, 
 we construct the associated univariate B-spline basis, 
 $\hat{\mathbb{B}}_{\Xi_i^{k},p}=\{\hat{B}_{1,k}^{(i)}(\hat{x}_k),\ldots,\hat{B}_{n_{k},k}^{(i)}(\hat{x}_k)\}$
 using the Cox-de Boor recursion formula, 
 see, e.g., 
\cite{LT:Hughes_IGAbook_2009} and \cite{CarlDeBoor_Splines_2001} 
for more details.
On  the mesh $T^{(i)}_{h_i,\widehat{\Omega}},$  
we define the multivariate
B-spline space 
$\hat{\mathbb{B}}_{\mathbf{\Xi}^d_i,k}$
to be the tensor-product of the corresponding univariate $\hat{\mathbb{B}}_{\Xi_i^{k},p}$ spaces. 
Accordingly,  the  B-spline basis of $\hat{\mathbb{B}}_{\mathbf{\Xi}^d_i,k}$ are defined by the tensor-product of the univariate B-spline basis functions, that is 
\begin{align}\label{0.00b1}
\hat{\mathbb{B}}_{\mathbf{\Xi}^d_i,p}=\otimes_{k=1}^{d}\hat{\mathbb{B}}_{\Xi_i^{k},p}
=\text{span}\{\hat{B}_{j}^{(i)}(\hat{x})\}_{{j}=1}^{n=n_1\cdot \ldots\cdot n_k\cdot \ldots\cdot n_d},
\end{align}
where each  $\hat{B}_{j}^{(i)}(\hat{x})$ has the form
\begin{align}\label{0.00b2}
\hat{B}_{j}^{(i)}(\hat{x})=&\hat{B}_{j_1}^{(i)}(\hat{x}_1)\cdot\ldots
           \cdot \hat{B}_{j_k}^{(i)}(\hat{x}_k)\cdot\ldots
           \cdot\hat{B}_{j_d}^{(i)}(\hat{x}_d), \,
           \text{with}\,\hat{B}_{j_k}^{(i)}(\hat{x}_k) \in \hat{\mathbb{B}}_{\Xi_i^{k},k}.
\end{align}
 In IgA framework, each $\Omega_i$ is considered as an image of a B-spline, NURBS, etc.,  parametrization mapping.
 Given the B-spline spaces and having defined the control points
 $\mathbf{C}_{j}^{(i)}$, we  parametrize each  subdomain 
 $\Omega_i$ by the  mapping
\begin{align}\label{0.0c}
 \mathbf{\Phi}_i: \widehat{\Omega} \rightarrow \Omega_i, \quad
 x=\mathbf{\Phi}_i(\hat{x}) = \sum_{{j=1}}^n \mathbf{C}^{(i)}_{j} \hat{B}_{j}^{(i)}(\hat{x})\in \Omega_i,
 \end{align}
\label{0.0c2}
 where $\hat{x} =\mathbf{\Phi}^{-1}_i(x)$, $i=1,\ldots,N$,
 cf. \cite{LT:Hughes_IGAbook_2009}. 
For $i=1,\ldots,N$, we construct the B-spline space $\mathbb{B}_{\mathbf{\Xi}^d_i,k}$ on $\Omega_i$ by
\begin{align}\label{0.0d2}
	\mathbb{B}_{\mathbf{\Xi}^d_i,p}:=\{B_{{j}}^{(i)}|_{\Omega_i}: B_{j}^{(i)}({x})=
  \hat{B}_{j}^{(i)}\circ\mathbf{\Phi}^{-1}_i({x}),{\ }\text{for}{\ }
  \hat{B}_{j}^{(i)}\in \hat{\mathbb{B}}_{\mathbf{\Xi}^d_i,p} \}. 
\end{align}
The global  B-spline space $V_{h}$ with components on every $\mathbb{B}_{\mathbf{\Xi}^d_i,p}$
is defined by
\begin{align}\label{0.0d1}
V_h:=V_{h_1}\cdot \ldots\cdot V_{h_N}:=\mathbb{B}_{\mathbf{\Xi}^d_1,p}\cdot \ldots\cdot  \mathbb{B}_{\mathbf{\Xi}^d_N,p}.
\end{align}
 
\begin{remark}\label{remark_00}
	As we point out in the previous subsection, the mappings in (\ref{0.0c}) 
		should 	provide    matching interface parametrizations. 
	Throughout the paper 
		we study the crime case where the mappings in (\ref{0.0c}) produce
	non-matching interface parametrizations.
\end{remark}
\begin{assume}\label{smooth_Phi_i}
	Assume that every $\mathbf{\Phi}_i,\, i=1,...,N$ is sufficiently smooth  and 
	there exist constants $0 < c < C$  such that $c \leq |detJ_{\mathbf{\Phi}_i} | \leq C$, where $J_{\mathbf{\Phi}_i}$ is the
        Jacobian matrix of $\mathbf{\Phi}_i$. 
\end{assume}
\vskip -0.5cm
{\color{black}
\begin{assume}\label{Assumption2_1}
For simplicity, we assume   that  $p\geq \ell$, cf.  Assumption \ref{Assumption1}. 
\end{assume}
}
\begin{figure}
 \begin{subfigmatrix}{4}
  \subfigure[]{\includegraphics[scale=0.35]{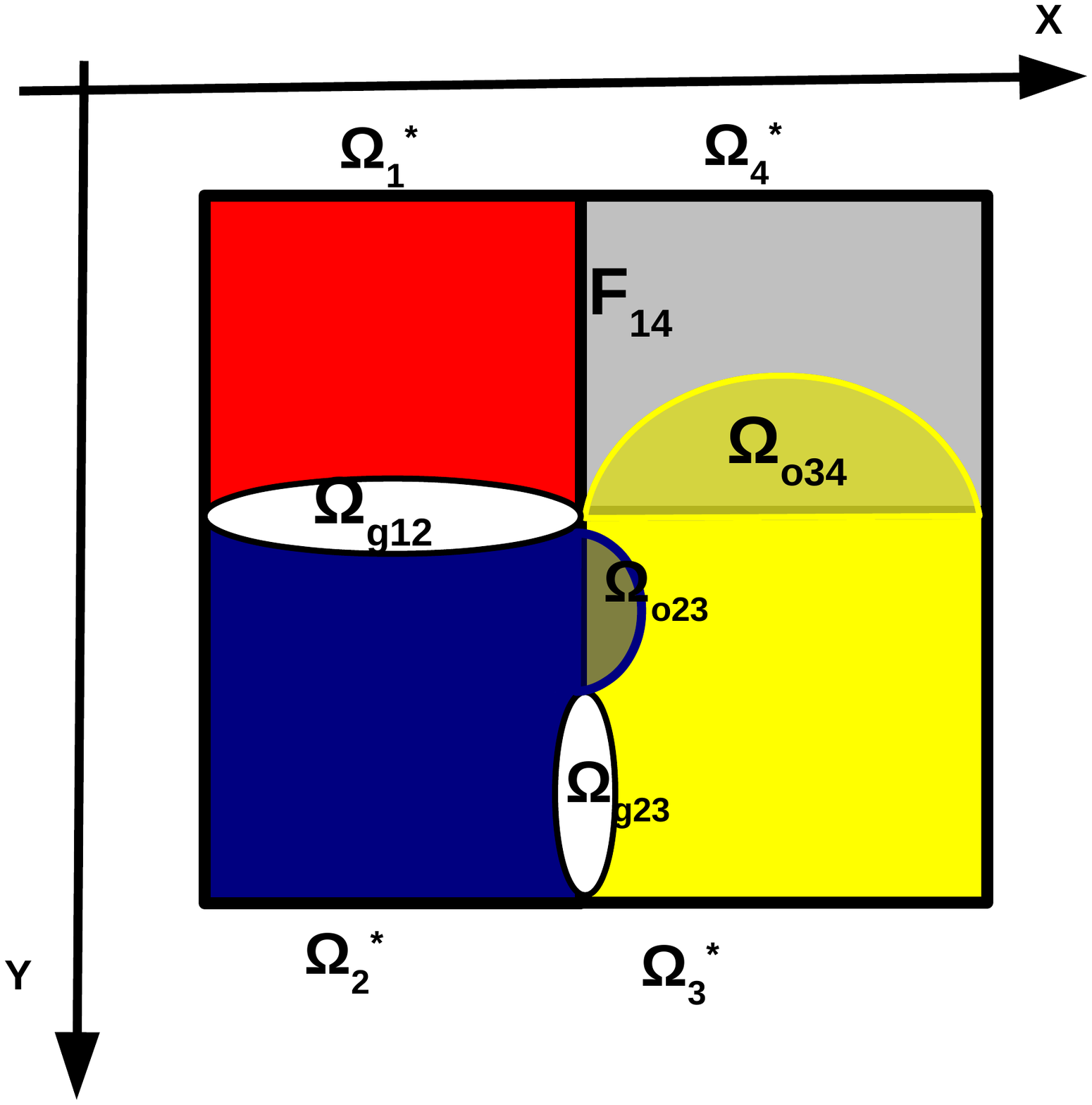}}
  \subfigure[]{\includegraphics[scale=0.35]{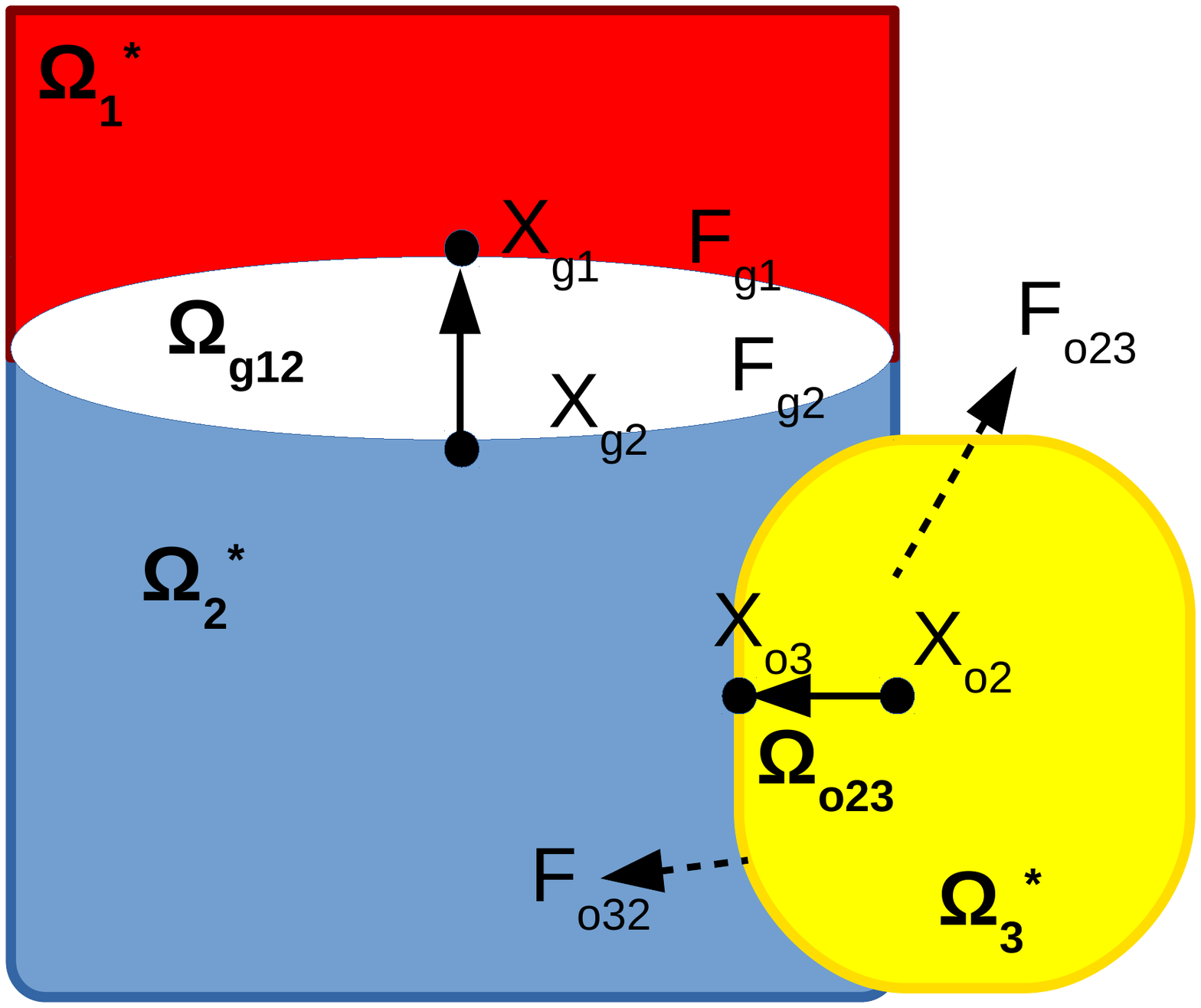}}
  \subfigure[]{\includegraphics[scale=0.35]{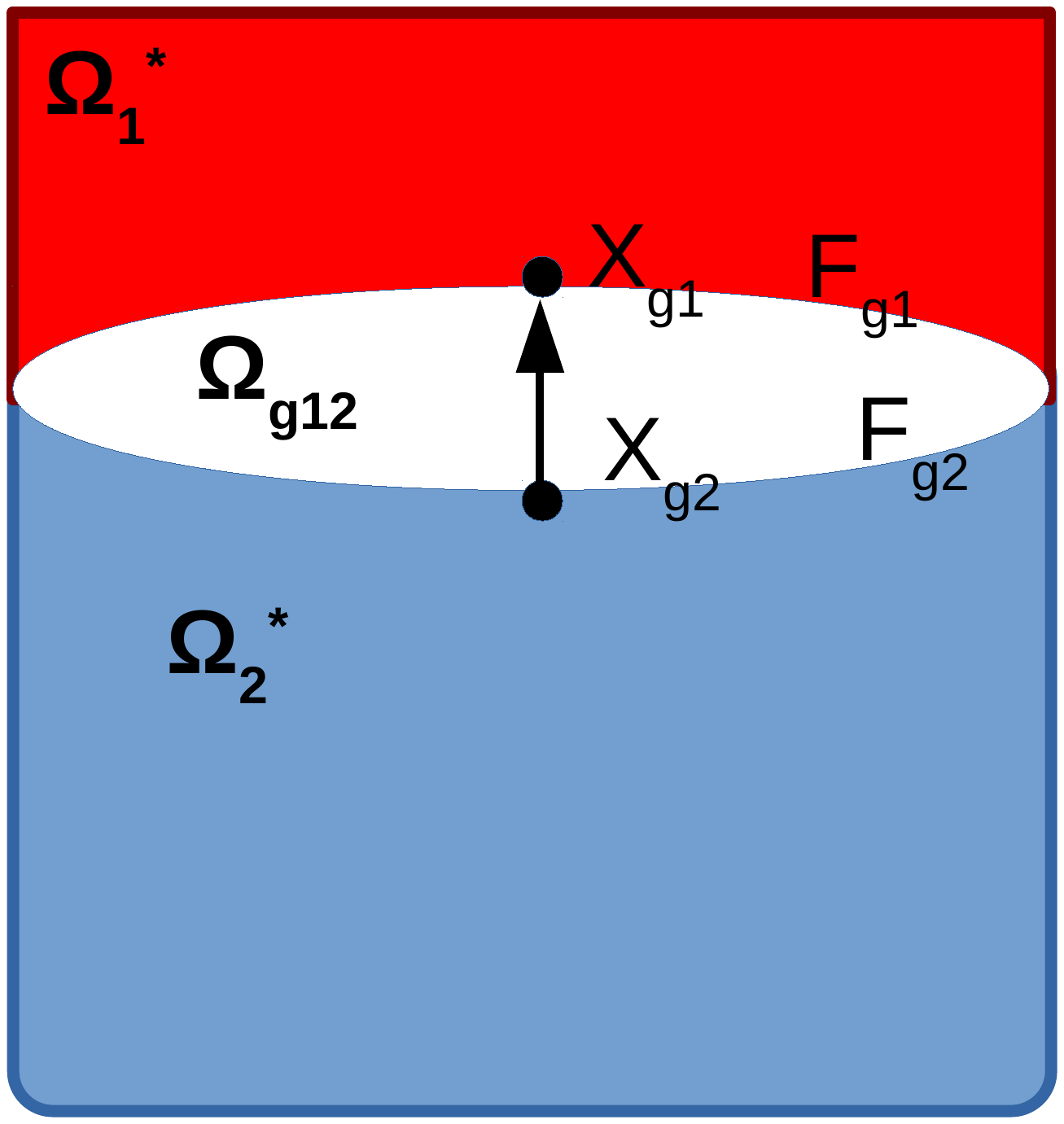}}
  \subfigure[]{\includegraphics[scale=0.35]{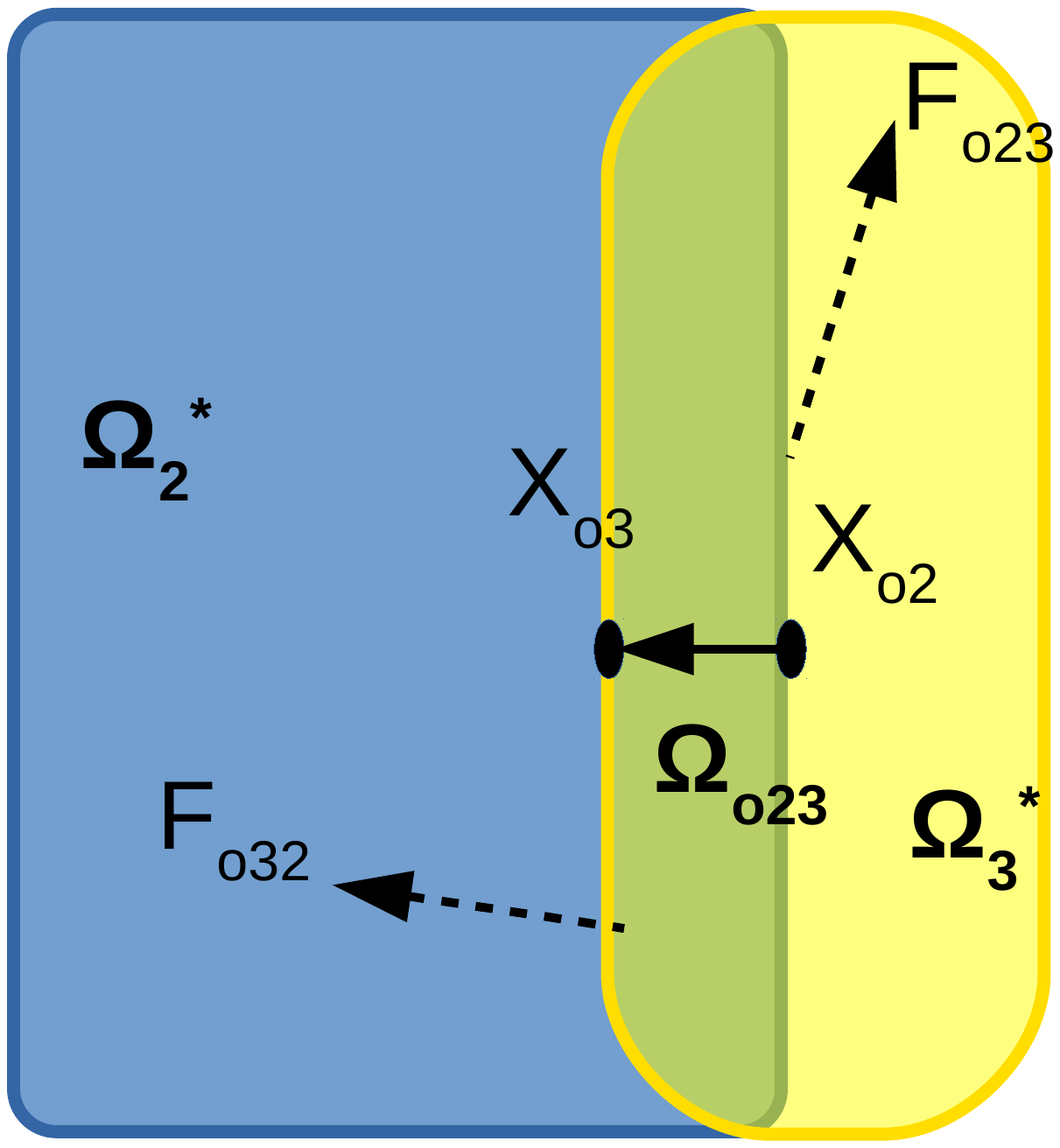}}
\end{subfigmatrix}
 \caption{ (a)  Illustration of a decomposition  including  gap and overlapping  regions,
           (b) the locations  of the diametrically opposite points on gap and overlapping boundaries,
           (c) a decomposition formed by two patches and a gap region,
           (d) a decomposition formed by two patches and a overlapping region, 
           }
 \label{Sub_doms_gap}
\end{figure}
\subsection{Gap regions}
\label{Gap_region}
In this section,  we describe a gap region which is located between two patches. 
Let $\Omega_1^*$ and $\Omega_2^*$ be two adjacent patches with the corresponding 
non-matching interface parametrizations
$\mathbf{\Phi}^*_1: \widehat{\Omega} \rightarrow \Omega^*_1$ and 
$\mathbf{\Phi}^*_2: \widehat{\Omega} \rightarrow \Omega^*_2$. Further, let $\Omega_{g12}$ to be the gab region such that
{\color{black}
$\overline{\Omega}_{g12} \subset \overline{\Omega}$ and $|\partial \Omega_{g12}\cap \partial \Omega| =0$, 
where $|\cdot|$ is the $(d-1)$-dimensional measure. }
We note that 
{\color{black}
$\overline{\Omega}^*_1 \cup \overline{\Omega}^*_2 \cup \overline{\Omega}_{g12} = \overline{\Omega}_1 \cup \overline{\Omega}_2$.
}
An illustration is given in  Figs. \ref{Sub_doms_gap}(b),(c). 
Without 
loss of generality,  we  consider the boundary of the gap region as
 $\partial \Omega_{g12}=F_{g1}\cup F_{g2}$, with  $F_{gi} \subset \Omega^*_i,\,i=1,2$, and we also assume that
 the face $F_{g2}$ is a {simple face}, meaning it can be described  as the set of points $(x,y,z)$ satisfying
\begin{equation}\label{5a0}
0\leq x  \leq x_{Mg},{\ } 0  \leq y \leq \psi_{g2}(x),{\ }  0  \leq z \leq \phi_{g2}(x,y),
\end{equation}
where $x_{Mg}$ is a fixed real number,  $\psi_{g2}$ and $\phi_{g2}$   are  given  smooth  functions, 
see Fig.~\ref{Sub_doms_gap}(c).
As a next step, we need to assign the points $x_{g2}\in F_{g2}$ to the points $x_{g1}\in F_{g1}$.
 We follow the same ideas as in  \cite{HoferLangerToulopoulos_2015a} and \cite{HoferToulopoulos_IGA_Gaps_2015a}. 
 Since  $F_{g2}$ is an 
 {\color{black} simple face}  and $F_{g1}$ is a 
  {\color{black} B-spline surface}, 
  due to the fact that it is the 
 the image  of some face  $\partial \widehat{\Omega}$ under the mapping $\mathbf{\Phi}^*_1$,
 we construct a parametrization of $F_{g1}$,  lets say
 $\mathbf{\Phi}_{g21}:F_{g2}\rightarrow F_{g1}$, 
  which is one-to-one and   defined as
 \begin{align}\label{parmatric_lr}
 x_{g2}\in F_{g2} \rightarrow \mathbf{\Phi}_{g21}(x_{g2}):=x_{g1}\in F_{g1},\quad \text{with }\quad
          \mathbf{\Phi}_{g21}(x_{g2})=x_{g2}+ \zeta_g(x_{g2})n_{F_{g2}},
 \end{align}
where $n_{F_{g2}}$ is the unit normal vector on $ F_{g2}$, see Fig. \ref{Sub_doms_gap}(c), and 
 $\zeta_g$ is a B-spline function. 
The parametrization $\mathbf{\Phi}_{g21}$ defined in (\ref{parmatric_lr}) helps us
to assign the diametrically opposite points located on $\partial \Omega_{g12}$,  
 see discussion in  \cite{HoferLangerToulopoulos_2015a} and \cite{HoferToulopoulos_IGA_Gaps_2015a}. 
We are only interested in small gap regions, see below (\ref{5b_1}), and, thereby, 
if $n_{F_{g1}}$ is the unit normal vector on $F_{g1}$,  we can suppose that $n_{F_{g2}} \approx -n_{F_{g1}}$.  Consequently,     
  we can define the mapping $\mathbf{\Phi}_{g12}:F_{g1}\rightarrow F_{g2}$  to be  
 \begin{align}\label{parmatric_rl}
        \mathbf{\Phi}_{g12}(x_{g1})= x_{g2},{\ }\text{with} \quad
        \mathbf{\Phi}_{g21}(x_{g2})=x_{g1}.
 \end{align}
 We note that both parametrizations $\mathbf{\Phi}_{g21}$  in (\ref{parmatric_lr}) and
  $\mathbf{\Phi}_{g12}$ in (\ref{parmatric_rl}) have been constructed under the consideration
  that one side is planar and $n_{F_{g2}} \approx -n_{F_{g1}}$. 
 As it has been explained in \cite{HoferLangerToulopoulos_2015a} and 
 \cite{HoferToulopoulos_IGA_Gaps_2015a},
 these parametrizations  simplify  the calculations and highlight 
 the main ideas of our approach.  Furthermore, they lead to an easy materialization of the whole method.  In  Section \ref{Section_numerics}, we present  numerical tests where all 
 the faces of the gap and overlapping regions   are curved surfaces.
 We finally  need to  quantify the size of the gap.
 Hence, we introduce  the gap width as
\begin{align}\label{5b_0}
d_g=&\max_{x_{g2}\in F_{g2}} |{x_{g2}}- \mathbf{\Phi}_{g21}(x_{g2})|.
\end{align}
We focus on   gap regions whose width decreases polynomially in $h$, that is,
\begin{align}\label{5b_1}
d_g \leq &  h^\lambda,\quad \text{with some}\quad  \lambda \geq  1.
\end{align}

\subsection{Overlapping  regions}
\label{overlapping_region}
We now describe the case of having a overlapping patch decomposition of $\Omega$. 
Again for simplicity in our analysis,  
we assume  that the overlapping region  is formed by involving  two patches and
it has the same form as the gap regions. 
Based on (\ref{DecmbOmega}), we suppose that there are two so called
\emph{physical} patches $\Omega_2$ and $\Omega_3$ that form a decomposition
of  $\Omega$, i.e., 
\begin{alignat}{2}\label{2.7_1}
 \overline{\Omega} = \overline{\Omega_2} \cup \overline{\Omega_3}, &\quad \overline{\Omega_2} \cap \overline{\Omega_3} = \emptyset, &
 {\ }\text{with} {\ } F_{23}=\partial \Omega_2 \cap \partial \Omega_3,
\end{alignat}
where $F_{23}$ is the  \emph{physical} interface between $\Omega_2$ and $\Omega_3$.
As we mentioned above, due to an incorrect segmentation procedure, we get 
two mappings, lets say $\mathbf{\Phi}_2^*:\widehat{\Omega}\rightarrow \Omega_2^*$ and 
$\mathbf{\Phi}_3^*:\widehat{\Omega}\rightarrow \Omega_3^*$, which cannot exactly parametrize the two physical patches 
${\Omega_2}$ and ${\Omega_3}$.
Let $\Omega_2^*$ and $\Omega_3^*$ 
be the two  patches  of the corresponding ``in-correct''
parametrizations, which form an overlapping decomposition of $\Omega$. Let 
 $\Omega_{o23}=\Omega_2^* \cap \Omega_3^*$ denote the overlapping region, see Fig.~\ref{Sub_doms_gap}(d), 
 and let $F_{oij}=\partial \Omega^*_{i}\cap \Omega^*_{j},$ with    $i,j=2,3$ and $i\neq j$, 
 denote the interior boundary faces of the patches which are related to the overlapping region. 
Furthermore,  let
${n}_{F_{oij}}$ denote the unit exterior normal vector to $F_{oij}$.
We assume that 
$\partial \Omega_{o23}=F_{o23}\cup F_{o32}$. 
{\color{black}
Without loss of generality, we
}
suppose that the following conditions hold: 
(i) the face $F_{o23}$ coincides with the physical interface, 
    i.e., $F_{o23}=F_{23}$, 
(ii) the face $F_{o23}$
    is a {simple face} and meaning that it can be described  as the set of points $(x,y,z)$ 
    satisfying the inequalities
\begin{equation}\label{5a0_o}
0\leq x  \leq x_{M_o},{\ } 0  \leq y \leq \psi_{o2}(x),{\ }  0 \leq z \leq \phi_{o2}(x,y),
\end{equation}
where $x_{Mo}$ is a fixed real number,  $\psi_{o2}$ and $\phi_{o2}$  
are  given  smooth  functions, see Fig.~\ref{Sub_doms_gap}(d). 
We  assign the  points located on $ F_{o23}$   to
the opposite points located on  $F_{o32}$, 
 by constructing   a parametrization for the face $F_{o32}$,  i.e.,
a  mapping 
$\mathbf{\Phi}_{o23}:F_{o23}\rightarrow F_{o32}$. 
In this way, we consider the face  $F_{o32}$ as the image 
{\color{black}
of $F_{o23}$ via the mapping  $\mathbf{\Phi}_{o23}$, 
i.e., $F_{o32}=\mathbf{\Phi}_{o23}(F_{o23})$.
}
Thus, each point
$x_{o3}\in F_{o32}$ becomes the image by means of $\mathbf{\Phi}_{o23}$ of a point $x_{o2}\in F_{o23}$, 
see Fig.~\ref{Sub_doms_gap}(e). 
{\color{black}
Using  the fact that $ F_{o32}$ is  a B-spline surface (it is  the image of a
face of $\partial \widehat{\Omega}$ under the mapping  $\mathbf{\Phi}_3^*$) and since 
we are interested  in  overlapping regions with  sufficient small widths, see below (\ref{5b_1_o}),  we define the mapping 
 $\mathbf{\Phi}_{o23}:F_{o23}\rightarrow F_{o32}$,  as
}
 \begin{align}\label{parmatric_lro}
 x_{o2}\in F_{o23} \rightarrow \mathbf{\Phi}_{o23}(x_{o2}):=x_{o3}\in F_{o32},\quad \text{with }\quad
          \mathbf{\Phi}_{o23}(x_{o2})=x_{o2}+ \zeta_o(x_{o2}){n}_{F_{o23}},
 \end{align}
where  
 $\zeta_o$ is a B-spline and ${n}_{F_{o23}}$ is the unit normal vector on ${F_{o23}}$.
In our analysis, we suppose that ${n}_{F_{o23}} \approx -{n}_{F_{o32}}$,  
and we define the mapping $\mathbf{\Phi}_{o32}:F_{o32}\rightarrow F_{o23}$  
as
 \begin{align}\label{parmatric_rlo}
        \mathbf{\Phi}_{o32}(x_{o3})= x_{o2},{\ }\text{with} \quad
        \mathbf{\Phi}_{o23}(x_{o2})=x_{o3}.
 \end{align}
Finally, we   need to 
 quantify the width of the overlapping region $\Omega_{o23}$, defined by
\begin{align}\label{5b_0_o}
d_o=&\max_{x_{o2}\in F_{o23}} |{x_{o2}}- \mathbf{\Phi}_{o23}(x_{o2})|.
\end{align}
We focus on  overlapping regions whose distance decreases polynomially in $h$, 
i.e.,
\begin{align}\label{5b_1_o}
d_o \leq &  h^\lambda,\quad \text{with some}\quad  \lambda \geq  1.
\end{align}
\subsection{Integrals on  the interior faces}
Let $F_{1}$ be the {simple face} of a gap or overlapping region and let 
 $F_2$ to be the corresponding opposite boundary faces which admits the  
 parametrization  $\mathbf{\Phi}_{12}:F_1\rightarrow F_2$ written as
$\mathbf{\Phi}_{12}(x_1,x_2)=(x_2(x_1,y_1),y_2(x_1,y_1),z_2(x_1,y_1))$, which is reduced to
$x_2=x_1, \,y_2=y_1$ and $z_2=\Phi_{12,z}(x_1,x_2)$, where $\Phi_{12,z}$ is identified by $\zeta_g$ in (\ref{parmatric_lr})
 or $\zeta_o$ in (\ref{parmatric_lro}). 
For a given smooth function $g$ on $F_{2}$, 
we have 
 \begin{alignat}{2} \label{U_o3U_o2_a2}
  \int_{F_{2}} g(x_{2},y_{2},z_{2})\,d \,F_{2} =&
  \int_{F_{1}} g(x_{1},y_{1},{\Phi}_{12,z}(x_{1},y_{1}))\,\|J\|\,d\,x_{1}\,d\,y_{1},
 \end{alignat}
 where $\|J\|=\sqrt{1+|\nabla{\Phi}_{12,z}(x_{1},y_{1})|^2}$.
 For simplicity,  we will  adopt the notation 
 \begin{subequations}\label{U_o3U_o2_a3}
 \begin{alignat}{2} 
  \int_{F_{2}} g\,d \,F_{2} :=& \int_{F_{2}} g(x_{2},y_{2},z_{2})\,d \,F_{2},{\ }\text{and}{\ }\\
 \int_{F_{1}} g(\mathbf{\Phi}_{12})\,\|J\|\,d\,F_{1}:=&
\int_{F_{1}} g(x_{1},y_{1},{\Phi}_{12,z}(x_{1},y_{1}))\,\|J\|\,d\,x_{1}\,d\,y_{1}.
 \end{alignat}
 \end{subequations}
 Under the   regularity of $\mathbf{\Phi}_{12}$, 
there are positive constants $c_m(\|J\|)$ and $C_M(\|J\|$ such that 
 \begin{alignat}{2} \label{U_o3U_o2_a4}
 C_m(\|J\|) \int_{F_{2}} g\,d \,F_{2} \leq 
 \int_{F_{1}} g(\mathbf{\Phi}_{12})\,d\,F_{1} 
 \leq C_M(\|J\|) \int_{F_{2}} g\,d \,F_{2}.
 \end{alignat}
Furthermore, let us denote 
the transpose of the Jacobian matrix of mapping $\mathbf{\Phi}_{12}$
 evaluated at $\mathbf{x}_{1}\in F_{1}$ by $D^\intercal\mathbf{\Phi}_{12}(\mathbf{x}_{1})$, and let $u$ be a smooth function defined on $\Omega$. Then an  application 
 of the chain rule yields 
 \begin{alignat}{2} \label{U_o3U_o2_a5}
 	(\nabla u)\circ \mathbf{\Phi}_{12} = (D^\intercal\mathbf{\Phi}_{12})^{-1}\nabla (u\circ \mathbf{\Phi}_{12}). 
 \end{alignat} 
 Now, by means of the assumptions imposed on the faces $F_{1}$ and $F_{2}$,  we have 
 $n_{F_{1}}=(0,0,1)$. 
 After some elementary calculations, 
 we can find  that
 \begin{alignat}{2} \label{U_o3U_o2_a6}
 \int_{F_{1}} \nabla u(\mathbf{\Phi}_{12}(\mathbf{x}_{1}))\cdot n_{F_{1}}\,d\,F_{1}=
 &\int_{F_{1}} \partial_z u \circ \mathbf{\Phi}_{12}(\mathbf{x}_{1})    \frac{\|J\|}{\|J\|}\,d\,F_{1} \sim 
 \int_{F_{2}} \partial_z u  \,d\,F_{2}.
 \end{alignat}
\section{\color{black} Discontinuous Galerkin IgA schemes for non-matching interfaces}
Let $\mathcal{T}^*_H(\Omega):=\cup_{i=1}^N\overline{\Omega_i^*}$ be 
an incorrect  decomposition of $\Omega$, which  may contain 
 overlapping and gap regions between the patches. 
One has to be careful in constructing the 
 numerical fluxes for the final dG IgA scheme. 
It is clear from the descriptions above that we cannot  directly apply  the numerical fluxes of the
dG IgA methods, which have been presented in 
\cite{LT:LangerToulopoulos:2014a} and  have been applied   on    matching parametrized interfaces.
  However, 
in \cite{HoferLangerToulopoulos_2015a} and  \cite{HoferToulopoulos_IGA_Gaps_2015a}, we recently presented 
dG IgA methods that can be applied on decompositions with gap regions. 
Since the (exact)  boundary data on the boundary of the gaps   are not known, 
we derived appropriate Taylor approximations for them. The Taylor approximations have been 
constructed using the  values of the solution, 
which are computed on the diametrically opposite points  on the gap boundary. 
These approximations are used to 
build up the  numerical fluxes on the boundary of the gaps.   
We here follow the same methodology
 in case of having overlapping regions too. 
%
\subsection{ Approximations of normal fluxes on the gap boundary}
Although dG IgA methods for decompositions including gap regions have been presented by the authors
in \cite{HoferLangerToulopoulos_2015a} and \cite{HoferToulopoulos_IGA_Gaps_2015a},
we repeat here the main parts for the completeness of the current paper. 
We present our analysis for the case of two patches described in Subsection \ref{overlapping_region}, 
see also Fig.~\ref{Sub_doms_gap}(c). 
For simplicity of the analysis below, we assume that the face $F_{g2}$ coincides with the physical interface $F_{12}$.
This implies that $\Omega_2^*=\Omega_2$. Let  $u_1,\,u_2$ and $u_{g12}$ denote the restrictions of the solution $u$
 to each domain $\Omega_1^*,\,\Omega_2^*$ and $\Omega_{g12}$, respectively. The same notation 
 is used to indicate the restrictions for the diffusion coefficient. By 
 the fact that $\Omega_2^*=\Omega_2$, we conclude that  $\rho_{g12}=\rho_1$.
By Assumption \ref{Assumption1}, we have that $u_i\in H^{\ell}(\Omega_i^*),$ where $\ell\geq 2$, $i=1,2$, and  the following interface conditions for the solution $u$ on the 
two faces of $\partial \Omega_{g12}$, see (\ref{Interf_Cond}), 
 \begin{align}\label{Interf_Cond_VFg}
	\llbracket u \rrbracket \big|_{F_{gi}}=0
	\quad\text{and}\quad 
        \llbracket \rho \nabla u\rrbracket\cdot n_{F_{gi}} \big|_{F_{gi}}=0\,
        {\ }\text{with}{\ }i=1,2.	
 \end{align}
Let $x_{g2}\in  F_{g2}$ and let $x_{g1}\in  F_{g1}$ 
be its corresponding diametrically opposite points such that $x_{g1}=\mathbf{\Phi}_{g21}(x_{g2})$ and  
$x_{g2}=\mathbf{\Phi}_{g12}(x_{g1}).$
Denoting $r_{g12} = x_{g1}-x_{g2}$ and having $ r_{g21}=-r_{g12}$  by means of 
(\ref{parmatric_lr}) and (\ref{parmatric_rl}),   we obtain that the normals are given by 
$ n_{F_{g2}}={r_{g12}}/{|r_{g12}|}$ and $ n_{F_{g1}}={r_{g21}}/{|r_{g21}|}$. For convenience,
we denote by  $\{\rho\}:=({\rho_{g12}+\rho_i})/{2}$ the average of the diffusion coefficient across 
$F_{gi},i=1,2$. Using the interface conditions     (\ref{Interf_Cond_VFg}),  and (\ref{7_1c}), we obtain the formulas
\begin{subequations}\label{7_1c_g}
\begin{alignat}{2}
\label{7_1c_g1}
	\nabla u_{g12}(x_{g2}) \cdot n_{F_{g2}} = \nabla u_{1}(x_{g1}) \cdot n_{F_{g2}} & -\frac{1}{|r_{g12}|}\big(  R^2 u_{1}(x_{g2} +s(x_{g1}-x_{g2}))+ R^2u_{g12}(x_{g1} +s(x_{g2}-x_{g1})) \big), \quad\\
\label{7_1c_g2}
	-\frac{1}{h} \big( u_{2}(x_{g2})-u_{1}(x_{g1})\big) = & \frac{|r_{g12}|}{h }\nabla u_{g12}(x_{g2})\cdot n_{F_{g2}}  + 
	\frac{1}{h} R^2u_{g12}(x_{g1} +s(x_{g2}-x_{g1})).
\end{alignat}
\end{subequations}
\subsection{The modified form on  decompositions with gaps}
\vskip -0.15cm
It has been shown in \cite{HoferLangerToulopoulos_2015a} that the solution $u$ of
(\ref{4a}) under the Assumption \ref{Assumption1} and (\ref{Interf_Cond_VFg}) satisfies 
 \begin{multline}\label{5a4}
  \int_{\Omega_2^*}\rho_2\nabla u\cdot \nabla \phi_h\,dx - 
   \int_{\partial \Omega_2^* \cap \partial \Omega}\rho_2\nabla u\cdot n_{\partial \Omega_2^*} \,\phi_h\,d\sigma 
 -  \int_{F_{g2}} \{\rho\nabla u\}\cdot n_{F_{g2}} \, \phi_h  \,d\sigma\\
  +\int_{\Omega_1^*}\rho_1\nabla u\cdot \nabla \phi_h\,dx - 
   \int_{\partial \Omega_1^*\cap\partial \Omega}\rho_1\nabla u\cdot n_{\partial \Omega_1^*} \,\phi_h\,d\sigma 
 - \int_{F_{g1}} \{\rho\nabla u\}\cdot n_{F_{g1}} \, \phi_h \,d\sigma  \\
= \int_{\Omega\setminus\overline{\Omega}_{g12}}f\phi_h\,dx,\quad  
\text{for all}{\ }  \phi_h\in V_h,
 \end{multline}
where the average of the traces across the interface is denoted by $\{\cdot\}$.
The variational form (\ref{5a4}) plays an essential role in
the derivation of the dG IgA scheme. 
 The normal flux terms  $\nabla u_{g12} \cdot n_{\partial \Omega_{g12}}$ appearing in (\ref{5a4}), e.g.,
 $\int_{F_{g2}}\{\rho\nabla u\}\cdot n_{F_{g2}}  \phi_h \,d\sigma=
 \int_{F_{g2}}\frac{1}{2}(\rho_2\nabla u_2+\rho_{g12}\nabla u_{g12})\cdot n_{F_{g2}}  \phi_h \,d\sigma$, 
 are  unknown.
They are going to be  approximated by means of the 
{
normal flux terms 
 $\rho_2\nabla u_{2} \cdot n_{F_{g2}}$ and $\rho_1\nabla u_{1} \cdot n_{F_{g1}}$.
}The adaption of these approximations in  (\ref{5a4}) helps us to couple the two different local problems
 on  $\Omega_2^*$ and $\Omega_1^*$, and finally, to construct the numerical fluxes on $F_{g1}$ and $F_{g2}$.
Using (\ref{Interf_Cond_VFg}) and (\ref{7_1c_g}), we can show that, 
see  \cite{HoferLangerToulopoulos_2015a} and \cite{HoferToulopoulos_IGA_Gaps_2015a},
\begin{multline}\label{7_d}
  \int_{F_{g2}}\big(\frac{\rho_2}{2}\nabla u_{2}(x_{g2}) +\frac{\rho_{g12}}{2}\nabla u_{g12}(x_{g2})\big)\cdot n_{F_{g2}}\phi_h  \\
  = \int_{F_{g2}}\Big(\frac{\rho_2}{2}\nabla u_{2}(x_{g2})+\frac{\rho_{g12}}{2}\nabla u_{1}(x_{g1})\Big)\cdot n_{F_{g2}}\phi_h  -
    \frac{\{\rho \}}{h}\big(u_{2}(x_{g2})-u_{1}(x_{g1})\big)\phi_h \,d\sigma \\
   -\int_{F_{g2}}\Big( \frac{\rho_{g12}}{2|r_{g12}|}R^2u_{g12}(x_{g1}+s(x_{g2}-x_{g1}))+ \frac{\rho_{g12}}{2|r_{g12}|} R^2u_{g12}(x_{g2}+s(x_{g1}-x_{g2})) \Big)\phi_h\,d\sigma \\
   +\int_{F_{g2}}\frac{ \{\rho\}}{h}\Big(|r_{g12}|\nabla u_{g12}\cdot n_{F_{g2}} + R^2u_{g12}(x_{g1}+s(x_{g2}-x_{g1}))\Big)\phi_h\,d\sigma,
\end{multline}
and similarly
\begin{multline}\label{7_d0}
 \int_{F_{g1}}\Big(\frac{\rho_1}{2}\nabla u_{1}(x_{g1}) +\frac{\rho_{g12}}{2}\nabla u_{g12}(x_{g1})\Big)\cdot n_{F_{g1}}\phi_h
 -\frac{ \{\rho\}} {h}\llbracket u(x_{g1})\rrbracket \phi_h\,d\sigma  \\
= \int_{F_{g1}}\Big(\frac{\rho_1}{2}\nabla u_{1}(x_{g1})+\frac{\rho_2}{2}\nabla u_{2}(x_{g2})\Big)\cdot n_{F_{g1}}\phi_h  -\frac{\{\rho\}}{h}\big(u_{1}(x_{g1})-u_{2}(x_{g2})\big)\phi_h \,d\sigma  \\
   -\int_{F_{g1}}\Big( \frac{\rho_{g12}}{2|r_{g21}|}R^2u_{g12}(x_{g2}+s(x_{g1}-x_{g2}))+ \frac{\rho_{g12}}{2|r_{g21}|} R^2u_{2}(x_{g1}+s(x_{g2}-x_{g1})) \Big)\phi_h\,d\sigma \\
   +\int_{F_{g1}}\frac{ \{\rho\}}{h}\Big(|r_{g21}|\nabla u_{g12}\cdot n_{F_{g1}} + R^2u_{g12}(x_{g2}+s(x_{g1}-x_{g2}))\Big)\phi_h\,d\sigma.
 \end {multline}
 Inserting (\ref{7_d}) and (\ref{7_d0}) into (\ref{5a4}), and after few calculations,  we obtain the following modified form  
%
 \begin{multline}\label{7_d1}
 \sum_{i=1}^2\int_{\Omega_i^*}\rho_i\nabla u\cdot \nabla \phi_h\,dx - 
   \int_{\partial \Omega_i^* \cap \partial \Omega}\rho_i\nabla u\cdot n_{\partial \Omega_i^*} \phi_h\,d\sigma 
 -\int_{F_{g2}}\Big(\frac{\rho_2}{2}\nabla u_{2}+\frac{\rho_{g12}}{2}\nabla u_{1}\Big)\cdot n_{F_{g2}}\phi_h  -
    \frac{\{\rho \}}{h}\big(u_{2}-u_{1}\big)\phi_h \,d\sigma \\
+   \int_{F_{g2}}\Big( \frac{\rho_{g12}}{2|r_{g12}|}R^2u_{g12}(x_{g1}+s(x_{g2}-x_{g1}))+ \frac{\rho_{g12}}{2|r_{g12}|} R^2u_{g12}(x_{g2}+s(x_{g1}-x_{g2})) \Big)\phi_h\,d\sigma  \\
-    \int_{F_{g2}}\frac{ \{\rho\}}{h}\Big(|r_{g12}|\nabla u_{g12}\cdot n_{F_{g2}} + R^2u_{g12}(x_{g1}+s(x_{g2}-x_{g1}))\Big)\phi_h\,d\sigma\\
-  \int_{F_{g1}}\Big(\frac{\rho_1}{2}\nabla u_{1}+\frac{\rho_2}{2}\nabla u_{2}\Big)\cdot n_{F_{g1}}\phi_h  -
    \frac{\{\rho\}}{h}\big(u_{1}-u_{2}\big)\phi_h \,d\sigma  \\
+   \int_{F_{g1}}\Big( \frac{\rho_{g12}}{2|r_{g21}|}R^2u_{g12}(x_{g2}+s(x_{g1}-x_{g2}))+ \frac{\rho_{g12}}{2|r_{g21}|} R^2u_{2}(x_{g1}+s(x_{g2}-x_{g1})) \Big)\phi_h\,d\sigma  \\
-   \int_{F_{g1}}\frac{ \{\rho\}}{h}\Big(|r_{g21}|\nabla u_{g12}\cdot n_{F_{g1}} + R^2u_{g12}(x_{g2}+s(x_{g1}-x_{g2}))\Big)\phi_h\,d\sigma= 
\int_{\Omega\setminus\overline{\Omega}_{g12}}f\phi_h\,dx, {\ }\text{for all}{\ }  \phi_h\in V_h.
 \end {multline}
%
\vskip -2cm
\subsection{The local problems on overlapping regions}
For the sake of simplicity, we present the analysis for the case of having only two patches. Let $\{\Omega_i\}_{i=2}^{N=3}$ be
a physical non-overlapping decomposition of $\Omega$ with  the interface $F_{23}$, see (\ref{DecmbOmega}). 
Then  (\ref{Orig_weak_form_Decomp}) is reduced to
\begin{align}\label{Orig_Local_Prb_A}
	\sum_{i=2,3} \int_{\Omega_i}\rho_i(x)\nabla u_i\nabla \phi\,dx
       -\int_{F_{23}} \llbracket \rho \nabla u \phi\rrbracket \cdot n_{F_{23}}\,d\sigma
       =\sum_{i=2,3} \int_{\Omega_i} f\phi\,dx,\quad 
        \text{for all}{\ } \phi\in H^1_0(\Omega), 
\end{align}
where the sub-index $i$ denotes the restriction of functions to the patch $\Omega_i$ 
for $i=2,3$. 
 Let $\Omega_i^*$, $i=2,3$, be two patches that form an overlapping decomposition of $\Omega$, with the overlapping region
 $\Omega_{o23}=\Omega_{2}^*\cap \Omega_{3}^*$, 
 see Fig.~\ref{Sub_doms_gap}(d). 
 Let $F_{oij}=\partial \Omega_i^* \cap \Omega_j^*,\,i,j=2,3$, $i\neq j$, be the interior faces that define $\partial \Omega_{o23}$. 
Without loss of generality, we suppose that $F_{o23}=F_{23}$, see (\ref{Orig_Local_Prb_A}), which in turn implies that
$\Omega_2^*=\Omega_2$ and $\Omega_3^*=\Omega_3\cup \Omega_{o23}$, see Figs. \ref{Sub_doms_gap}(d). 
We can view $\{\Omega_i^*\}_{i=2,3}$ as an extension of the physical decomposition $\{\Omega_i\}_{i=2,3}$ by $\Omega_{o23}$.
To proceed on   decomposition $\{\Omega_i^*\}_{i=2,3}$, 
we  consider  the following perturbed  local problems on $\Omega_2^*$ and $\Omega_3^*$. 
The sub-index indicates the restrictions of the quantities to the patches, 
and $u_{o,2}^*$ will denote the restriction of $u_2^*$ to $\Omega_{o23}$,
and $u_{o,3}^*$ the corresponding restriction of $u_3^*$ to  $\Omega_{o23}$.
 Given the source function $f$ and Dirichlet data $u_D $ we consider  the local problems:
find  $u_2^*\in H^{1}(\Omega_2^*)$ and $u_3^*\in H^{1}(\Omega_3^*)$ such that 
 \begin{subequations}\label{VF_o1_LocalProblems}
 \label{VF_o1_LocalProblems_1}
\begin{alignat}{2}
\begin{cases}
   a^{(2)}_{o,\rho_2}(u_2^*,\phi_2) =l^{(2)}_f(\phi_2), \quad\text{for all}\; \phi_2 \in H^{1}_0(\Omega_2^*)\\
 u_2^*=u_3^*,\,\text{on}\,F_{o23},\quad 
 u_2^*=u_D,\,\text{on}\,\partial \Omega_2^*\cap \partial \Omega,\\
 \text{where }{\ } a^{(2)}_{o,\rho_2}(u_2^*,\phi_2)=
 \int_{\Omega_2^*}\rho_2\nabla u_2^*\cdot \nabla \phi_2\,dx - 
   \int_{\partial \Omega_2^* \cap \partial \Omega}\rho_2\nabla u_2^*\cdot n_{\partial \Omega_2^*} \phi_2\,d\sigma 
 -  \int_{F_{o23}} \rho_2\nabla u_2^*\cdot n_{F_{o23}}  \phi_2  \,d\sigma,\\
  \text{and}\quad l^{(2)}_f(\phi_2)=\int_{\Omega_2^*}f\phi_2\,dx, 
 \end{cases}
 \end{alignat}
 \begin{alignat}{2}
 \label{VF_o1_LocalProblems_2}
 \begin{cases}
 a^{(3)}_{o,\rho_3}(u_3^*,\phi_3)=l^{(3)}_f(\phi_3), \quad\text{for all}\; \phi_3 \in H^{1}_0(\Omega_3^*)\\
     u_3^*=u_2^*,\,\text{on}\,F_{o32},\quad
 u_3^*=u_D,\,\text{on}\,\partial \Omega_3^*\cap \partial\Omega,\\
 \text{where}\quad a^{(3)}_{o,\rho_3}(u_3^*,\phi_3)=
 \int_{\Omega_3^*}\rho_3\nabla u_3^*\cdot \nabla \phi_3\,dx - 
   \int_{\partial \Omega_3^*\cap\partial \Omega}\rho_3\nabla u_3^*\cdot n_{\partial \Omega_3} \phi_3\,d\sigma 
 - \int_{F_{o32}} \rho_3\nabla u_3^*\cdot n_{F_{o32}} \phi_3 \,d\sigma,\\
  \text{and}\quad  l^{(3)}_f(\phi_3)=  \int_{\Omega_3^*}f^*\phi_3\,dx.
 \end{cases}
\end{alignat}
\end{subequations}
 \vskip -0.05cm
We note that (\ref{VF_o1_LocalProblems_2}) is not  equivalent to (\ref{Orig_Local_Prb_A}).  Furthermore, in (\ref{VF_o1_LocalProblems}), we set  $\rho_{o23}=\rho_3$.
 In general, the artificial problems defined in (\ref{VF_o1_LocalProblems})
 are not consistent with the original problem (\ref{Orig_Local_Prb_A}).
 Since the B-spline spaces are defined on 
 the overlapping domains $\Omega^*_2$ and $\Omega_3^*$, the dG IgA scheme will be determined based on
 variational problems (\ref{VF_o1_LocalProblems}). 
In correspondence with  Assumption \ref{Assumption1}, we make the following assumption. 
\begin{assume}\label{Wl2_u_o3}
Let $\ell \geq 2$ be an integer.
For the solutions of (\ref{VF_o1_LocalProblems}), we assume  that  $u_3^*\in H^{\ell}(\Omega_3^*)$ and
$u_2^* \in H^{\ell}(\Omega_2)$. 
\end{assume}
{
The interface conditions given  in (\ref{VF_o1_LocalProblems}) and  (\ref{Interf_Cond}), 
Assumption \ref{Wl2_u_o3} and Assumption \ref{Assumption1} yield
%
\begin{align}\label{Interf_Cond_VFo}
       \begin{cases} 
    u_2 =  u_2^*,& \quad  \text{on}\quad  \Omega_2,\\
    u_3 =  u_3^*,& \quad  \text{on}\quad  \Omega_3,\\
  (\rho_3 \nabla u^*_{o,3} -\rho_3 \nabla u_{3})\cdot n_{F_{o23}}=
  (\rho_2 \nabla u_{o,2}^* -\rho_3 \nabla u_3^*) \cdot n_{F_{o23}} =0, & \quad  \text{on}\quad F_{o23}.
  \end{cases}
\end{align}
%
}
Now, we show that, in the limit case of $d_o\rightarrow 0$, we can recover the 
physical continuity conditions across the interfaces, see (\ref{Interf_Cond}). 
%
If $x_{o2}$ and $x_{o3}$ are two opposite points on $\partial \Omega_{o23}$, see Fig. \ref{Sub_doms_gap}(d), then by   (\ref{5b_0_o}) we have 
that   $ |x_{o2}-x_{o3}|\leq d_o$, and, consequently, 
we have 
\begin{alignat}{2}\label{Cont_u_3}
	|u_3^*(x_{o2})-u_2^*(x_{o3})|=|u_3^*(x_{o2})-u_3^*(x_{o3})|  {\xrightarrow{d_o\rightarrow 0} } 0,
\end{alignat}
due to the continuity of $u_3^*$. 
 \begin{proposition}\label{Prop1_1}
    	Let $\Omega_{o23}$ be the overlapping region  with $\partial \Omega_{o23}=F_{o23}\cup F_{o32}$,   and
    	let the parametrizations $\mathbf{\Phi}_{o23}$ and $\mathbf{\Phi}_{o32}$ be defined as in
    	(\ref{parmatric_lro}) and in (\ref{parmatric_rlo}). Then there exists a positive constant
    	 $C_1$   	 such that 
    	\begin{align}\label{7_d0_a1}
   \Big|\int_{F_{o32}}\rho_3\nabla u^*_{o,3}\cdot n_{F_{o23}}\,d F_{o32} -\int_{F_{o32}}\rho_3\nabla u^*_{o,3}\cdot 
   n_{F_{o32}}\,d F_{o32} \Big| \leq &
  C_{1} d_o \|\rho_3\nabla u^*_{o,3} \|_{L^2(F_{o32})},
   \end{align}
   where the normal vectors  $n_{F_{o23}}$ and $n_{F_{o32}}$ are defined towards the same direction. 
    \end{proposition}
  \begin{proof} The proof is given in \cite{HoferLangerToulopoulos_2015a} for the case of gap regions. The same arguments 
  can be applied for the case under consideration.
  $\BLACKBOX$
  \end{proof}%
\begin{lemma}
	Let Assumptions \ref{Assumption1} and \ref{Wl2_u_o3} hold. Then there is a positive constant\\
	$C(\|u^*_{o,3}\|_{W^{2,2}(\Omega_{o23})},\|\nabla u^*_{o,3} \|_{L^2(F_{o23})})$ independent of $h$, such that
	\begin{alignat}{2} \label{U_o3U_o2_2}
        \| (\rho_3\nabla u^*_{o,3}-\rho_2\nabla u_{o,2}^*)\cdot n_{F_{o32}} \|_{L^2(F_{o32})} & \leq C d_o.
	\end{alignat}
\end{lemma}
\begin{proof}
Let $x_{o3}=\mathbf{\Phi}_{o23}(x_{o2})$. 
	Using the interface conditions (\ref{Interf_Cond}) and (\ref{Interf_Cond_VFo})
	as well as  (\ref{U_o3U_o2_a2}), we get
\begin{multline*}
\| (\rho_3\nabla u^*_{o,3}-\rho_2\nabla u_{o,2}^*)\cdot n_{F_{o32}} \|^2_{L^2(F_{o32} )}
\leq  2 \int_{F_{o23}} \big((\rho_3\nabla u^*_{o,3}(\mathbf{\Phi}_{o23})-\rho_2\nabla u_{o,2}^*(\mathbf{\Phi}_{o23})) \cdot n_{F_{o32}}\big)^2 \|J\|
\,d \,F_{o23} \\
\leq  4 \int_{F_{o23}} \Big(\big(\rho_3\nabla u^*_{o,3}(\mathbf{\Phi}_{o23})\cdot n_{F_{o32}}-\rho_3\nabla u_{3}(x_{o2}) \cdot n_{F_{o23}}\big)^2 + 
      \big(\rho_2\nabla u_{o,2}^*(\mathbf{\Phi}_{o23})\cdot n_{F_{o32}}-\rho_2\nabla u_{2}(x_{o2}) \cdot n_{F_{o23}}\big)^2\Big)\|J\| \,d \,F_{o23} \\
\leq  4 \int_{F_{o23}} \Big(\big(\rho_3\nabla u^*_{o,3}(\mathbf{\Phi}_{o23})\cdot n_{F_{o32}}-\rho_3\nabla u^*_{o,3}(x_{o2}) \cdot n_{F_{o23}}\big)^2 + 
      \big(\rho_2\nabla u_{2}(\mathbf{\Phi}_{o23})\cdot n_{F_{o32}}-\rho_2\nabla u_{2}(x_{o2}) \cdot n_{F_{o23}}\big)^2\Big)\|J\| \,d \,F_{o23} \\
\leq   8
\int_{F_{o23}} \Big(\rho_3^2\big( (\nabla u^*_{o,3}(\mathbf{\Phi}_{o23})-\nabla u^*_{o,3}(x_{o2})) \cdot n_{F_{o32}}
                    +\nabla u^*_{o,3}(x_{o2}) \cdot ( n_{F_{o32}} - n_{F_{o23}})\big)^2 + \\
  \qquad     \rho_2^2\big( (\nabla u_{2}(\mathbf{\Phi}_{o23})-\nabla u_{2}(x_{o2})) \cdot n_{F_{o32}}
                    +\nabla u_{2}(x_{o2}) \cdot ( n_{F_{o32}} - n_{F_{o23}})\big)^2 \Big)\|J\| \,d \,F_{o23} \\
\leq  8 \max(\rho_3^2,\rho^2_2)(T_1+T_2+T_3+T_4).\qquad
\end{multline*}

We now proceed to bound each of the $T_i,\,i=1,\ldots,4$ terms.
For the term $T_1$, we have
\begin{alignat}{2} \label{U_o3U_o2_3}
	T_1 \leq & \int_{F_{o32}}\Bigg|\int_{x_{o2}}^{\mathbf{\Phi}_{o23}} \Big( \frac{\partial}{\partial t}\nabla u^*_{o,3}(t)\Big)\,dt\Bigg|^2\,d \,F_{o23} \leq 
	d_o\,\int_{F_{o32}}\int_{x_{o2}}^{\mathbf{\Phi}_{o23}} \Big( \frac{\partial }{\partial t}\nabla u^*_{o,3}(t)\Big)^2\,dt\,d \,F_{o23} \\
	\nonumber
	                \leq & d_o\|u^*_{o,3}\|_{W^{2,2}(\Omega_{o23})}.
\end{alignat}
{For the term $T_2$, we use (\ref{7_d0_a1}) and we immediately have} 
$
T_2 \leq  C d_o^2\, \|\nabla u^*_{o,3} \|^2_{L^2(F_{o23})}. 
$
The terms $T_3$ and $T_4$ can be analogously bounded. Gathering  the previous bounds, we derive (\ref{U_o3U_o2_2}).
$\BLACKBOX$
\end{proof}  
Thus, by (\ref{Interf_Cond_VFo}), (\ref{Cont_u_3}) and (\ref{U_o3U_o2_2}), we infer  that 
	 the conditions 	(\ref{Interf_Cond}) are recovered  	as $d_o \rightarrow 0$. 
\begin{remark}\label{remark_2}
	Under the assumption 
	{$n_{F_{o32}} \approx - n_{F_{o23}} $}, 
	the terms $T_2$ and $T_4$  can be ignored 
	from the previous estimates.
\end{remark}
%
%
\subsection{Approximations of  normal fluxes and the modified form on overlapping regions}
The normal fluxes in (\ref{VF_o1_LocalProblems}) on the faces $F_{o23}$ and $F_{o32}$ of the overlapping boundary, 
must be appropriately  modified in order to couple the two local problems.
Then, we use these modifications in order to introduce Taylor approximations of the 
normal fluxes and to construct finally the numerical fluxes on $\partial \Omega_{23}$, which in turn  
couple the local patch-wise discrete problems. 
\subsubsection{\color{black} Approximations of normal fluxes by Taylor expansions.}
  Let $x_{o2}\in F_{o23}$ and $x_{o3}\in F_{o32}$ be such that 
 $ x_{o3}=\mathbf{\Phi}_{o23}(x_{o2}) $.
 Denoting $r_{o23}=x_{o2}-x_{o3}$ and using the assumption that $r_{o23}=-r_{o32}$, we obtain that
 $n_{F_{o23}}=\frac{r_{o23}}{|r_{o23}|}=-n_{F_{o32}}$.
%
Using the interface conditions of (\ref{VF_o1_LocalProblems}),  (\ref{Interf_Cond_VFo}) and
(\ref{7_1c}),    
{\color{black}
we obtain the formulas
}
\begin{subequations}\label{VF_o1_1}
\begin{align}
\label{VF_o1_1_a}
	\nabla u^{*}_3(x_{o3}) \cdot n_{F_{o23}} = \nabla u^{*}_{2}(x_{o2}) \cdot n_{F_{o23}} & -\frac{1}{|r_{o23}|}\big(  R^2u^{*}_{2}(x_{o3} +s(x_{o2}-x_{o3}))+ R^2u^{*}_{3}(x_{o2} +s(x_{o3}-x_{o2}) )\big) \\
\label{VF_o1_1_b}
	-\frac{1}{h} \big( u^{*}_{2}(x_{o3})-u^{*}_{3}(x_{o2})\big) = & \frac{|r_{o23}|}{h }\nabla u^{*}_2(x_{o3})\cdot n_{F_{o3}}  + 
	\frac{1}{h} R^2u^{*}_{2}(x_{o2} +s(x_{o3}-x_{o2})).
\end{align}
\end{subequations}
\subsubsection{\color{black} The modified form.}
To treat the overlapping nature of the IgA parametrizations, we use 
the bilinear forms in (\ref{VF_o1_LocalProblems}), 
the interface conditions given in  (\ref{Interf_Cond_VFo}),  
    the fact that $F_{o23}=F_{23}$ and  $\Omega_2^*=\Omega_2$, and obtain
\begin{alignat}{2}\label{VF_o2_a}
\nonumber 
a^{(2)}_{o,\rho_2}(u_2^*,\phi_h) + a^{(3)}_{o,\rho_3}(u_3^*,\phi_h) = \\
\nonumber
 	\int_{\Omega_2^*}\rho_2\nabla u_2^*\cdot \nabla \phi_h\,dx - &
   \int_{\partial \Omega_2^* \cap \partial \Omega}\rho_2\nabla u_2^*\cdot n_{\partial \Omega_2^*} \phi_h\,d\sigma \\
   \nonumber
     & - \int_{F_{o23}} \frac{1}{2}\big(\rho_2\nabla u_2^*(x_{o2})+\rho_2\nabla u_2^*(x_{o2})\big)\cdot n_{F_{o23}} \phi_h  
- \frac{\{\rho\}}{h}(u^*_{2}(x_{o2})- u^*_{2}(x_{o2}))  \phi_h  \,d\sigma \\
\nonumber
+\int_{\Omega_3^*}\rho_3\nabla u_3^*\cdot \nabla \phi_h\,dx - &
   \int_{\partial \Omega_3^*\cap\partial \Omega}\rho_3\nabla u_3^*\cdot n_{\partial \Omega_3^*} \phi_h\,d\sigma \\
   \nonumber
 - &\int_{F_{o32}} \frac{1}{2}\big({\rho_3}\nabla u_3^*(x_{o3})+
                                  {\rho_3}\nabla u_3^*(x_{o3})\big)\cdot n_{F_{o32}}\phi_h  -
                                   \frac{\{\rho\}}{h}(u_3^*(x_{o3})- u_3^*(x_{o3}))  \phi_h  \,d\sigma    \\                    
=& \int_{\Omega_2^*}f\phi_h\,dx + \int_{\Omega_3^*}f\phi_h\,dx,\quad  \text{for all}\quad  \phi_h\in V_h.
 \end{alignat}
  Using in (\ref{VF_o2_a}) the continuity conditions across the faces that are given in (\ref{VF_o1_LocalProblems}), 
  consequently employing  the Taylor expansions (\ref{VF_o1_1})  and adopting the notation
  $x_{o3}:=\mathbf{\Phi}_{o23}(x_{o2})$, $x_{o2}:=\mathbf{\Phi}_{o32}(x_{o3})$,            we get
\begin{multline}\label{VF_o2}
 	\int_{\Omega_2^*}\rho_2\nabla u_2^*\cdot \nabla \phi_h\,dx - 
   \int_{\partial \Omega_2^* \cap \partial \Omega}\rho_2\nabla u_2^*\cdot n_{\partial \Omega_2^*} \phi_h\,d\sigma 
 +\int_{\Omega_3^*}\rho_3\nabla u_3^*\cdot \nabla \phi_h\,dx - 
   \int_{\partial \Omega_3^* \cap \partial \Omega}\rho_3\nabla u_3^*\cdot n_{\partial \Omega_3^*} \phi_h\,d\sigma \\
 -  \int_{F_{o23}} \frac{1}{2}\big(\rho_2\nabla u_2^*(x_{o2})+\rho_3\nabla u_3^*(x_{o3})\big)\cdot n_{F_{o23}}  \phi_h  
 +\frac{\rho_2}{2|r_{o23}|}\big(  R^2u^*_{2}(x_{o3} +s(x_{o2}-x_{o3}))+ R^2u^*_{3}(x_{o2} +s(x_{o3}-x_{o2})) \big ) \phi_h   \\
 -\frac{\{\rho\}}{h}(u_2^*(x_{o2})-u_3^*(x_{o3}))\phi_h+
  \frac{\{\rho\}}{h}\big(|r_{o23}|\nabla u^*_{3}(x_{o2})\cdot n_{F_{o23}} +
  R^2u^*_{3}(x_{o3} +s(x_{o2}-x_{o3})) \big ) \phi_h  \,d\sigma \\
    -  \int_{F_{o32}} \frac{1}{2}\big(\rho_3\nabla u_3^*(x_{o3})+\rho_2\nabla u_2^*(x_{o2})\big)\cdot n_{F_{o32}}  \phi_h  
 +\frac{\rho_3}{2|r_{o32}|}\big(  R^2u^*_{2}(x_{o3} +s(x_{o2}-x_{o3}))+ R^2u^*_{3}(x_{o2} +s(x_{o3}-x_{o2})) \big ) \phi_h 
   \\
   -\frac{\{\rho\}}{h}(u_3^*(x_{o3})-u_2^*(x_{o2}))\phi_h+
  \frac{\{\rho\}}{h}\big(|r_{o32}|\nabla u_{2}^{*}(x_{o3})\cdot n_{F_{o32}} -
  R^2u^*_{2}(x_{o2} +s(x_{o3}-x_{o2})) \big ) \phi_h  \,d\sigma \\
= \int_{\Omega_2^*}f\phi_h\,dx + \int_{\Omega_3^*}f\phi_h\,dx,\quad  \text{for all}\quad  \phi_h\in V_h.
 \end{multline}

\subsection{{The consistency error.}}
As we pointed out above, due to the overlapping of the diffusion coefficient on $\Omega_{o23}$ the solution 
$u_3^*|_{\Omega_{o23}}$ is different from the physical solution $u|_{\Omega_{o23}}$  given by (\ref{Orig_Local_Prb_A}).
In particular, by the interface conditions (\ref{Interf_Cond_VFo}), we have
\begin{align}\label{3_4_1_a}
 \nonumber
	a^{(3)}_{o,\rho_3}(u_3^*,\phi_3)=&
 \int_{\Omega_3^*}\rho_3\nabla u_3^*\cdot \nabla \phi_3\,dx - 
   \int_{\partial \Omega_3^*\cap\partial \Omega}\rho_3\nabla u_3^*\cdot n_{\partial \Omega_3} \phi_3\,d\sigma  
 - \int_{F_{o32}} \rho_3\nabla u_3^*\cdot n_{F_{o32}} \phi_3 \,d\sigma  \\
 \nonumber
 =&\int_{\Omega_{o23}}\rho_3\nabla u_{o,3}^*\cdot \nabla \phi_3\,dx - 
 \int_{F_{o23}} {\rho_3}\nabla u_{o,3}^*\cdot n_{F_{o23}}\phi_3\,d\sigma -\int_{F_{o23}} {\rho_3}\nabla u_{3}^*\cdot (-n_{F_{o23}})\phi_3\,d\sigma \\
 \nonumber
 -& \int_{F_{o32}} {\rho_3}\nabla u_{o,3}^*\cdot n_{F_{o32}}\phi_3\,d\sigma    
 +\int_{\Omega_3}\rho_3\nabla u^*_3\cdot \nabla \phi_3\,dx -
 \int_{\partial \Omega_3^*\cap\partial \Omega}\rho\nabla u_3^*\cdot n_{\partial \Omega_3^*} \phi_3\,d\sigma \\
 =&  \int_{\Omega_3^*}f\phi_3\,dx,\quad\text{for}\quad \phi_3 \in W^{1,2}_0(\Omega_3^*).
\end{align}
The physical relevant form with $\rho=\rho_2$ on $\Omega_{o23}$ is
\begin{align}\label{3_4_1_b}
 \nonumber
	a^{(3)}_{o,\rho_2}(u,\phi_3)=&
 \int_{\Omega_{o23}}\rho_2\nabla u_{o,2}^*\cdot \nabla \phi_3\,dx - 
 \int_{F_{o23}} {\rho_2}\nabla u_{o,2}^*\cdot n_{F_{o23}}\phi_3\,d\sigma -\int_{F_{o23}} {\rho_3}\nabla u_{3}\cdot (-n_{F_{o23}})\phi_3\,d\sigma \\
 \nonumber
 -& \int_{F_{o32}} {\rho_2}\nabla u_{o,2}^*\cdot n_{F_{o32}}\phi_3\,d\sigma    
 +\int_{\Omega_3}\rho_3\nabla u_3\cdot \nabla \phi_3\,dx -
 \int_{\partial \Omega_3^*\cap\partial \Omega}\rho\nabla u_3\cdot n_{\partial \Omega_3^*} \phi_3\,d\sigma \\
 =&  \int_{\Omega_3^*}f\phi_3\,dx,\quad\text{for}\quad \phi_3 \in W^{1,2}_0(\Omega_3^*).
\end{align}
Now, by the construction of the local problems, it holds that $u_{o,2}^*=u_{2}|_{\Omega_{o23}}$, where 
$u_{2}|_{\Omega_{o23}}$ is the restriction of solution $u$ given by (\ref{Orig_Local_Prb_A}) to ${\Omega_{o23}}$, and
as well it holds $u_{3}|_{\Omega_{3}}=u_3^*|_{\Omega_{3}}$. 
Therefore, 
from (\ref{3_4_1_a}) and (\ref{3_4_1_b}), we obtain that the
difference $\rho_2\nabla u_{2}-\rho_3\nabla u^*_{o,3}$ on $\Omega_{o23}$  satisfies 
$
	\int_{\Omega_{o23}}(\rho_2\nabla u_{2}-\rho_3\nabla u^*_{o,3})\cdot \nabla \phi_3\,dx=0,
$
and taking $\phi_3=(\rho_2 u_{2}-\rho_3 u^*_{o,3})$, we obtain
 $\int_{\Omega_{o23}}\big|\rho_2\nabla u_{2}-\rho_3\nabla u^*_{o,3}\big|^2\,dx=0.$
Also, it follows by (\ref{VF_o1_LocalProblems}) that
\begin{alignat}{2}\label{3_4_2b}
\int_{F_{o32}} \frac{\{\rho\}}{h}(u_2(x_{o3})- u_{o,3}^*(x_{o3}))^2  \,d\sigma = 
\int_{F_{o23}} \frac{\{\rho\}}{h}(u_2(x_{o2})- u_{o,3}^*(x_{o2}))^2  \,d\sigma=0. 
 \end{alignat}
and thus 
\begin{alignat}{2}\label{3_4_4a}
\int_{\Omega_{o23}}\big|\rho_2\nabla u_{2}-\rho_3\nabla u^*_{o,3}\big|^2\,dx + 
\frac{\{\rho\}}{h} \Big(
\int_{F_{o32}} (u_2(x_{o3})- u_{o,3}^*(x_{o3}))^2  \,d\sigma + 
\int_{F_{o23}} (u_{o,3}^*(x_{o2})- u_2(x_{o2}))^2  \,d\sigma\Big)=0.
\end{alignat}

--------------------------------------------------------------------------------------------------
\subsection{The dG IgA   form on  general decompositions}
In the previous section, relation (\ref{7_d1}) has been derived for a decomposition consisting  of two patches $\Omega_1^*$ and $\Omega_2^*$, that are separated
by the gap region $\Omega_{g12}$. Working in the same spirit, the  form (\ref{VF_o2}) has been derived for 
a decomposition of $\Omega$ formed by two overlapping patches $\Omega_2^*$ and $\Omega_3^*$. It is 
clear that for a general decomposition $\mathcal{T}^*_{H}(\Omega):=\cup^{N}_{i=1}\Omega^{*}_i$ of $\Omega$ that includes gap and overlapping regions,
similar expressions  can be derived by defining the corresponding   Taylor expansions.
Below, we describe the proposed dG IgA scheme for a general decomposition.
\par
So far, we denoted the restrictions of the  solution of the perturbed problems on  each
$\Omega_i^*$ by $u_i^*$. 
That was useful because the solution of the perturbed problem does not coincide with the solution of (\ref{4a}).
In the following sections, we aim at avoiding lengthy formulas with complicated notations. Hence,
we will denote the solution obtained on $\mathcal{T}^*_{H}(\Omega)$ by  $u$,
 and its restriction on every patch $\Omega^*_i$ by $u_i$, independent of having gaps, overlaps or matching interfaces. 
The  corresponding  diffusion coefficient is denoted by $\rho_i$. \\
Let  $\ell\geq 2$ be  an integer and let $\mathcal{F}^*$ be the set of all the interior faces of 
$\{\partial \Omega_i^*\}_{i=1}^N$.
In association with  $\mathcal{T}^*_{H}(\Omega)$ and $\mathcal{F}^*$, we 
  introduce  the space
\begin{align}\label{3.5.1a}
	V^*:=\{v\in 
	H^\ell(\Omega^*_i),\,
	\text{for}\,i=1,\ldots,N: 
	\llbracket v \rrbracket |_{F}=0\,\text{for all}\,F\in \mathcal{F}^*\},
\end{align}
where $\llbracket v \rrbracket |_{F}$ denotes the jump of $v$ across $F$. 
{\color{black}
By the definition  of the 
patch-wise 
solutions, see e.g. (\ref{Interf_Cond_VFg}) and
(\ref{VF_o1_LocalProblems}), we can assume that $u\in V^*$. 
}
In order to proceed   with our analysis, we first define    
the {dG}-norm $\|.\|_{dG}$ associated with  $\mathcal{T}^*_H(\Omega)$. For all $v\in V_{h}^{*}:=V^*+V_h$,
\begin{align}\label{3.5.2a}
\nonumber
	\|v\|^2_{dG} = & \sum_{i=1}^N\Big(\rho_i\|\nabla v_i\|^2_{L^2(\Omega^*_i)} + 
                      \frac{\rho_i}{h}\|v_i\|^2_{L^2(\partial \Omega^*_i\cap \partial\Omega)}  
                     + \sum_{F_{gj}\subset \partial \Omega_i^*}   \frac{\{\rho\}}{h}\|v_i  \|^2_{L^2(F_{gi})}\\
                     +&  \sum_{F_{oij}\subset \partial \Omega_i^*}   \frac{\{\rho\}}{h}\|v_i  \|^2_{L^2(F_{oij})}+
                       \sum_{F_{ij}\subset \partial \Omega_i^*}   \frac{\{\rho\}}{h}\|v_i -v_j \|^2_{L^2(F_{ij})}\Big),
\end{align}
where $F_{gj}$, $F_{oij}$ and $F_{ij}$  are the interior interfaces related to gap regions, overlapping regions and matching interfaces, respectively,
see Fig.~\ref{Sub_doms_gap}(b). 
For  convenience,  we employ the following  notation for the Taylor residuals, see (\ref{7_1c}) and (\ref{VF_o1_1}), 
\vskip -0.5cm
\begin{subequations}\label{Notations_Residuals}
\begin{align}
	R^2u_{gij}=& \Big( \frac{\rho_{gij}}{2|r_{gij}|}R^2u_{j}(x_{gi}+s(x_{gj}-x_{gi}))+ \frac{\rho_{gij}}{2|r_{gij}|} R^2u_{i}(x_{gj}+s(x_{gi}-x_{gj})) \Big),\\
    R_{\nabla,gij}=&\frac{ \{\rho\}}{h}\Big(|r_{gij}|\nabla u_{j}\cdot n_{F_{gj}} + R^2u_{j}(x_{gi}+s(x_{gj}-x_{gi}))\Big),\\
	R^2u_{oij}=&\frac{\rho_i}{2|r_{oij}|}\big(  R^2 u_{o,i}(x_{oj} +s(x_{oi}-x_{oj})+ R^2u_{o,j}(x_{oi} +s(x_{oj}-x_{oi}) \big ),\\
R_{\nabla,oij}    =& \frac{\{\rho\}}{h}\big(|r_{oij}|\nabla u_{o,j}(x_{oi})\cdot n_{F_{oij}} +
                        R^2u_{o,i}(x_{oj} +s(x_{oi}-x_{oj})) \big ).
\end{align}
\end{subequations}
Furthermore, let $\mathbf{\Phi}_{gij}:F_{gi}\rightarrow F_{gj}$ and $\mathbf{\Phi}_{oij}:F_{oij}\rightarrow F_{oji}$ be the mappings of the faces of gap and overlapping regions, respectively.
Recalling (\ref{7_d1}) and (\ref{VF_o2}) and using the notations (\ref{Notations_Residuals}), 
we deduce the identity
\vskip -0.67cm
\begin{multline}\label{7_d3}
\sum_{i=1}^N\Big( \int_{\Omega_i^*}\rho_i\nabla u\cdot \nabla \phi_h\,dx - 
   \int_{\partial \Omega^*_i \cap \partial \Omega}\rho_i\nabla u\cdot n_{\partial \Omega_i} \phi_h\,d\sigma \Big)\\
 -\sum_{i=1}^N\sum_{F_{gj}\subset \partial\Omega_i^*}  \int_{F_{gj}}\Big(\frac{\rho_i}{2}\nabla u_i+\frac{\rho_j}{2}\nabla u_j(\mathbf{\Phi}_{gij})\Big)\cdot n_{F_{gj}}\phi_h  -
  \frac{\{\rho\}}{h}\big(u_i-u_j(\mathbf{\Phi}_{gij})\big)\phi_h\, d\sigma + \int_{F_{gj}}\Big(R^2u_{g,ij}+ R_{\nabla,g,ij}\Big)\phi_h\, d\sigma\\
  -\sum_{i=1}^N\sum_{F_{oij}\subset \partial\Omega_i^*}  \int_{F_{oij}}\Big(\frac{\rho_i}{2}\nabla u_i+\frac{\rho_j}{2}\nabla u_j(\mathbf{\Phi}_{oij})\Big)\cdot n_{F_{oij}}\phi_h  -
  \frac{\{\rho\}}{h}\big(u_i-u_j(\mathbf{\Phi}_{oij})\big)\phi_h\, d\sigma +\int_{F_{oij}}\Big(R^2u_{o,ij}+ R_{\nabla,o,ij}\Big)\phi_h\, d\sigma\\
  \hspace*{-1cm}
  -\sum_{i=1}^N\sum_{F_{ij}\subset \partial\Omega_i^*}  \int_{F_{ij}}\Big(\frac{\rho_i}{2}\nabla u_i+\frac{\rho_j}{2}\nabla u_j\Big)\cdot n_{F_{ij}}\phi_h  -
  \frac{\{\rho\}}{h}\big(u_i-u_j\big)\phi_h\, d\sigma 
 =  \sum_{i=1}^N\int_{\Omega_i^*}f\phi_h\,dx, {\ }\phi_h \in V_h,
\end{multline}
that holds for the  solution $u$,
where the integrals over the faces in (\ref{7_d3}) are defined as in (\ref{7_d1}) and 
(\ref{VF_o2}).  We observe that the terms appearing in (\ref{7_d3})  are the terms that are expected to 
 appear  in a  dG scheme, of course, excluding the Taylor remainder terms. 
In view of this, we    define the  forms
$ B_{\Omega_i^*}(\cdot,\cdot):V_h^*\times V_h \rightarrow \mathbb{R}$,
$R^2_{g}(\cdot,\cdot):V_h^*\times V_h  \rightarrow \mathbb{R}$, 
$R^2_{o}(\cdot,\cdot):V_h^*\times V_h  \rightarrow \mathbb{R}$, 
$R_{\nabla,g}(\cdot,\cdot):V_h^*\times V_h  \rightarrow \mathbb{R}$,
$R_{\nabla,o}(\cdot,\cdot):V_h^*\times V_h  \rightarrow \mathbb{R}$, 
$R_{\Omega_o,\Omega_g}(\cdot,\cdot):V_h^*\times V_h  \rightarrow \mathbb{R}$,
and the linear functional
$l_{f, \Omega_i^*}:V_h \rightarrow \mathbb{R}$ by
\vskip -0.4cm
 \begin{subequations}\label{7_d3_b}
\begin{align}\label{Var_Form_Gap_Overl}
B_{\Omega_i^*}(u,\phi_h) =& \sum_{i=1}^N\Big(\int_{\Omega^*_i}\rho_i\nabla u_i\cdot \nabla \phi_h\,dx - 
   \int_{\partial \Omega^*_i \cap \partial \Omega}\rho_i\nabla u\cdot n_{\partial \Omega^*_i} \phi_h\,d\sigma \\
   \nonumber
 &  -\sum_{F_{gj}\subset \partial\Omega_i^*}  \int_{F_{gj}}\Big(\frac{\rho_i}{2}\nabla u_i+\frac{\rho_j}{2}\nabla u_j\Big)\cdot n_{F_{gj}}\phi_h  -
  \frac{\eta\{\rho\}}{h}\big(u_i-u_j\big)\phi_h\, d\sigma \\
  \nonumber
  &-\sum_{F_{oij}\subset \partial\Omega_i^*}  \int_{F_{oij}}\Big(\frac{\rho_i}{2}\nabla u_i+\frac{\rho_j}{2}\nabla u_j\Big)\cdot n_{F_{oij}}\phi_h  -
  \frac{\eta\{\rho\}}{h}\big(u_i-u_j\big)\phi_h\, d\sigma \\
   \nonumber
   &-\sum_{F_{ij}\subset \partial\Omega_i^*}  \int_{F_{ij}}\Big(\frac{\rho_i}{2}\nabla u_i+\frac{\rho_j}{2}\nabla u_j\Big)\cdot n_{F_{ij}}\phi_h  -
  \frac{\eta\{\rho\}}{h}\big(u_i-u_j\big)\phi_h\, d\sigma\Big),
  \end{align}
 \vskip -0.2cm
  \begin{alignat}{2}\label{Var_Form_Residuals}
 \nonumber
 R_{g}(u,\phi_h)=        & \sum_{i=1}^N\sum_{F_{gj}\subset \partial\Omega_i^*} R^2u_{g,ij},              & 
  R_{\nabla,g}(u,\phi_h)=&\sum_{i=1}^N\sum_{F_{gj}\subset \partial\Omega_i^*} \eta  R_{\nabla,g,ij},\\
	R_{o}(u,\phi_h)=&\sum_{i=1}^N\sum_{F_{oij}\subset \partial\Omega_i^*} R^2u_{o,ij},               &
R_{\nabla,o}(u,\phi_h) =& \sum_{i=1}^N\sum_{F_{oij}\subset \partial\Omega_i^*} \eta R_{\nabla,o,ij},\\ 
\nonumber
R_{\Omega_o,\Omega_g}(u,\phi_h) = & R_{g}(u,\phi_h) +  R_{\nabla,g}(u,\phi_h)+ & R_{o}(u,\phi_h)+  R_{\nabla,o}(u,\phi_h), &\\ 
\nonumber
 {\ }\hskip 2cm
 l_{f,\Omega_i^*}(\phi_h) = & \sum_{i=1}^N\int_{\Omega_i^*}f\phi_h\,dx,
\end{alignat}
\end{subequations}
where $\eta >0$ is a parameter that 
{
will be defined by the coercivity of the resulting dG bilinear form 
on the IgA spaces $V_h$, see proof of Lemma~\ref{lemma_0}.
}
We note that  the forms $Ru_{o}$ and $R_{\nabla,o}$  in (\ref{Var_Form_Residuals}) are related to overlapping patches
and  were not been introduced
in \cite{HoferLangerToulopoulos_2015a,HoferToulopoulos_IGA_Gaps_2015a} 
for the derivation of the dG IgA scheme in the case of   gap regions. 
For the case under consideration, the introduction  of $ R_{\nabla,g}$ and $R_{\nabla,o}$ simplifies the analysis of the method.  In order to establish a practically usable dG IgA variational problem,
we must get rid of the  terms related to Taylor residuals in the  dG IgA 
bilinear form. Also, we  prefer the weak  enforcement of the  boundary conditions. 
Hence, we introduce the bilinear form 
$B_h(\cdot,\cdot):V_h\times V_h \rightarrow \mathbb{R}$
 and the linear form $F_h: V_h \rightarrow \mathbb{R}$
as follows
\begin{equation}\label{7_d5}
	B_h(u_h,\phi_h) = B_{\Omega_i^*}(u_h,\phi_h) +
	\sum_{i=1}^N\frac{\eta\rho_i}{h}\int_{\partial\Omega_i^* \cap \partial \Omega}u_h \phi_h\,d\sigma,
\end{equation}
\begin{equation}\label{7_d6}
	F_h(\phi_h) = l_{f,\Omega_i^*}(\phi_h)+
	\sum_{i=1}^N\frac{\eta\rho_i}{h}\int_{\partial\Omega_i^*\cap \partial \Omega}u_D \phi_h\,d\sigma.
\end{equation} 
Finally, our dG IgA scheme reads as follows: find $u_h\in V_h $ such that
\begin{equation}\label{7_d4}
	B_h(u_h,\phi_h) = F_h(\phi_h), \quad \text{for all} {\ }\phi_h\in V_h.
\end{equation}
{
We immediately conclude that the variational identity
}
\begin{align}\label{7_d3_a}
 B(u,\phi_h):= B_{h}(u,\phi_h) + R_{\Omega_o,\Omega_g}(u,\phi_h)=
  F_h(\phi_h),{\ }\forall \phi_h \in V_{h},
\end{align}
holds for the solution $ u \in V^* $.
Next we show several results that we are going to  use in our error analysis. 
For convenience, we introduce the following notations:
  \begin{subequations}\label{Notations_gap_Overl}
  \begin{align}
  \mathcal{K}_{\partial \Omega_{g},\partial \Omega_{o}}=& \sum_{\Omega_{gij}}
    \|\nabla u_{gij}\|_{L^2(\partial \Omega_{gij})} +
    \sum_{\Omega_{oij}}
    \|\nabla u_{oij}\|_{L^2(\partial \Omega_{oij})},\\
   \mathcal{K}_{\Omega_{g},\Omega_{o}}=& \sum_{\Omega_{gij}}
    \|\sum_{|\alpha|=2}|D^\alpha u_{gij}|\|_{L^2(\Omega_{gij})} +
    \sum_{\Omega_{oij}}
    \|\sum_{|\alpha|=2}|D^\alpha u_{oij}|\|_{L^2(\Omega_{oij})}\\
     \mathcal{K}_{g,o}=&\mathcal{K}_{\partial \Omega_{g},\partial \Omega_{o}} +\mathcal{K}_{\Omega_{g},\Omega_{o}}.
   \end{align}  
   \end{subequations} 
\begin{lemma}\label{lemma4}
Under the Assumption  \ref{Assumption2} and 
the setting (\ref{5b_1}),  there exist  positive constants 
$C_1(\rho, d)$ and $C_2(\rho,d)$ such 
that the estimates 
\begin{subequations}\label{7_i0}
 \begin{flalign}\label{7_i0_1}
  | R_{\nabla,g}(u,\phi_h)+ R_{\nabla,o}(u,\phi_h)|\leq & C_1 
  \|\phi_h\|_{dG} \Big(   \mathcal{K}_{\partial \Omega_{g},\partial \Omega_{o}} h^{\beta} +
     \mathcal{K}_{\Omega_{g},\Omega_{o}}\,h^{\gamma}\Big), \\
   \label{7_i0_2}
   |R_{g}(u,\phi_h)+R_{o}(u,\phi_h)| \leq & C_2 \|\phi_h\|_{dG}\,  \mathcal{K}_{\Omega_{g},\Omega_{o}}
   h^{\gamma},
   \end{flalign}
   \end{subequations} 
   holds 
   for all $u \in V^*$ and $\phi_h \in V_h$.
  where we have previously used the notations (\ref{Notations_Residuals})  and the setting 
  $\beta=\lambda- 0.5$
  and
  $\gamma=\lambda+1$.
%
 \end{lemma}
\begin{proof}
In  \cite{HoferLangerToulopoulos_2015a,HoferToulopoulos_IGA_Gaps_2015a}, we proved the  bounds (\ref{7_i0}) for the case of 
gap regions. The same arguments can be followed for showing the same bounds for the case of overlapping regions. 
Essentially,  bounds (\ref{7_i0})  are  generalizations of these two aforementioned cases, and thus the details of the proofs are omitted. 
$\BLACKBOX$
\end{proof}  
\begin{lemma}\label{lemma00} Let  $\beta=\lambda-\frac{1}{2}$ and $\gamma=\lambda+1$.
Then there is a constant 
{ $C=C(\eta,\rho) \geq 0$} 
independent of $h$ such that the estimate 
\begin{align}\label{7_d7a}
B_h(u,\phi_h)\leq & {\color{black} C(\eta,\rho)} \Big( \big(\|u\|^2_{dG} + \sum_{i=1}^N
  \rho_i h\| \nabla u_i\|^2_{L^2(\partial \Omega^*_{i})}\big)^{\frac{1}{2}} + 
   \mathcal{K}_{g,o} (h^{\gamma}+h^{\beta})\Big)
   \|\phi_h\|_{dG},
\end{align}
 holds  for all 
 $u \in V_h^*$ and $\phi_h \in V_h$,
 where $\mathcal{K}_{g,o}$ is defined as in (\ref{Notations_Residuals}). 
\end{lemma}
\begin{proof}
Applying
inequality (\ref{HolderYoung}), we can show that
\begin{align}\label{7_d7a_00}
\big|\sum_{i=1}^N \int_{\Omega_i^*}\rho^{\frac{1}{2}}\nabla u\cdot \rho^{\frac{1}{2}}\nabla \phi_h\,dx \big| \leq
\|u\|_{dG}\|\phi_h\|_{dG}.
\end{align}
Next, we  give   bounds for the whole flux terms.  Let $F_{ij}$ be the matching interfaces on
$\mathcal{T}_H^*(\Omega)$.  
A direct application of   lemma 5.2 in \cite{LT:LangerToulopoulos:2014a} gives
\begin{multline}\label{7_d7a_0}
\sum_{i=1}^N\sum_{F_{ij}\subset \partial\Omega_i^*}
\Big|  \int_{F_{ij}}\Big(\frac{\rho_i}{2}\nabla u_i+\frac{\rho_j}{2}\nabla u_j\Big)\cdot n_{F_{ij}}\phi_h  -
  \frac{\eta\{\rho\}}{h}\big(u_i-u_j\big)\phi_h\, d\sigma \Big|\\
   +\Big|\int_{\partial \Omega_i^* \cap \partial \Omega}\rho_i\nabla u_i\cdot n_{\partial \Omega_i^*} \phi_h\,d\sigma \Big| 
   \leq  C(\eta,\rho)\Big( \sum_{i=1}^N \rho_i\,  h\| \nabla u_i\|^2_{L^2(\partial \Omega_i^*)}\Big)^{\frac{1}{2}} \|\phi_h\|_{dG}.
\end{multline}
%
For the  flux terms on the gap faces $F_{gi}\subset \partial \Omega_i^*$, we proceed as follows: The mappings
$\mathbf{\Phi}_{gij}:F_{gi}\rightarrow F_{gj}$, the triangle   inequality, conditions 
 (\ref{Interf_Cond_VFg}) and  (\ref{7_1c_a}) yield
 \begin{multline}\label{7_d7b}
 \sum_{i=1}^N\sum_{F_{gi}\subset \partial\Omega_i^*}\Big|\int_{F_{gi}}\Big(\frac{\rho_i}{2}\nabla u_i+
 \frac{\rho_j}{2}\nabla u_j(\mathbf{\Phi}_{gij}) \Big)\cdot n_{F_{gi}}\phi_h\,d\sigma \Big| \\
\leq  \sum_{i=1}^N\sum_{F_{gi}\subset \partial\Omega_i^*}
   \Big|\int_{F_{gi}}  \rho_i \nabla u_i\cdot n_{F_{gi}}
   \phi_h \,d\sigma \Big| +
   \Big|\int_{F_{gi}}  \rho_j \nabla u_j(\mathbf{\Phi}_{gij})\cdot n_{F_{gi}} \phi_h \,d\sigma \Big| \\
   \leq \sum_{i=1}^N\sum_{F_{gi}\subset \partial\Omega_i^*} \Big|\int_{F_{gi}}  \rho_i \nabla u_i\cdot n_{F_{gi}}
   \phi_h \,d\sigma \Big| +
   \Big|\int_{F_{gi}}  \Big(\rho_{gij} \nabla u_{gij}\cdot n_{F_{gi}} +  R^2 u_{g,ij}\Big)  \phi_h \,d\sigma \Big| \\
   \leq \sum_{i=1}^N\sum_{F_{gi}\subset \partial\Omega_i^*} 2   C(\rho)\Big|\int_{F_{gi}}   \nabla u_i\cdot n_{F_{gi}}
   \phi_h \,d\sigma \Big| +\Big|\int_{F_{gi}}  R^2 u_{g,ij}\phi_h \,d\sigma  \Big| \\
   \leq \sum_{i=1}^N
   2C(\rho)\Big(   \rho_i\, h\| \nabla u_i\|^2_{L^2(\partial \Omega_i)}\Big)^{\frac{1}{2}} \|\phi_h\|_{dG}+
   \|\phi_h\|_{dG}  \mathcal{K}_{\Omega_g,\Omega_o}   h^{\gamma},
 \end{multline}
 where  estimate  (\ref{7_i0_2})   has been used.  
 Also, we need to bound the jump terms   in $B_h(\cdot,\cdot)$.
 Following the same arguments  as in (\ref{7_d7b}),  using (\ref{7_1c_b}) and estimate (\ref{7_i0_1}),  we can show
 \begin{align}\label{7_d7e}
 \nonumber
\sum_{i=1}^N\sum_{F_{gi}\subset \partial\Omega_i^*}\Big|\int_{F_{gi}}  \frac{\{\rho\}}{h}\big(u_i-u_j(\mathbf{\Phi}_{gij})\big)\phi_h \,d\sigma \Big| 
      \leq C \sum_{i=1}^N\sum_{F_{gi}\subset \partial\Omega_i^*} \Big| 
    \int_{F_{gi}}  R_{\nabla,g,ij}\phi_h \,d\sigma \Big| \\
    \leq C(\eta,\rho) \Big(  \mathcal{K}_{\partial \Omega_{gij},\partial \Omega_{oij}} h^{\beta} +
     \mathcal{K}_{\Omega_{g},\Omega_{o}}\,h^{\gamma}\Big) \|\phi_h\|_{dG}.
 \end{align}
%
  In a similar, way we can   verify
 \begin{multline}\label{7_d7f}
\sum_{i=1}^N\sum_{F_{oij}\subset \partial\Omega_i^*}\Big |  \int_{F_{oij}}\Big(\frac{\rho_i}{2}\nabla u_i+\frac{\rho_j}{2}\nabla u_j\Big)\cdot n_{F_{oij}}\phi_h  -
  \frac{\eta\{\rho\}}{h}\big(u_i-u_j\big)\phi_h\, d\sigma \Big| \\
  \leq C(\eta,\rho) \Big(\sum_{i=1}^N
   \big(\rho_i\,h\| \nabla u_i\|^2_{L^2(\partial \Omega^*_{i})}\big)^{\frac{1}{2}} + 
   \mathcal{K}_{g,o} (h^{\gamma}+h^{\beta})\Big)
   \|\phi_h\|_{dG},
 \end{multline}
 Finally, collecting all the above   estimates,
 we can deduce assertion (\ref{7_d7a}). 
$\BLACKBOX$  
\end{proof}

We point out that the terms $\mathcal{K}_{g,o} (h^{\gamma}+h^{\beta})$ in (\ref{7_d7a}) appear due to the 
estimation of the multi-directional Taylor  remainder terms, 
which are involved in the approximation of the   normal fluxes on 
$\partial \Omega_i^*$. 
In \cite{HoferLangerToulopoulos_2015a}, a similar bound has been shown 
for the case of 
uni-directional
Taylor expansions, working in a different direction. 
In fact, estimate (\ref{7_d7a}) generalizes  this estimate for the case of general gap and overlapping regions.

\begin{proposition}\label{Propos_1}
{\color{black}	
        Let $\Xi$ be a knot vector that forms a partition of $I=[0,1]$,
	and let
	$\mathbb{B}_{{\Xi}_i,p}$ be the corresponding B-spline space of fixed degree $p$, see 
	(\ref{0.0d2}).  
	Then, for a given $\phi_1 \in \mathbb{B}_{{\Xi}_i,p}$ 
	with $\phi_1(x)> 0$  a.e. on $I$,
	there 	exists a positive constant $c$ such that
	 inequality 
	\begin{align}\label{3_5_16_aa}
	\int_{0}^1 \big(1-\frac{1}{\sqrt{1+\phi_1(x)}}\big)(\phi_2(x))^2\,dx 
	\geq
	c \, \int_{0}^1 (\phi_2(x)^2)\,dx,
	\end{align}
	holds  	for all $\phi_2 \in \mathbb{B}_{{\Xi}_i,p}$.
}
\end{proposition}
{\color{black}
\begin{proof} 
If $\phi_1(x) >0$ everywhere in $I$, then (\ref{3_5_16_aa}) can easily be   shown.
Next, we give the proof for the case of
existing only one interior point $x_0$ such that $\phi_1(x_0)=0$. 
The proof can be generalized to other cases. 
	Let $x_0$ be an interior point of $I$  such that $\phi_1(x_0)=0$. 
	Then, for given, sufficiently small $\epsilon >0$, 	
	there exist a $0 < \delta < 1/2$ such that
	 $\phi_1(x)\leq \epsilon$ for all 
	 $x\in (x_0-\delta,x_0+\delta) {\color{black} \subset (0,1)} $ 
	 {\color{black} and $\phi_1(x)\geq \epsilon$ outside this interval.}
	 We now introduce the notation $\Delta^-=x_0-\delta$ and $\Delta^+=x_0+\delta$. 
		By the continuity of $\phi_2^2 := (\phi_2)^2$ in $[\Delta^-,\Delta^+]$, 
	there exist 
	{\color{black} non-negative constants}
	{ $m_{2}$ and $M_2$} such that
	$m_2 \leq \phi_2^2(x)\leq M_2$ for all $x\in [\Delta^-,\Delta^+]$. 
	Finally, we arrive at the estimates
	\begin{multline}\label{3_5_16_ab}
		\int_{0}^1 \big(1-\frac{1}{\sqrt{1+\phi_1(x)}}\big)\phi_2^2(x)\,dx \geq \\
		\int_{0}^{\Delta^-} \big(1-\frac{1}{\sqrt{1+\phi_1(x)}}\big)\phi_2^2(x)\,dx +
		\int_{\Delta^-}^{\Delta^+} \big(1-\frac{1}{\sqrt{1+\phi_1(x)}}\big)\phi_2^2(x)\,dx +
		\int_{\Delta^+}^1 \big(1-\frac{1}{\sqrt{1+\phi_1(x)}}\big)\phi_2^2(x)\,dx \\
		\geq \int_{0}^{\Delta^-} \big(1-\frac{1}{\sqrt{1+\phi_1(x)}}\big)\phi_2^2(x)\,dx +
		     \int_{\Delta^+}^1 \big(1-\frac{1}{\sqrt{1+\phi_1(x)}}\big)\phi_2^2(x)\,dx 
		 -  \int_{\Delta^-}^{\Delta^+} \big(1-\frac{1}{\sqrt{1+\epsilon}}\big)M_2\,dx \\
=
\int_{0}^{\Delta^-} \big(1-\frac{1}{\sqrt{1+\phi_1(x)}}\big)\phi_2^2(x)\,dx 
+\int_{\Delta^+}^1  \big(1-\frac{1}{\sqrt{1+\phi_1(x)}}\big)\phi_2^2(x)\,dx -
2\,\delta\,{b(\epsilon)M_2 },
			\end{multline}
where, for simplicity, we introduced the notation  
$b(\epsilon) = 1- (1+\epsilon)^{-1/2}$. 
Let $x_1\notin [\Delta^-,\Delta^+]$ be an interior point  where  $\phi_2^2(x_1)$ is strictly positive, e.g., $\phi_2^2$ has a local maximum 
at $x_1$. Then there exist $\sigma(\epsilon) \geq\delta$   such  that
$\big(1-(1+\phi_1(x))^{-1/2}\big)\phi_2^2(x)\geq 4 \,b(\epsilon)M_2$ 
for all $x\in (x_1-\sigma,x_1+\sigma) {\color{black} \subset (0,1)}$.
	Let us denote $\Sigma^-=x_1-\sigma$ and $\Sigma^+=x_1+\sigma$. Without loss of generality, we can
	assume that $\Sigma^+ < \Delta^-$. 
Finally, we have
\begin{multline}\label{3_5_16_ac}
\int_{0}^1 \big(1-\frac{1}{\sqrt{1+\phi_1(x)}}\big)\phi_2^2(x)\,dx \geq 
\int_{0}^{\Sigma^-} b(\epsilon)\phi_2^2(x)\,dx + \int_{\Delta^+}^{1} b(\epsilon)\phi_2^2(x)\,dx  \\
+ \int_{\Sigma^-}^{\Sigma+}  \big(1-\frac{1}{\sqrt{1+\phi_1(x)}}\big)\phi_2^2(x)\,dx+ 
\int_{\Sigma^+}^{\Delta^-} b(\epsilon)\phi_2^2(x)\,dx -2\,\delta\,b(\epsilon)\,M_2 \\
\geq 
\int_{0}^{\Sigma^-} b(\epsilon)\phi_2^2(x)\,dx + \int_{\Delta^+}^{1} b(\epsilon)\phi_2^2(x)\,dx 
+ \int_{\Sigma^+}^{\Delta^-} b(\epsilon)\phi_2^2(x)\,dx  \\
+ 2\,\int_{\Sigma^-}^{\Sigma+} b(\epsilon)\,M_2\,dx +
0.25\int_{\Sigma^-}^{\Sigma+} b(\epsilon)\phi_2^2(x)\,dx + 0.25\int_{\Delta^-}^{\Delta+} b(\epsilon)\phi_2^2(x)\,dx
 -2\,\delta\,b(\epsilon)\,M_2 \\
\geq 
\int_{0}^{\Sigma^-} b(\epsilon)\phi_2^2(x)\,dx + \int_{\Delta^+}^{1} b(\epsilon)\phi_2^2(x)\,dx +
\int_{\Sigma^+}^{\Delta^-} b(\epsilon)\phi_2^2(x)\,dx\\
0.25\int_{\Sigma^-}^{\Sigma+} b(\epsilon)\phi_2^2(x)\,dx +
0.25\int_{\Delta^-}^{\Delta+} b(\epsilon)\phi_2^2(x)\,dx,
\end{multline}
where we used that $\sigma \geq\delta$ and $\int_{\Sigma^-}^{\Sigma+}b(\epsilon)M_2\,dx\geq 
\int_{\Delta^-}^{\Delta+}b(\epsilon)\phi_2^2(x)\,dx.$ This proves the assertion. 
$\BLACKBOX$  
\end{proof} 
}
\begin{proposition}\label{Propos_2}
	Let $F_1$ and $F_2$ be the two opposite faces of a gap (or overlapping region) and 
	let $\mathbf{\Phi}_{12}:F_{1}\rightarrow F_{2}$ and $\mathbf{\Phi}_{21}:F_{2}\rightarrow F_{1}$
	be the two parametrization mappings with Jacobian norms $\|J_{\mathbf{\Phi}_{12}}\|$ and $\|J_{\mathbf{\Phi}_{21}}\|$, 
	respectively. 
	There 	exist a constant $C >0$ such that the following inequality holds:
	\begin{flalign}\label{3_5_16_a}
	\nonumber
		 \int_{F_{1}}\big(\phi_{1,h}(x_1)-\phi_{2,h}(\mathbf{\Phi}_{12}(x_1)\big)\phi_{1,h}(x_1)\, dx_1 + 
		 \int_{F_{2}}\big(\phi_{2,h}(x_2)-\phi_{1,h}(\mathbf{\Phi}_{21}(x_2))\big)\phi_{2,h}(x_2)\, dx_2 \\
		 \geq C \big(\|\phi_{1,h}\|^2_{L^2(F_1)} +\|\phi_{2,h}\|^2_{L^2(F_2)}\big).
	\end{flalign}
\end{proposition}
\begin{proof}
	Introducing the notation
	\begin{multline*}
	\phi_h|_{[F_1,F_2]}=\int_{F_{1}}\big(\phi_{1,h}(x_1)-\phi_{2,h}(\mathbf{\Phi}_{12}(x_1)\big)\phi_{1,h}(x_1)\, dx_1 + 
		 \int_{F_{2}}\big(\phi_{2,h}(x_2)-\phi_{1,h}(\mathbf{\Phi}_{21}(x_2))\big)\phi_{2,h}(x_2)\, dx_2,   
\end{multline*}
	 we can derive the following estimates
	\begin{multline}\label{3_5_16_c}
	\phi_h|_{[F_1,F_2]}
=	\int_{F_{1}}\phi^2_{1,h}(x_1) - \phi_{2,h}(\mathbf{\Phi}_{12}(x_1)\big)\phi_{1,h}(x_1)\, dx_1 +
	 \int_{F_{2}}\phi^2_{2,h}(x_2)-\phi_{1,h}(\mathbf{\Phi}_{21}(x_2))\phi_{2,h}(x_2)\, dx_2  \\
\geq 
\int_{F_{1}}\phi^2_{1,h}(x_1) - \frac{1}{2}\phi^2_{2,h}(\mathbf{\Phi}_{12}(x_1)\|J_{\mathbf{\Phi}_{12}}\|
-\frac{1}{2\|J_{\mathbf{\Phi}_{12}}\|}\phi^2_{1,h}(x_1)\, dx_1  \\
	 +\int_{F_{2}}\phi^2_{2,h}(x_2)-
	 \frac{1}{2}\phi^2_{1,h}(\mathbf{\Phi}_{21}(x_2))\|J_{\mathbf{\Phi}_{21}}\|
	-\frac{1}{2\|J_{\mathbf{\Phi}_{21}}\|} \phi^2_{2,h}(x_2)\, dx_2 \\
\geq 
\frac{1}{2}\int_{F_{1}} (1-\frac{1}{\|J_{\mathbf{\Phi}_{12}}\|}\big)\phi^2_{1,h}(x_1)\,dx_1 +
\frac{1}{2}\int_{F_{2}} (1-\frac{1}{\|J_{\mathbf{\Phi}_{21}}\|}\big)\phi^2_{2,h}(x_2)\,dx_2 .
	\end{multline}
Noting that $\|J_{\mathbf{\Phi}_{12}}\|=\sqrt{1+ |\nabla {\Phi}_{12,z}|^2} $   and 
$\|J_{{\mathbf{\Phi}}_{21}}\|=\sqrt{1+ |\nabla {\Phi}_{21,z}|^2}$, see (\ref{U_o3U_o2_a2}),  and
 using Proposition \ref{Propos_1}, 
 we easily obtain inequality (\ref{3_5_16_a}). 
$\BLACKBOX$  
\end{proof}

\begin{lemma}\label{lemma_0}
	The bilinear form $B_h(\cdot,\cdot)$ in (\ref{7_d5}) is bounded and elliptic on $V_h$, i.e., 
	 there are positive constants $C_M$ and $C_m$ such that the estimates
	\begin{align}\label{B_dG_bound}
		\qquad B_h(v_h,\phi_h) &\leq C_M \|v_h\|_{dG}\|\phi_h\|_{dG}
		\quad \text{and}\quad  B_h(v_h,v_h) \geq C_m \|v_h\|^2_{dG},                             
	\end{align}
	hold for all $\phi_h\in V_h$ 
	{\color{black} provided that $\eta$ is sufficiently large.}
\end{lemma}
\begin{proof}
Let $v_h,\,\phi_h\in V_h$.
A simple application of the Cauchy-Schwartz inequality (\ref{HolderYoung}) yields
\begin{flalign}\label{3_5_22_a}
	\sum_{i=1}^N\int_{\Omega_i^*}\rho_i\nabla v_{i,h}\cdot\nabla \phi_{i,h}\,dx & \leq \|v_{h}\|_{dG}\,\|\phi_{h}\|_{dG}.
\end{flalign}	
Now, proceeding as in (\ref{7_d7a_0}) and then using (\ref{HolderYoung}) and trace inequalities for B-spline functions, 
see \cite{LT:LangerToulopoulos:2014a}, we can infer that
\begin{flalign}\label{3_5_22}
\nonumber
	\sum_{i=1}^N\sum_{F_{ij}\subset \partial\Omega_i^*}
\Big|  \int_{F_{ij}}\Big(\frac{\rho_i}{2}\nabla v_{i,h}+\frac{\rho_j}{2}\nabla v_{j,h}\Big)\cdot n_{F_{ij}}\phi_{i,h}  -
  \frac{\eta\{\rho\}}{h}\big(v_{i,h}-v_{j,h}\big)\phi_{i,h}\, d\sigma \Big|\\
   +\Big|\int_{\partial \Omega_i^* \cap \partial \Omega}\rho_i\nabla v_{i,h}\cdot n_{\partial \Omega_i^*} v_{i,h}\,d\sigma \Big| 
   \leq  C_{tr}\| v_h\|_{dG} \|\phi_{h}\|_{dG}.
\end{flalign}
For the fluxes on the faces of the gap boundaries, we follow  the same steps as in (\ref{3_5_22}).
Using using (\ref{U_o3U_o2_a6}),
we have that 
 \begin{multline}\label{3_5_23}
 	\Big|\sum_{i=1}^N\sum_{F_{gj}\subset \partial\Omega_i^*}
 	\int_{F_{gj}}\Big(\frac{\rho_i}{2}\nabla v_{i,h}+\frac{\rho_j}{2}\nabla v_{j,h}\Big)\cdot n_{F_{gj}}\phi_{i,h}  -
  \frac{\eta\{\rho\}}{h}\big(v_{i,h}-v_{j,h}\big)\phi_{i,h}\, d\sigma \Big| \\ 
  \leq 
  \sum_{i=1}^N\sum_{F_{gj}\subset \partial\Omega_i^*}
 	\int_{F_{gj}}\Big(\frac{\rho_i}{2}|\nabla v_{i,h}|\,|\phi_{i,h}|+
 	\frac{\rho_j}{2}|\nabla v_{j,h}|\,|\phi_{i,h}|  +
  \frac{\eta\{\rho\}}{h}\big(|v_{i,h}|\,|\phi_{i,h}| +  |v_{j,h}|\,|\phi_{i,h}|\big)\, d\sigma  \\ 
    \leq 2C_{tr}C(\rho)C(\|J_{\mathbf{\Phi}_{gij}}\|) \Big(\sum_{i=1}^N
   h\rho_i\| \nabla v_{i_h}\|^2_{L^2(\partial \Omega^*_{i})}\Big)^{\frac{1}{2}}\Big(\sum_{i=1}^N
   h^{-1}\{\rho\}\| \phi_{i,h}\|^2_{L^2(\partial \Omega^*_{i})}\Big)^{\frac{1}{2}} \\
  + 2C(\|J_{\mathbf{\Phi}_{gij}}\|) \Big(\sum_{i=1}^N
   \frac{\eta\{\rho\}}{h}\|  v_{i_h}\|^2_{L^2(\partial \Omega^*_{i})}\Big)^{\frac{1}{2}}\Big(\sum_{i=1}^N
   \frac{\eta\{\rho\}}{h}\| \phi_{i,h}\|^2_{L^2(\partial \Omega^*_{i})}\Big)^{\frac{1}{2}} 
   \leq 
   C \|v_{h}\|_{dG}\,\|\phi_{h}\|_{dG}.
 \end{multline}
For the interior faces of the overlapping patches, we follow the same procedure as in (\ref{3_5_23}):
\begin{multline}\label{3_5_24}
\Big| \sum_{i=1}^N\sum_{F_{oij}\subset \partial\Omega_i^*}  \int_{F_{oij}}\Big(\frac{\rho_i}{2}\nabla v_{i,h}+\frac{\rho_j}{2}\nabla v_{j,h}\Big)\cdot n_{F_{oij}}\phi_{i,h}  -
  \frac{\eta\{\rho\}}{h}\big(v_{i,h}-v_{j,h}\big)\phi_{i,h}\, d\sigma\Big| \leq 
   C \|v_{h}\|_{dG}\,\|\phi_{h}\|_{dG}.
\end{multline}
Collecting the above 
estimates,
we can derive the left inequality in (\ref{B_dG_bound}).
\par
In order to show the coercivity of $B_h(\cdot,\cdot)$ on the IgA space $V_h$,
we treat   the normal flux terms on every $\partial \Omega_i^*$ 
 as in Lemma \ref{lemma00}. More precisely, for the matching interfaces, using (\ref{HolderYoung}) and trace inequality,
  see details in \cite{LT:LangerToulopoulos:2014a}, we can show
  \begin{multline}\label{3_5_24a}
  - \sum_{i=1}^N\sum_{F_{ij}\subset \partial\Omega_i^*}
    \int_{F_{ij}}\Big(\frac{\rho_i}{2}\nabla v_{i,h}+\frac{\rho_j}{2}\nabla v_{j,h}\Big)\cdot n_{F_{ij}}v_{i,h}  +
  \frac{\eta\{\rho\}}{h}\big(v_{i,h}-v_{j,h}\big)v_{i,h}\, d\sigma \\
  \geq -C_{tr}\sum_{i=1}^N \|\rho_i^{\frac{1}{2}}\nabla u_{i,h}\|_{L^{2}(\Omega_i^*)}
         \Big(\sum_{i=1}^N \sum_{F_{ij}\subset \partial \Omega_i^*} 
         \frac{\{\rho\}}{h}\|v_i -v_j \|^2_{L^2(F_{ij})}\Big)^{\frac{1}{2}} 
+         \Big(\sum_{i=1}^N \sum_{F_{ij}\subset \partial \Omega_i^*} 
         \frac{\eta\{\rho\}}{h}\|v_i -v_j \|^2_{L^2(F_{ij})}\Big).
  \end{multline}
Following  the same steps as in (\ref{3_5_23}) and consequently 
proceeding  as in (\ref{3_5_24a}) and  using (\ref{3_5_16_aa}), 
we can derive estimate
\begin{multline}\label{3_5_24aa}
   - \sum_{i=1}^N\sum_{F_{gj}\subset \partial\Omega_i^*}
    \int_{F_{gj}}\Big(\frac{\rho_i}{2}\nabla v_{i,h}+\frac{\rho_j}{2}\nabla v_{j,h}\Big)\cdot n_{F_{oij}}v_{i,h}  +
  \frac{\eta\{\rho\}}{h}\big(v_{i,h}-v_{j,h}\big)v_{i,h}\, d\sigma \\
   - \sum_{i=1}^N\sum_{F_{oij}\subset \partial\Omega_i^*}
    \int_{F_{oij}}\Big(\frac{\rho_i}{2}\nabla v_{i,h}+\frac{\rho_j}{2}\nabla v_{j,h}\Big)\cdot n_{F_{oij}}v_{i,h}  +
  \frac{\eta\{\rho\}}{h}\big(v_{i,h}-v_{j,h}\big)v_{i,h}\, d\sigma \\  
  \geq -2C_{tr}C(\rho,\|J_{\mathbf{\Phi}_{ij}}\|) \Big(\sum_{i=1}^N
   h\rho_i\| \nabla v_{i_h}\|^2_{L^2(\partial \Omega^*_{i})}\Big)^{\frac{1}{2}}\Big(\sum_{i=1}^N
   h^{-1}\{\rho\}\| v_{i,h}\|^2_{L^2(\partial \Omega^*_{i})}\Big)^{\frac{1}{2}} + \\
   2C(\|J_{\mathbf{\Phi}_{ij}}\|) \Big(\sum_{i=1}^N
   \frac{\eta\{\rho\}}{h}\|  v_{i_h}\|^2_{L^2(\partial \Omega^*_{i})}\Big).\qquad
    \end{multline}
for the interior faces of gap and overlapping regions.   
Now, applying the second inequality of (\ref{HolderYoung}) in (\ref{3_5_24a}) and (\ref{3_5_24aa}) 
using  $\epsilon$ quite small and choosing $\eta$ to be large enough, 
we  can 
arrive at
the right inequality in (\ref{B_dG_bound}).
$\BLACKBOX$  
\end{proof} 
\par
{
The ellipticity the dG IgA bilinear form $B_h(\cdot,\cdot)$
on $V_h$ immediately yields that our dG IgA scheme (\ref{7_d4})  has a unique solution. 
}
Note that the solution $u$ satisfies   (\ref{7_d3_a}) but 
does not  satisfy the 
dG IgA variational problem
(\ref{7_d4}). The  dG IgA discretization derived is not consistent. 
We present below   the error analysis borrowing  ideas from 
weak consistent FE methods, see, e.g., \cite{ERN_FEM_book}. 
\subsection{Error estimates}
We are now in the position to
derive   error estimate for the proposed dG IgA scheme (\ref{7_d4}) under the Assumption \ref{Assumption1}. 
The procedure that we follow is similar to the corresponding procedure in \cite{HoferLangerToulopoulos_2015a} 
for the case of simple gap regions. 
The main differences  are due to the use of estimate (\ref{7_d7a}), which is referred     to the  case
of existing general gap and overlapping regions.  
 For the completeness of the paper, we  highlight below the main steps of the error analysis.  
The linearity of $B_h(\cdot,\cdot)$, see (\ref{7_d5}) and  (\ref{7_d3_b}), 
and the relations (\ref{7_d6}) and (\ref{7_d3_a}) yield
\begin{multline}\label{4.5_b}
	B_h(u_h-z_h,\phi_h) = B(u,\phi_h) +
	\sum_{i=1}^N\frac{\eta\rho}{h}\int_{\partial\Omega_i^* \cap \partial \Omega}(u-u_D) \phi_h\,d\sigma 
		-B_h(z_h,\phi_h) + F_h(\phi_h)-l_{f,\Omega_i^*}(\phi_h)  \\
=	B_h(u,\phi_h)+	R_{\Omega_o,\Omega_g}(u,\phi_h) -
	\sum_{i=1}^N\frac{\eta\rho}{h}\int_{\partial\Omega_i^* \cap \partial \Omega}u_D \phi_h\,d\sigma 
	-B_h(z_h,\phi_h)+ \sum_{i=1}^N\frac{\eta\rho}{h}\int_{\partial\Omega_i^* \cap \partial \Omega}u_D \phi_h\,d\sigma\\
	=B_h(u-z_h,\phi_h) +  R_{\Omega_o,\Omega_g}(u,\phi_h),\quad \text{for all} \quad \phi_h, {\ }z_h\in V_h.
\end{multline}
We choose $\phi_h=u_h-z_h$ in (\ref{4.5_b}). 
Let $\mathcal{K}_{g,o}$ be defined as in (\ref{Notations_gap_Overl}) and
the parameters  $\beta:=\lambda-\frac{1}{2}$, and $\gamma:=\lambda+1$.
Then, Lemma \ref{lemma_0},  Lemma \ref{lemma00} and the estimates (\ref{7_i0}) and (\ref{7_d7a})  imply
the estimate
\begin{align}\label{4.5.d}
	\|u-u_h\|_{dG} \leq C \Big( \big(\|u-z_h\|_{dG}^2 + 
   \sum_{i=1}^N \rho_i\,h\| \nabla (u-z_h)\|^2_{L^2(\partial \Omega_i^*)}\big)^{\frac{1}{2}}
+	(h^{\beta}+h^{\gamma})\,\mathcal{K}_{g,o}\Big)
\end{align}
that holds for all $z_h \in V_h$.
Now,  we can prove the main error estimate. 
 Such an estimate requires  quasi-interpolation estimates  of B-splines.  
  Using   results of multidimensional B-spline interpolation, 
 see  \cite{LT:LangerToulopoulos:2014a},   we can construct a quasi-interpolant
$\Pi_h: H^{\ell}\rightarrow V_h$ with $\ell \geq 1 $,  
  such that the following interpolation estimates are true.
 \begin{lemma}\label{lemma7}
 Let $u$ satisfy Assumption \ref{Assumption1} and let $\Pi_h$ be the interpolation operator as mentioned above.  
 Then
 there exist  constants $C_i>0$, $i=1,\ldots, N$, depending on the mappings $\mathbf{\Phi}_i$ and other geometric characteristics but independent of the grid sizes $h_i$  such that the interpolation estimate
  \begin{align}\label{4.5.d_1}
  	\Big(\|u-\Pi_h u\|^2_{dG} +\sum_{i=1}^N \rho_i\,h\| \nabla (u-\Pi_h u)\|^2_{L^2(\partial \Omega_i^*)}\Big)^{\frac{1}{2}}
  	  	& \leq   \sum_{i=1}^N C_i h_i^{s}\|u_i\|_{H^{{\color{black}\ell}}(\Omega_i^*)},
  \end{align}
   holds, with $s= \min(p+1,\ell)-1$.
 \end{lemma}
  \begin{proof}
  The proof for the interpolation error estimate (\ref{4.5.d_1}) can be given following the same 
  steps as in Lemma 10 from \cite{LT:LangerToulopoulos:2014a}.
  $\BLACKBOX$  
\end{proof} 

\begin{theorem}\label{Theorem_1_estimates}
Let $u$ be the solution of problem (\ref{7_d3_a}),  $u_h$ be the corresponding
dG IgA solution of problem (\ref{7_d4}), and let  $d_g = h^\lambda$ and $d_o = h^\lambda$ with $\lambda \geq 1$.  
Then the  error estimate
\begin{align}\label{4.5_e}
{
\|u-u_h\|_{dG} \lesssim  h^{r} \big( \sum_{i=1}^N\|u\|_{H^{{\ell}}(\Omega_i^*)} +
	 \mathcal{K}_{g,o}\big),
	 }
\end{align}
holds,  where $r=\min(s,\beta)$ with $s= \min(p+1,\ell)-1$ and $\beta=\lambda-\frac{1}{2}$,  and  
 $\mathcal{K}_{g,o}$ defined in (\ref{Notations_gap_Overl}). 
\end{theorem}
\begin{proof}
The required estimate  immediately follows   by   applying 
the quasi-interpolation error estimates (\ref{4.5.d_1}) in (\ref{4.5.d}).
	$\BLACKBOX$  
\end{proof} 
\begin{remark}\label{remark_3}
	The proceeding estimate is referred to the case where the maximum width 
	$d_M=\max(d_g,d_o)$ is  of order $\mathcal{O}(h^\lambda)$.
		If the widths $d_g$ and $d_o$ are fixed, 
	i.e., are not decreased as we refine the meshes, 
	then 	we can see that the bounds in (\ref{7_i0}) are given in terms of $d_M$. 
	Thus, 	using  (\ref{7_i0}) in the  analysis above 
	we can infer 	that the estimate (\ref{4.5_e}) 	will take the form
	\begin{align}\label{4.5_f}
		\|u-u_h\|_{dG} \lesssim  
		h^{s}  \sum_{i=1}^N\|u\|_{H^{{ \ell}}(\Omega_i^*)} +
	 d_M\,h^{-\frac{1}{2}}\,\mathcal{K}_{g,o},
	\end{align}
where  $s= \min(p+1,\ell)-1$ and $d_M=\max(d_g,d_o).$ 
\end{remark}
\begin{remark}\label{remark_4}
	The estimate given in (\ref{4.5_e}) concerns the distance between $u_h\in V_h$ and the solution $u\in V^*$ of
	the perturbed problem defined on $\mathcal{T}^*_H(\Omega)$. Based on (\ref{3_4_4a}), we can infer that the same estimate holds
	for the physical solution $u$ given by (\ref{Orig_weak_form_Decomp}).
\end{remark}

\section{Efficient IETI-DP Solvers}
\label{Section_IETI}
If we denote  the B-spline basis functions of $V_h$ on patch $\Omega^*_i$ by $\{B_j^{(i)}\}_{j=1}^n$, 
then the dG IgA discretization (\ref{7_d4}) leads to  the linear system
 \begin{align}
 \label{FullLinSys}
 	K\mathbf{u}=\mathbf{f},
 \end{align}
where $K$ is a symmetric, positive definite matrix assembled from the patch local matrices $K\sMP$ with entries $K_{j_1 j_2}\sMP=B_h(B_{j_2}\sMP,B_{j_1}\sMP)$   and the right hand side $\mathbf{f}$ assembled from $\mathbf{f}\sMP$ with entries $\mathbf{f}\sMP_{j_1}=F(B_{j_1}\sMP)$. 
The vector $\mathbf{u}$ is nothing but the coefficient (control point) 
representation of the dG-IgA solution $u_h\in V_h$.
 The dG-formulation (\ref{7_d4}) perfectly fits  to the method developed in \cite{HLT:HoferLanger:2016a}, the so called dG-IETI-DP method. This method is an adaption of the FETI-DP method for a composite FE and dG methods, see \cite{HLT:DryjaGalvisSarkis:2013a,HLT:DryjaSarkis:2014a}, to IgA. The first analysis of the method was done in \cite{HLT:VeigaChinosiLovadinaPavarino:2010a}, also in cases of having $C^\nu,$ $\nu>0,$ continuity across interfaces. In the subsection below, we summarize the key ingredients of the dG-IETI-DP method, and for notational simplicity, we will restrict ourselves to the two-dimensional  case. In order to use the dG-IETI-DP method for domains with gaps and overlaps, 
 some adaptations have to be done  due to the fact that
 the two sides of the interface are now not geometrically the same.
We aim at finding a reformulation of (\ref{FullLinSys}) which is easily parallelizeable and provides robustness  
with respect to the discretization parameter $h$ and 
large jumps in the diffusion coefficient $\rho$ across the interfaces.
The first step is to introduce additional dofs on the interface to decouple the  patchwise 
local problems and introducing Lagrange multipliers, say $\boldsymbol{\lambda}$, in order to enforce continuity of $u_h$. In the classical continuous Galerkin (cG) case this procedure is well known, see e.g. \cite{HLT:Pechstein:2013a,HLT:ToselliWidlund:2005a}. 
However, it is not obvious 
how this can be done in the dG case.  
An appropriate technique was first introduced in \cite{HLT:DryjaGalvisSarkis:2013a}  
and extended to three-dimensional  {\color{black}
diffusion problems}
in  \cite{HLT:DryjaSarkis:2014a}. 
The basic idea consists in introducing
an additional layer of dofs on the interface and introduce Lagrange multipliers between each layer. This enforces continuity in a certain sense. More precisely, we now work on an extended domain $\Omega_e\sMP$ given by, 
\begin{align*}
  \overline{\Omega}^{(i)}_e := \overline{\Omega^*}^{(i)} \cup \{\bigcup_{\ell\in{\mathcal{I}}_{\mathcal{F}}^{(i)}} \overline{F}_{\ell i}\},
\end{align*}
where $\mathcal{I}_{\mathcal{F}}^{(i)}$ is the set of all
indices of neighbouring patches and $\overline{F}_{\ell i}$ is the edge of ${\overline{\Omega^*}}\sMP[\ell]$ shared with ${\overline{\Omega^*}}\sMP[i]$. In the case of overlaps, we define $\overline{F}_{\ell i}$ to be $\overline{F}_{o\ell i}$, i.e., the face of ${\overline{\Omega^*}}\sMP[\ell]$, which lies in ${\overline{\Omega^*}}\sMP[i]$.
This also leads to the corresponding B-spline space 
\begin{align*}
  V_{h,e}^{(i)} := V_{h}^{(i)} + \prod_{l\in{\mathcal{I}}_{\mathcal{F}}^{(i)}}V_{h}^{(i)}(\overline{F}_{\ell i}),
\end{align*}
where $V_{h}^{(i)}(\overline{F}_{\ell i}) := \text{span}\{B_j^{(\ell)} \,|\, \text{supp}\{B_j^{(\ell)}\}\cap\overline{F}_{\ell i} \neq \emptyset\} 
\subset V_{h}^{(\ell)}$. Moreover, a function in $u_h^{(i)} \in V_{h,e}^{(i)} $ will be represented as
\begin{align}
 u_h^{(i)} = \{(u_h^{(i)})^{(i)}, \{(u_h^{(i)})^{(\ell)}\}_{\ell\in {\mathcal{I}}_{\mathcal{F}}^{(i)}}\},
\end{align}
where $(u_h^{(i)})^{(i)}$ and $(u_h^{(i)})^{(\ell)}$  are the restrictions of $u_h^{(i)}$ to ${\Omega^*}^{(i)}$ and $\overline{F}_{\ell i}$, respectively. Next, we need the notion of ``continuity'' on the interface, cf. Definition 3.1 in \cite{HLT:HoferLanger:2016a}. We say, a function 
$u_h\in V_{h,e}:=V_{h,e}^{(1)}+\cdots+V_{h,e}^{(N)}$ 
is continuous across the interface, if 
\begin{align*}
    (\boldsymbol{u}^{(i)})^{(i)}_s &=  (\boldsymbol{u}^{(\ell)})^{(i)}_t \quad\forall (s,t)\in B_e(i,\ell), \;\forall \ell\in{\mathcal{I}}_{\mathcal{F}}^{(i)},\\
    (\boldsymbol{u}^{(i)})^{(\ell)}_s &=  (\boldsymbol{u}^{(\ell)})^{(\ell)}_t \quad\forall (s,t)\in B_e(\ell,i), \;\forall \ell\in{\mathcal{I}}_{\mathcal{F}}^{(i)}
\end{align*}
holds for all $i=1,\ldots,N$, where $B_e(i,\ell)$ denotes the set of all index pairs $(s,t)$, such that the $s$-th basis function in $V_h\sMP$ can be identified with the $t$-th basis function in $V_h\sMP(\overline{F}_{\ell i})$. Since these $m$ conditions are linear, there exists a matrix $B$, which realize them in the form  $B\mathbf{u}=0$. 
Let $K_e:=\text{diag}(K_e\sMP)$ and $\mathbf{f}_e:= [\mathbf{f}_e\sMP]$. 
We can now reformulate (\ref{FullLinSys}) as follows: 
find $(u_h,\boldsymbol{\lambda}) \in V_{h,e} \times \mathbb{R}^m:$
\begin{align}
    \label{equ:saddlePoint}
     \MatTwo{K_e}{B^T}{B}{0} \VecTwo{\mathbf{u}}{\boldsymbol{\lambda}} = \VecTwo{\mathbf{f}_e}{0}
 \end{align}
 In order to guarantee the invertibility of $K_e$ and a global information transfer,
 we introduce the set of primal variables   $\Psi \subset V_{h,e}^{'}$ and consider the system, restricted to the space 
 \begin{align*}
 	\widetilde{V}_{h,e}=\{v\in V_{h,e}: \psi(v\sMP)= \psi(v\sMP[\ell]), \forall \psi\in \Psi, \forall i>\ell\}.
 \end{align*}
Typical choices for $\Psi$ are:
\begin{itemize}
    \item Vertex evaluation: $\psi^\mathcal{V}(v) = v(\mathcal{V})$,
    \item Face averages: $\psi^{F}(v) = \frac{1}{|{F}|}\int_{{F}}v\,ds{\ } \text{for all}\,F\in\mathcal{F}^*$,
\end{itemize}
where the definitions have to be adopted in order to fit to the additional layer on the interface, cf. Definition 3.2 in \cite{HLT:HoferLanger:2016a}. 
Finally, from (\ref{equ:saddlePoint}), we obtain the {dG IETI-DP system} 
 of the form 
\begin{align}
    \label{equ:saddlePointReg}
     \MatTwo{\widetilde{K}_e}{\widetilde{B}^T}{\widetilde{B}}{0} \VecTwo{\mathbf{u}}{\boldsymbol{\lambda}} = \VecTwo{\widetilde{\mathbf{f}_e}}{0},
 \end{align}
 where $\widetilde{K}_e$ is SPD. From this, we derive the Schur complement problem:
 \begin{align}
   \label{equ:SchurFinal}
      \text{Find } \boldsymbol{\lambda} \in \mathbb{R}^m: \quad F_{S}\boldsymbol{\lambda} = \mathbf{d},
\end{align}
where $F_{S}:= \widetilde{B} \widetilde{K}_e^{-1}\widetilde{B}^T \text{and } \mathbf{d}:= \widetilde{B}\widetilde{K}_e^{-1} \widetilde{\mathbf{f}_e}$. 
The system \ref{equ:SchurFinal} is solved by means of
the preconditioned conjugate gradient method, using the so called scaled Dirichlet preconditioner $M_{sD}^{-1}$
defined by
\begin{align}
   \label{equ:scaled_Dirichlet}
	M_{sD}^{-1} := B_D S_e B_D^T,
   \end{align}
where $S_e$ denotes the
block diagonal matrix of the patch local Schur complements $S_e\sMP$, 
and $B_D$ is a scaled version of the matrix $B$. The matrix $S_e\sMP$ is defined by $ {K}^{(i)}_{e,B_e B_e } - {K}^{(i)}_{e,B_e I}\left({K}^{(i)}_{e,II}\right)^{-1} {K}^{(i)}_{e,IB_e}$, where the subscript  $I$ denotes the interior dofs and $B_e$ the dofs on the interface and additional layer. 
The matrix $B_D$ is defined such that the operator enforces the constraints
\begin{align*}
  {\delta^\dagger}^{(\ell)}_t(\boldsymbol{u}^{(i)})^{(i)}_s -  {\delta^\dagger}^{(i)}_s(\boldsymbol{u}^{(\ell)})^{(i)}_t &= 0 \quad\forall (s,t)\in B_e(i,\ell), \;\forall \ell\in{\mathcal{I}}_{\mathcal{F}}^{(i)},\\
  {\delta^\dagger}^{(\ell)}_t(\boldsymbol{u}^{(i)})^{(\ell)}_s -  {\delta^\dagger}^{(i)}_s(\boldsymbol{u}^{(\ell)})^{(\ell)}_t &= 0 \quad\forall (s,t)\in B_e(\ell,i), \;\forall \ell\in{\mathcal{I}}_{\mathcal{F}}^{(i)},
\end{align*}
for $(s,t)\in B_e(i,\ell)$,  where
$\textstyle{
{\delta^\dagger}^{(i)}_s := 
  \eta^{(i)}_s / \sum_{l \in {\mathcal{I}}_{\mathcal{F}}^{(i)}} \eta^{(l)}_t}
  $
 is an appropriate scaling. Typical choices for $\eta^{(i)}_s$ are
$ \eta^{(i)}_s := \rho^{(i)} = \rho(x)|_{\Omega^{(i)}}$ (coefficient scaling)
and $\eta^{(i)}_s := {\boldsymbol{K}_e}_{s,s}^{(i)}$ (stiffness scaling).
For an efficient way to implement the action of $F$ and $M_{sD}^{-1}$, we refer, e.g., to \cite{HLT:Pechstein:2013a,HLT:HoferLanger:2016b}. As concluded in \cite{HLT:HoferLanger:2016b}, we expect the condition number bound
\begin{align}
\label{equ:condBound}
       \kappa(M_{sD}^{-1} F_{S}|_{\text{ker}(F_S)^\perp}) \leq C \max_i\left(1+\log\left(\frac{H^{(i)}}{h^{(i)}}\right)\right)^2,
\end{align}
where $H^{(i)}$ and $h^{(i)}$ are the patch size and mesh size, respectively, 
and the positive  constant $C$ is independent of $H^{(i)}$, $h^{(i)}$ and $\rho$.
\section{Numerical tests}
\label{Section_numerics}
In this section, we perform several numerical tests  with different shapes of gap and overlap regions as well as combinations with inhomogeneous diffusion coefficients for two- and three- dimensional problems. We investigate  the order of accuracy of the dG IgA scheme proposed in (\ref{7_d5}).  All examples have been performed using  second degree ($p=2$)  B-spline spaces, apart from one where   third degree ($p=3$) B-splines have been used. 
We  compare the error convergence rates versus the grid size $h$ for several  gap/overlapping  distances $d_M=\max\{d_o,d_g\}=h^\lambda$, with  $\lambda \in \{1,2,2.5,3\}$. Every example has been solved
 applying several mesh refinement steps with $h_i,h_{i+1},\ldots,$ satisfying Assumption \ref{Assumption2}.
 The numerical convergence rates $r$ have been  computed by the ratio $r =\frac{\ln(e_i/e_{i+1})}{\ln (h_i /h_{i+1} )}, \,i=1,2,\ldots$, where the error $e_i:=\|u-u_h\|_{dG}$ is always computed  
on the  meshes $\cup_{i=1}^N T^{(i)}_{h_i,\Omega_i^*}$.
We mention that, in the test cases, we use highly smooth solutions on each patch, i.e., $p+1 \leq \ell$, and therefore the order  $s$ in  (\ref{4.5_e}) and (\ref{4.5.d_1}) becomes $s = p$.
    The predicted values of  power  $\beta$, the order $s$ and the order $r$ in (\ref{4.5_e}) for 
    several values of $\lambda$  are displayed in Table \ref{table_value_r}. 
All tests have been performed in G+SMO \cite{gismoweb},
which is a generic object-oriented C++ library for IgA computations,
see also \cite{HLT:JuettlerLangerMantzaflarisMooreZulehner:2014a,LangerMantzaflarisMooreToulopoulos_IGAA_2014a}.
%
\begin{table}
  \centering
  \begin{tabular}{|c||c|c|c|c|}
 \hline
  & \multicolumn{4}{|c|}{B-spline degree $p$   }\\ \hline  
  & \multicolumn{4}{|c|}{Smooth solutions, $u\in H^{\ell\geq p+1}$}  \\ \hline
   $d_M=h^\lambda$       &$\lambda=1$& $\lambda=2$ &$\lambda=2.5$ &$\lambda=3$\\\hline
  $\beta:=$              &0.5        &  1.5        & 2                    &2.5\\   \hline
  $s:=$                  &$p$        &  $p$        & $p$                   &$p$\\   \hline
  $r:=$                  &$0.5$      &  $1.5$      & $\min(p,\beta)$       &$\min(p,\beta)$\\ 
 \hline
\end{tabular}
\caption{The values of the predicted order $r$ of the estimate (\ref{4.5_e}) in Theorem 1.}
\label{table_value_r}
\end{table}
In any test case, the  gap and overlap regions  are artificially created by moving 
the control points, which are related to the interfaces $F_{ij}$, 
in the direction of $n_{F_{ij}}$ or of $-n_{F_{ij}}$.  In order to solve the resulting linear system, we use the dG-IETI-DP method, 
that is described in  Section~\ref{Section_IETI},  
as a fast and robust solver, see also \cite{HLT:HoferLanger:2016a}. 
We utilize OpenMP to
parallelise this method with respect to the patches. We use the conjugate gradient method for solving (\ref{equ:SchurFinal}) preconditioned with the scaled Dirichlet preconditioner (\ref{equ:scaled_Dirichlet}). We start with zero initial guess and iterate until we reach a reduction of the initial residual in the $\ell_2$-norm by a factor $10^{-10}$.
\par
The numerical examples presented in \cite{HoferLangerToulopoulos_2015a} and \cite{HoferToulopoulos_IGA_Gaps_2015a} have been performed  on domains described by two patches and gaps only. 
In this work, we extend the examples to regions with gaps and overlaps, 
and to domains consisting of several patches. This gives rise to an efficient use of domain decomposition solvers
like the dG-IETI-DP method
introduced in Section~\ref{Section_IETI}. 
Moreover,   the introduction of dG techniques
on the subdomain interfaces makes the use of non-matching meshes  easier, see \cite{LT:LangerToulopoulos:2014a}.
Keeping a constant linear relation between the sizes of the different meshes, 
the approximation properties
of the method are not affected, see \cite{LT:LangerToulopoulos:2014a}. 
In the examples  below, we exploit this advantage of the dG methods and 
first solve two-dimensional problems considering
non-matching meshes.  The   convergence rates are expected  to be the same as those displayed in Table \ref{table_value_r}. 
\subsection{Two-dimensional numerical examples}
The control points with the corresponding knot vectors of the domains given in Example 1-3 are available under the names \verb+yeti_mp2.xml+, \verb+9pSquare.xml+ and \verb+4x1pCurved.xml+ as \verb+xml+ files in 
G+SMO\footnote{G+SMO: https://www.gs.jku.at/trac/gismo}.

\paragraph{Example  1: uniform diffusion coefficient $\rho_i=1,\, i=1,\ldots,N$.} 
The first numerical example is a simple test case demonstrating
the  applicability of the proposed technique for constructing 
dG IgA scheme on 
{\color{black}  segmentations
}
including  gaps and overlaps  with general shape. 
The domain  $\Omega$ with the $N=21$ subdomains $\Omega_i^*$ 
 and the initial mesh  are shown in Fig.~\ref{Fig1_Test_1Gaps}(a). We note that we consider non-matching meshes across the interfaces. 
The Dirichlet boundary condition and the right hand side $f$ are determined by the exact solution 
$u(x,y)=\sin(\pi(x+0.4)/6)\sin(\pi(y+0.3)/3)+x+y $. 
 In this example, we consider the homogeneous diffusion case, i.e., 
$\rho_i=1$ for all  $\Omega_i^*,\,i=1,\ldots,N$.\\
 We performed {four} groups of  computations, where for every group 
 the size  $d_M=\max\{d_g,d_o\}$ was defined to be $\mathcal{O}(h^\lambda)$, 
 with $\lambda \in \{1,2, 2.5,3\}$. 
 In Fig.~\ref{Fig1_Test_1Gaps}(b) we present the discrete solution for $d_M = h$.
 Since we are using second-order ($p=2$) B-spline space, based on Table 
 \ref{table_value_r}, we expect optimal convergence rates for $\lambda =2.5$ and $\lambda = 3$.  
 The numerical convergence rates for several levels of mesh refinement are plotted in Fig.~\ref{Fig1_Test_1Gaps}(c). They are in very good	agreement with the theoretically predicted estimates given in Theorem \ref{Theorem_1_estimates}, 
 see also Table \ref{table_value_r}. We observe that we have optimal rates for the cases where $\lambda \geq 2.5$.
    As a second test in the same example,  we solve
   the problem considering  gap and overlapping regions with fixed width $d_M=0.004$.                 
                   Since the width  of the gap remains fixed,  Theorem \ref{Theorem_1_estimates} yields  
                   the   error estimate given in (\ref{4.5_f}). 
                                    We start by solving  the problem on coarse meshes with 
                   $h^2 > {d_M}h^{-\frac{1}{2}}$,           then we continue to use finer meshes 
                   and finally we solve the problem on meshes  with grid size $h^2 < {d_M}h^{-\frac{1}{2}}$.
                        The convergence rates associated with this computations  are shown in Fig. \ref{Fig1_Test_1Gaps}(d).
                   For the first mesh levels,           
                   we obtain the expected optimal convergence rates, since the error coming from the approximation properties of the B-spline space, see first term in (\ref{4.5_f}), dominates the 
                   the  approximation error coming from the approximation of the normal fluxes on 
                   $\partial \Omega_{gij}$ and $\partial \Omega_{oij}$.
                   In the finer levels, where $h^2 \sim {d_M}h^{-\frac{1}{2}}$,
                   the  rates  are  gradually  reduced where eventually become close to zero, 
                   because the whole discretization error  is not further decreasing as we refine the mesh. 
                   As we move into the most refined meshes, 
                   the error related to  the approximation of B-splines is negligible compared to the error related
                   to the approximation of the fluxes on  $\partial \Omega_{gij}$ and $\partial \Omega_{oij}$,
                   which in fact, based on the  form of
                   the   second term     on the right hand side in (\ref{4.5_f}), this error 
                   seems to increase with rate $-0.5$.  
                 In the numerical computations, this fact is depicted in the negative rate $r=-0.48$ that we have found 
                      on the last  level of refinement in  Fig.  \ref{Fig2_Test_2Gaps}(d).                                  
\\
The dG-IETI-DP method as described in Section \ref{Section_IETI} performs very well. The condition number $\kappa$ and CG iterations (It.) are summarized in Table~\ref{IETI_Ex1} for the case $d_M=h$ with coefficient and stiffness scaling. We observe the quasi-optimal behaviour of the condition number with respect to $h$ and that the existence of gaps and overlaps does not 
{\color{black}
affect
}
the condition number, cf. (\ref{equ:condBound}). Moreover, both scaling gives nearly identical results.
 \begin{table}[htp]
\centering
\begin{tabular}{|r|c|cc|cc|cc|cc|}\hline
 \multicolumn{2}{|c|}{  }& \multicolumn{4}{|c|}{$W^{\mathcal{V}}$}& \multicolumn{4}{|c|}{$W^{\mathcal{V}+\mathcal{F}}$} \\ \hline
 \multicolumn{2}{|c|}{ } &  \multicolumn{2}{|c|}{coefficient scal.} & \multicolumn{2}{|c|}{stiffness scal.} &  \multicolumn{2}{|c|}{coefficient scal.} & \multicolumn{2}{|c|}{stiffness scal.} \\ \hline
	    $\#$ dofs & $H/h$ & $\kappa$ & It. & $\kappa$ & It. &$\kappa$ & It. & $\kappa$&  It. \\ \hline
	      654   & 4    &   1.61  &  10  &   1.61   &     10  &   1.24  &  7    &  1.25  &   7    \\ \hline
	      1616  & 8   &    2.02  &  12  &   2.02   &     11  &   1.36  &  8    &  1.35  &   8    \\ \hline
	      4716  & 16   &   2.48  &  13  &   2.51   &     13  &   1.56  &  10   &  1.56  &   10   \\ \hline
	      15620 & 32   &   3.00  &  14  &   3.08   &     15  &   1.88  &  11   &  1.88  &   11   \\ \hline
	      56244 & 64   &   3.56  &  15  &   3.78   &     16  &   2.25  &  12   &  2.26  &   12   \\ \hline
	      212756& 128  &   4.17  &  17  &   4.66   &     18  &   2.68  &  14   &  2.73  &   14  \\ \hline
	      826836& 256  &   4.82  &  18  &   5.94   &     21  &   3.17  &  15   &  3.28  &   17  \\ \hline
  \end{tabular}
  \caption{Example 1: Dependence of the condition number $\kappa$ and the number It. of iterations on $H/h$ for the preconditioned system with coefficient and stiffness scaling. 
  Choice of primal variables: vertex evaluation (left), vertex evaluation and edge averages (right).}
\label{IETI_Ex1}
  \end{table}

   \begin{figure}[h]
 \subfigure[]{\includegraphics[width=4.0cm, height=5.5cm]{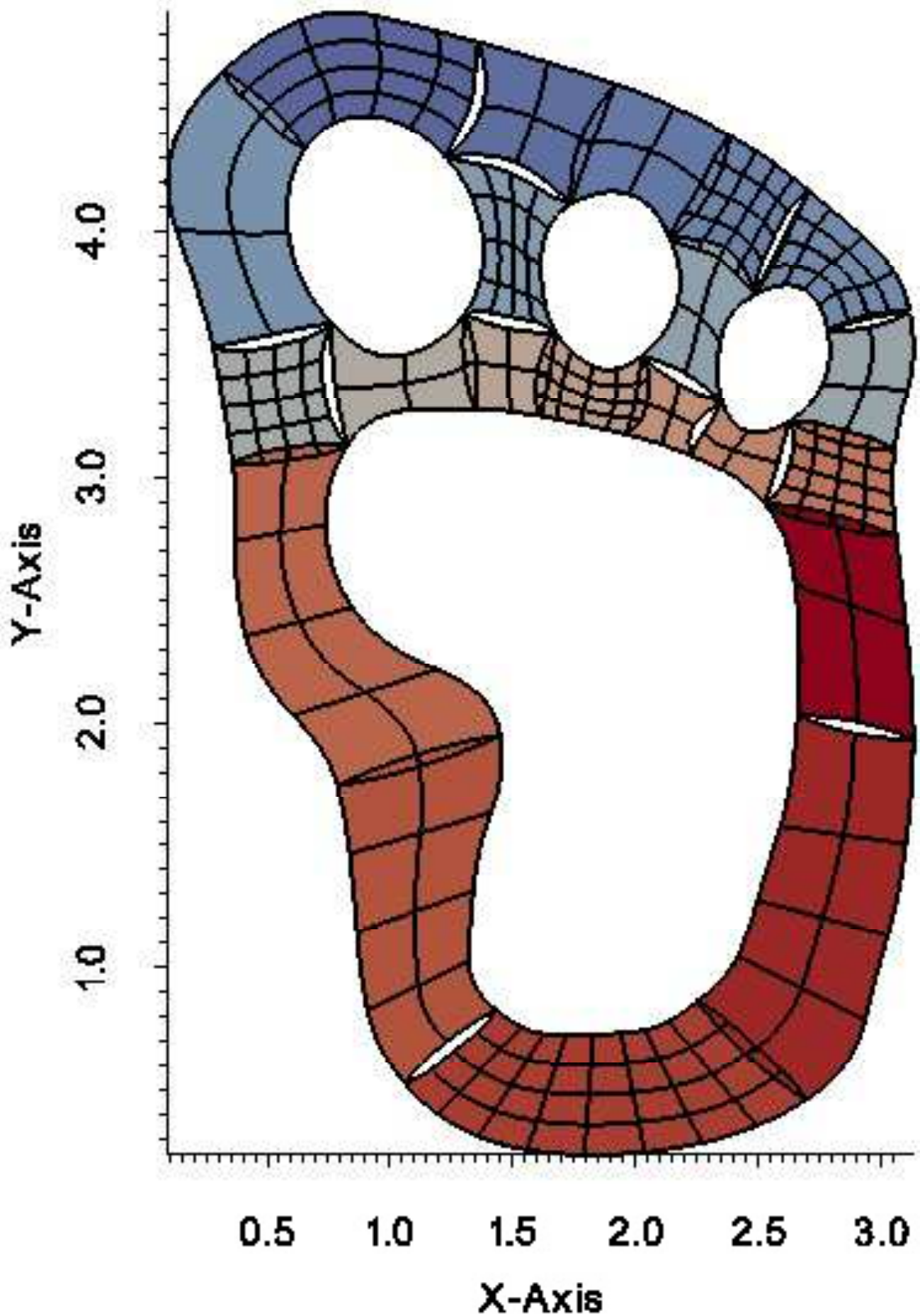}}\quad
 \subfigure[]{\includegraphics[width=4.0cm, height=5.5cm]{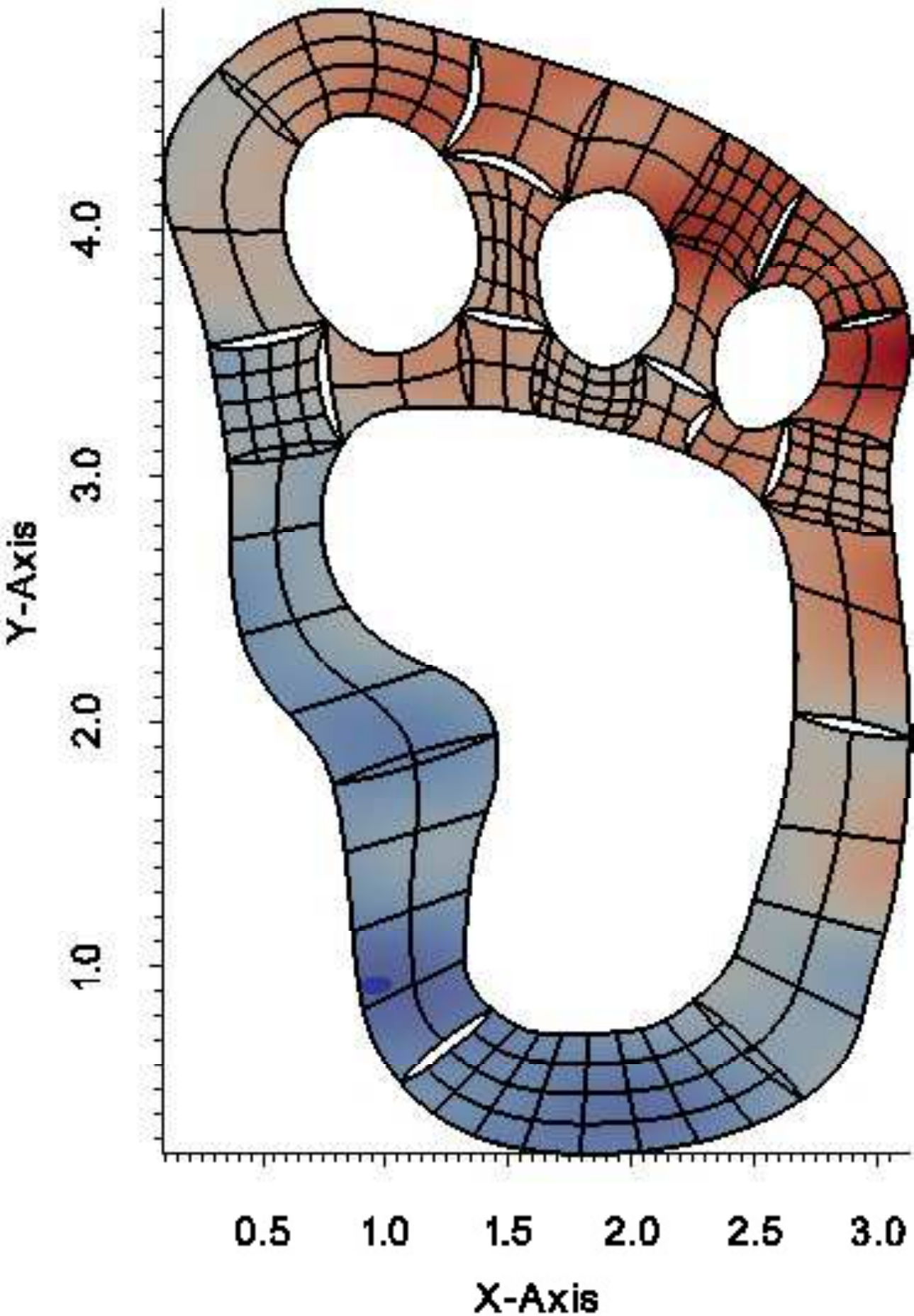}}\quad
\subfigure[]{\raisebox{1cm}{\includegraphics[width=4.20cm, height=4.25cm]{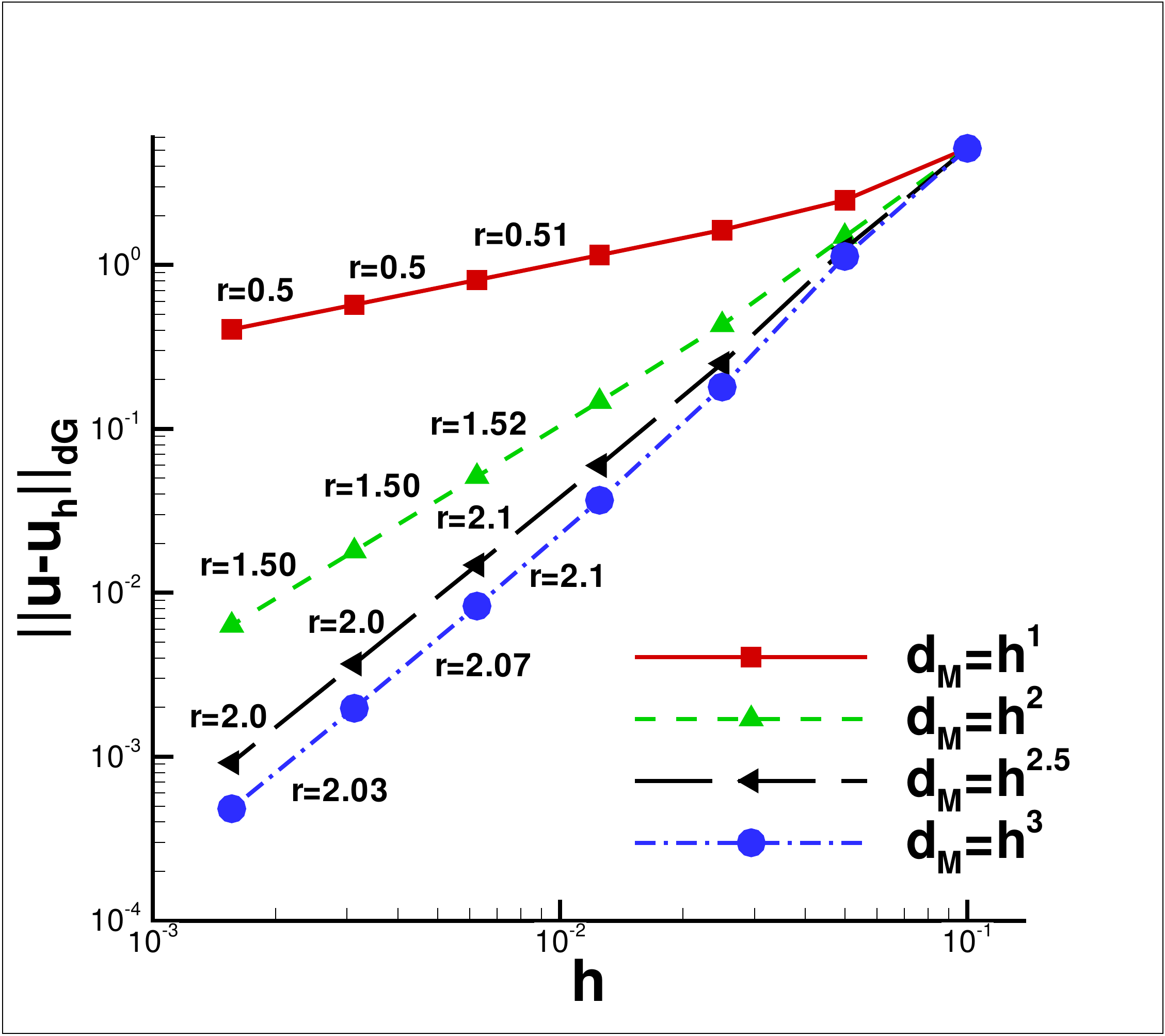}}}
 \subfigure[]{\raisebox{1cm}{\includegraphics[width=4.20cm, height=4.25cm]{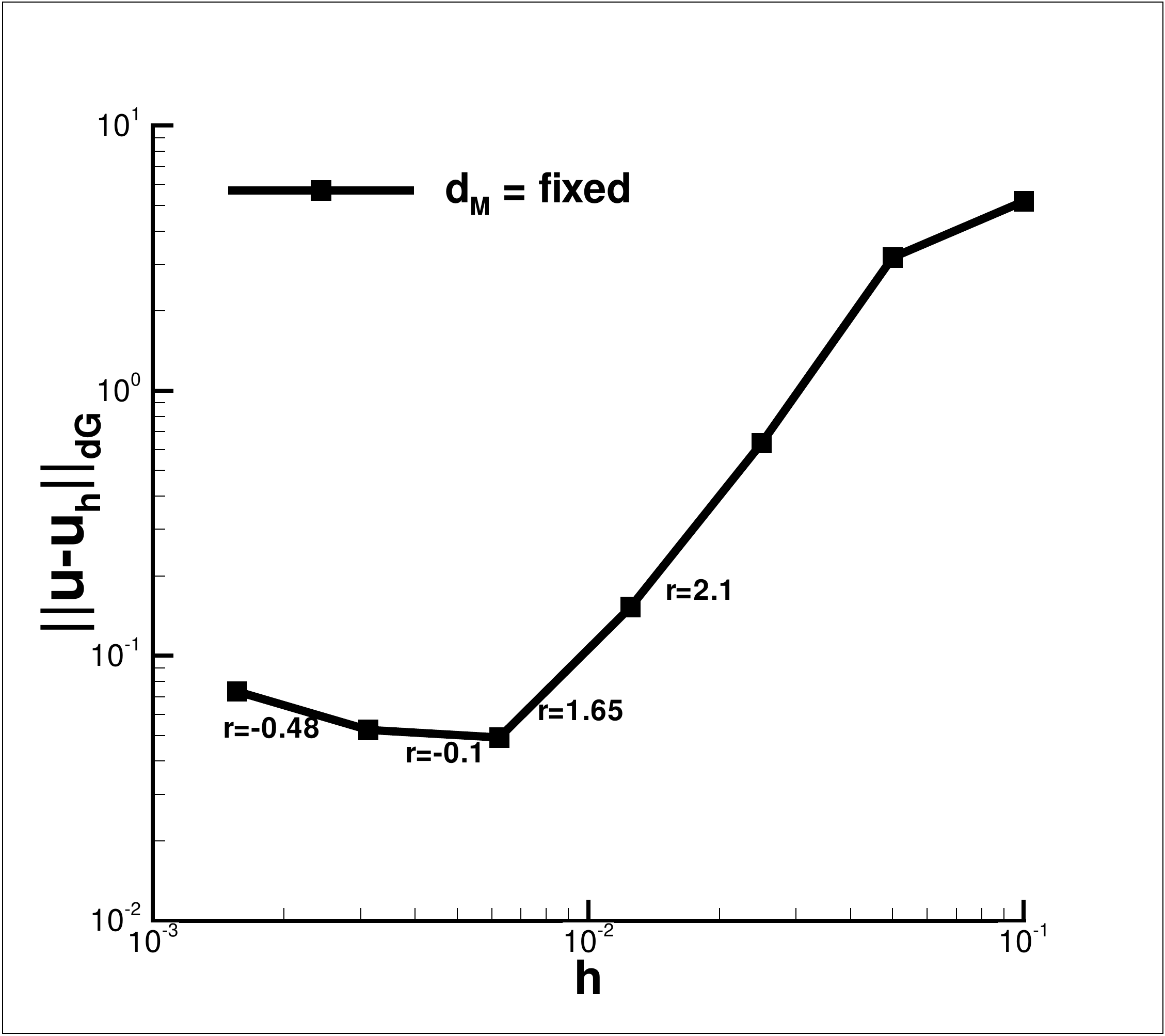}}}
   \caption{Example 1: (a) The patches $\Omega_i^*$ with the initial mesh. 
    (b) The contours of the $u_h$ solution for $d_g=d_o = h$. 
    (c)    The convergence rates for the different values of $\lambda$.
    (d)    The convergence rates for fixed gap and overlapping widths, $d_M=0.004$.}
   \label{Fig1_Test_1Gaps}
 \end{figure}
%
 \paragraph{Example  2: different diffusion coefficient.  
 } In the second example, we study the case of having smooth solutions on each $\Omega_i$ but discontinuous coefficient,  i.e.,  we set  $\rho_i \in \{1, \frac{3 \pi}{4}\approx 2.35 ,10^4\}$ according to the pattern in Fig.~\ref{Fig2_Test_2Gaps}(b). We consider a square build by a  $3\times3$ grid of patches, giving in total $N=9$ patches. The domain $\Omega$ with its patches $\Omega^*_i$  and the solution after one refinement are   presented  in Fig.~\ref{Fig1_Test_1Gaps}(a). The exact solution is given by 
\begin{align}\label{NE_1}
u(x,y)=
    \begin{cases}
	    \exp(-\sin(3\pi x))\sin(\frac{7\pi y}{2})& \text{if } {\ }x\in [0,1] \\
	    ((2x^2-1)+ (x-1)^2(x-2)\gamma)\sin(\frac{7\pi y}{2}) & \text{if} {\ }x\in (1,2]\\
	    7\exp(-\cos(\pi (x+1)/2)\sin(\frac{7\pi y}{2}) & \text{if} {\ }x\in (2,3],
  \end{cases} 
\end{align}
   where $\gamma = 14\cdot 10^4-8$. The  boundary conditions and the source function $f$ are  determined by (\ref{NE_1}).  Note that in this test case, we have  $\llbracket u \rrbracket |_{F}=0$ as well $\llbracket \rho \nabla u \rrbracket |_{F}\cdot n_F=0$ for the normal flux on the interfaces $F\in \mathcal{F}^*$.  The problem has been solved on  several meshes  
   following a sequential refinement process, where we set  $d_M=h^\lambda$, with $\lambda \in\{1,2,3,3.5,4\}$.  For the numerical tests, we use 
   {\color{black}
   B-splines of the degree $p=3$.
   }
   Hence, we expect optimal rates for $\lambda = 3.5$ and $\lambda = 4$. 
   In Fig.~\ref{Fig2_Test_2Gaps}(a) the approximate solution is presented on a relative coarse mesh  with $d_M = 0.06$.  The computed rates are presented in  Fig.~\ref{Fig2_Test_2Gaps}(c).
   By this example,  we numerically validate the   predicted convergence rates on $\mathcal{T}_H^*$ with gaps and overlaps 
   for diffusion problems with discontinuous coefficient and smooth solutions. 
   \par
   Moreover, the dG-IETI-DP method seems to be robust with respect to jumping diffusion coefficients in the presence  of complicated gaps and overlaps. The condition number $\kappa$ and CG iterations (It.) are summarized in Table~\ref{IETI_Ex2} for the case $d_M=h$ using coefficient and stiffness scaling. For this domain, the usage of more primal variables brings a significant advantage, 
   see column $W^{\mathcal{V}+\mathcal{E}}$ in Table~\ref{IETI_Ex2}.
   
\begin{table}[h!]
\centering
\begin{tabular}{|r|c|cc|cc|cc|cc|}\hline
\multicolumn{2}{|c|}{  }& \multicolumn{4}{|c|}{$W^{\mathcal{V}}$}& \multicolumn{4}{|c|}{$W^{\mathcal{V}+\mathcal{F}}$} \\ \hline
\multicolumn{2}{|c|}{ } &  \multicolumn{2}{|c|}{coefficient scal.} & \multicolumn{2}{|c|}{stiffness scal.} &  \multicolumn{2}{|c|}{coefficient scal.} & \multicolumn{2}{|c|}{stiffness scal.} \\ \hline
	    $\#$ dofs & $H/h$ & $\kappa$ & It. & $\kappa$ & It. &$\kappa$ & It. & $\kappa$&  It. \\ \hline
	    788    & 4    &   3.96   &  11    &  3.97   &    11  &  1.67  &   8    &   1.67  &   8   \\ \hline
	    1452   & 8   &    3.86   &  11    &  3.89   &    11  &  1.63  &   8    &   1.63  &   9   \\ \hline
	    3356   & 16   &   4.36   &  12    &  4.44   &    12  &  1.78  &   9    &   1.78  &   10  \\ \hline
	    9468   & 32   &   5.11   &  13    &  5.32   &    15  &  2.01  &   10   &   2.02  &   11  \\ \hline
	    30908  & 64   &   6.03   &  15    &  6.56   &    16  &  2.31  &   11   &   2.32  &   12  \\ \hline
	    110652 & 128  &   7.11   &  16    &  8.37   &    17  &  2.65  &   12   &   2.71  &   13 \\ \hline
  \end{tabular}
  \caption{Example 2: Dependence of the condition number ($\kappa$) and the number  CG iterations (It.) on $H/h$ for the preconditioned system with coefficient and stiffness scaling.  Choice of primal variables: vertex evaluation (left), vertex evaluation and edge averages (right).}
\label{IETI_Ex2}
  \end{table}
        
  \begin{figure}
   \begin{subfigmatrix}{3}
 \subfigure[]{\includegraphics[width=5.7cm, height=5.75cm]{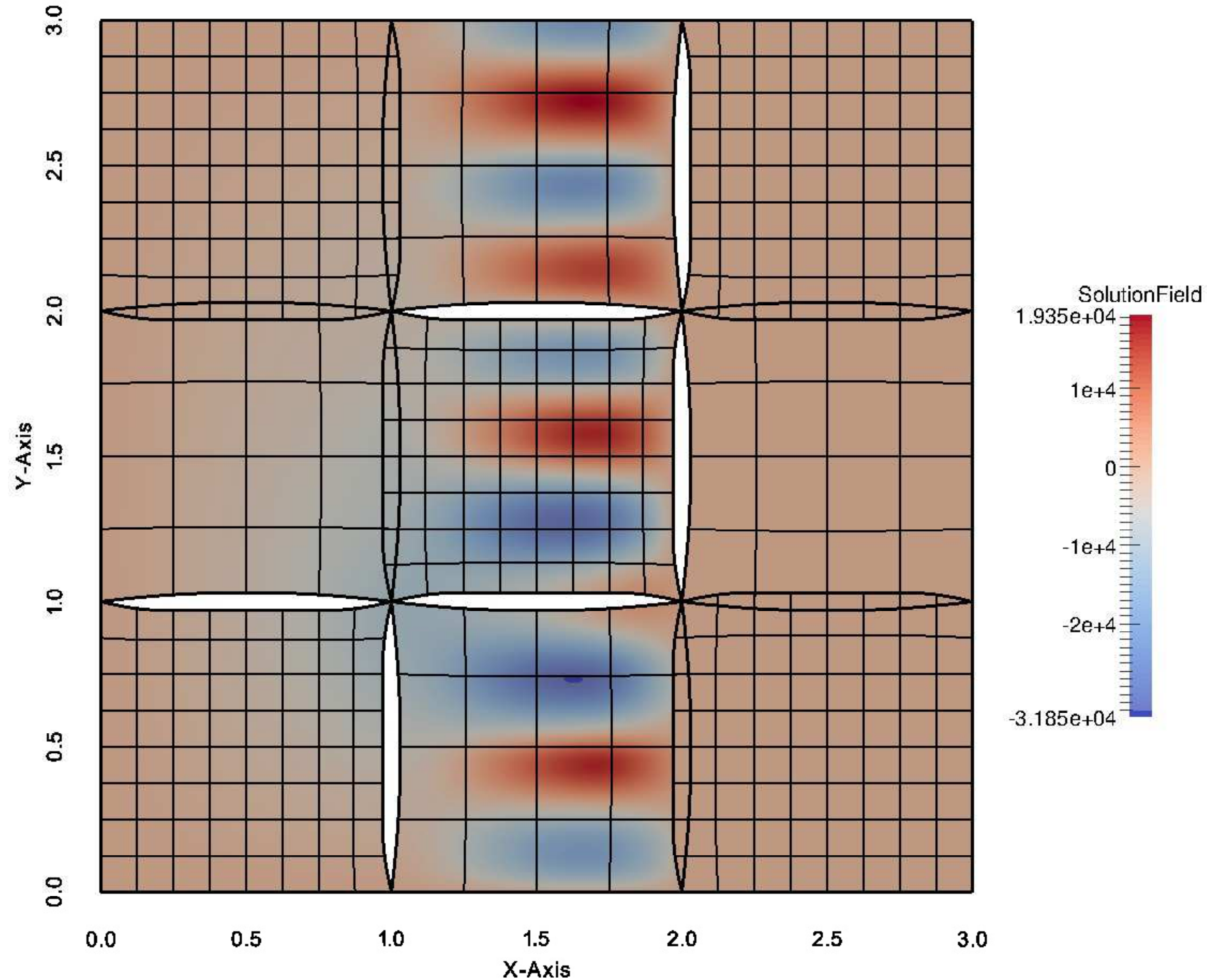}}
 \subfigure[]{\includegraphics[width=5.7cm, height=5.75cm]{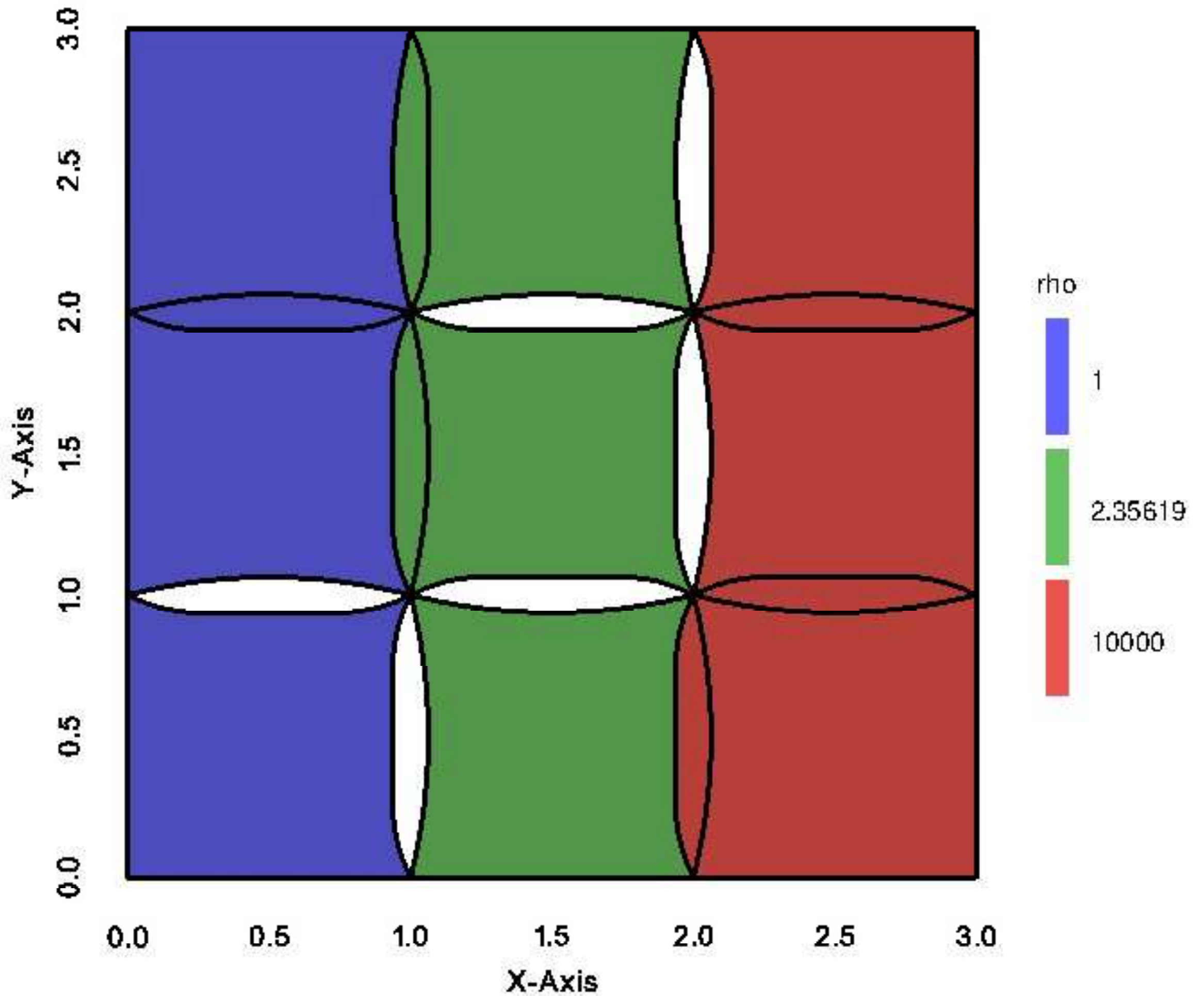}}
  \subfigure[]{\includegraphics[width=5.7cm, height=5.75cm]{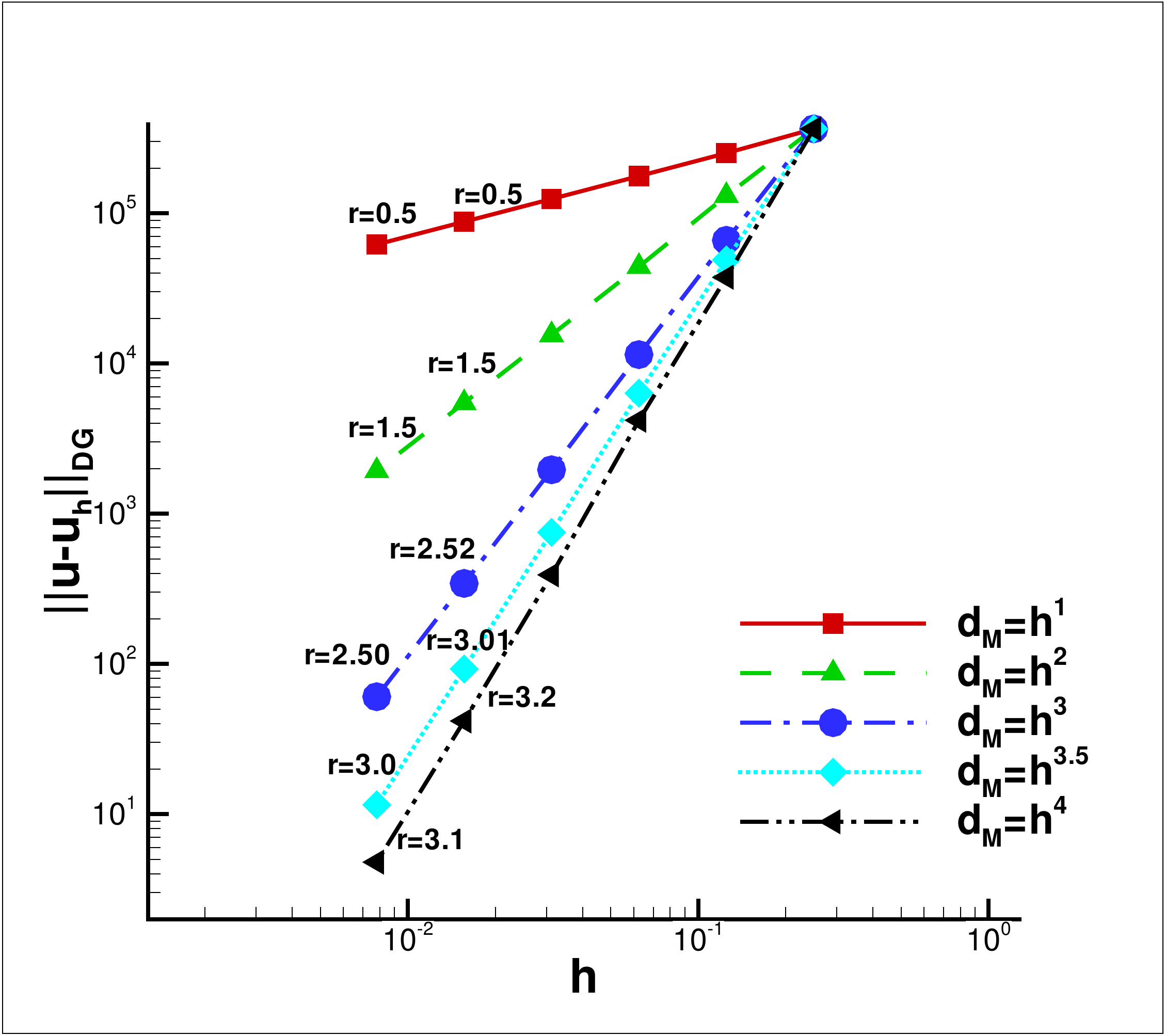}}
  \end{subfigmatrix}
   \caption{Example 2: (a) The contours of $u_h$ on the subdomains $\Omega_i$ with $d_g=0.06$,
                       (b) Pattern of diffusion coefficients $\rho_i$
                       (c) The convergence rates for the 5 choices of $\lambda$. }
   \label{Fig2_Test_2Gaps}
 \end{figure}
 \paragraph{Example  3: gap-overlap on the same interface and matching parametrized interfaces.} 
 In the this test, we apply the proposed method (\ref{7_d5}) on decompositions having a more complex gap and overlap regions.
 The domain is given by 4 patches with non-matching meshes, where we artificially created 
 {\color{black}
 both an overlap part and a gap part on one of the interfaces 
 as depicted in Fig.~\ref{Fig3_Test_GapOverlap}(a). 
 }
 This region is located at the interface between $\Omega_2$ and $\Omega_3$ at $x=0$,  
 We note that the gap and the overlap are not separated by patches as in the previous examples. In addition, we  consider inhomogeneous diffusion coefficients, i.e., $\rho_1=\rho_2=3\pi$ and $\rho_3=\rho_4=3$. The exact solution of the problem is given by
 	    \begin{align}\label{NE_3}
 	                  u(x,y)=
 	                  \begin{cases}
                             \sin(\pi(3x+y)) & \text{if } (x,y)\in \Omega_1,\Omega_2, \\
                             \sin(\pi(3\pi x+y)) & \text{if} {\ }(x,y)\in \Omega_3,\Omega_4.
                          \end{cases}
                          \end{align} 
 The source   function $f$ and $u_D$ are  manufactured  by the exact solution. 
 {\color{black}
 We mention that the interface conditions 
 $\llbracket u \rrbracket |_{F}=0$   and  $\llbracket \rho \nabla u \rrbracket |_{F}\cdot n_F=0$
 hold.
 }
 We have computed the convergence rates    for
 varying  size $d_M=h^\lambda$  for  $\lambda=1,2,2.5$ and $\lambda=3$, 
 where   $h$ is the maximum  size of the patch meshes.
  In Fig.~\ref{Fig3_Test_GapOverlap}(b), we plot the  contours of  $u$ on the domains $\Omega_1^*,\ldots,\Omega_4^*$ on  different meshes in case of $d_M=h$. 
 In Fig.~\ref{Fig3_Test_GapOverlap}(c), we plot the convergence rates. We observe that the values of the convergent rates for all  different cases of $\lambda$ confirm the theoretically  predicted values, see Table \ref{table_value_r}. The computational rates attain
 the optimal value $r=2$ for $\lambda=2.5$ and $\lambda = 3$, which is in agreement with the other examples and is the expected value for this problem with smooth solution.
By this example, we demonstrated the ability of the proposed method to approximate  oscillatory solutions of the diffusion problem (\ref{4b}) with the expected accuracy, in case of complex gap and overlap regions  using different patch meshes. 
\par
  The dG-IETI-DP method also performed in this case very well. Since we have only a small number of interface dofs, we get a quite small condition number and iteration count, see Table~\ref{IETI_Ex3}. However, due to the presence of quite few interface dofs, the usage of additional face averages as primal variables does not offer  any significant advantage.  
\begin{table}[h!]
\centering
\begin{tabular}{|r|c|cc|cc|cc|cc|}\hline
 \multicolumn{2}{|c|}{  }& \multicolumn{4}{|c|}{$W^{\mathcal{V}}$}& \multicolumn{4}{|c|}{$W^{\mathcal{V}+\mathcal{F}}$} \\ \hline
 \multicolumn{2}{|c|}{ } &  \multicolumn{2}{|c|}{coefficient scal.} & \multicolumn{2}{|c|}{stiffness scal.} &  \multicolumn{2}{|c|}{coefficient scal.} & \multicolumn{2}{|c|}{stiffness scal.} \\ \hline
	    $\#$ dofs & $H/h$ & $\kappa$ & It. & $\kappa$ & It. &$\kappa$ & It. & $\kappa$&  It. \\ \hline
	     184    & 2   &    1.12   &   4    &   1.11   &   5     &  1.11   &   4    &  1.1    &   4      \\ \hline
	     460    & 4   &    1.15   &   4    &   1.15   &   5     &  1.14   &   5    &  1.15   &   5     \\ \hline
	     1372   & 8   &    1.17   &   5    &   1.21   &   6     &  1.17   &   5    &  1.21   &   6     \\ \hline
	     4636   & 16  &    1.21   &   5    &   1.31   &   6     &  1.21   &   5    &  1.3    &   6     \\ \hline
	     16924  & 32  &    1.24   &   5    &   1.49   &   7     &  1.24   &   6    &  1.48   &   7     \\ \hline
	     64540  & 64  &    1.28   &   5    &   1.92   &   9     &  1.28   &   6    &  1.9    &   9     \\ \hline
	     251932 & 128 &    1.3    &   5    &   2.88   &   11    &  1.3    &   6    &  2.84   &   11   \\ \hline
  \end{tabular}
  \caption{Example 3, 2D example with inhomogeneous  diffusion coefficient on a domain with a complicated interface. Dependence of the condition number ($\kappa$) and the number  CG iterations (It.) on $H/h$ for the preconditioned system with coefficient and stiffness scaling. 
  Choice of primal variables: vertex evaluation (left), vertex evaluation and edge averages (right).}
\label{IETI_Ex3}
  \end{table}
  
            \begin{figure}
   \begin{subfigmatrix}{3}
 \subfigure[]{\includegraphics[width=5.7cm, height=2.75cm]{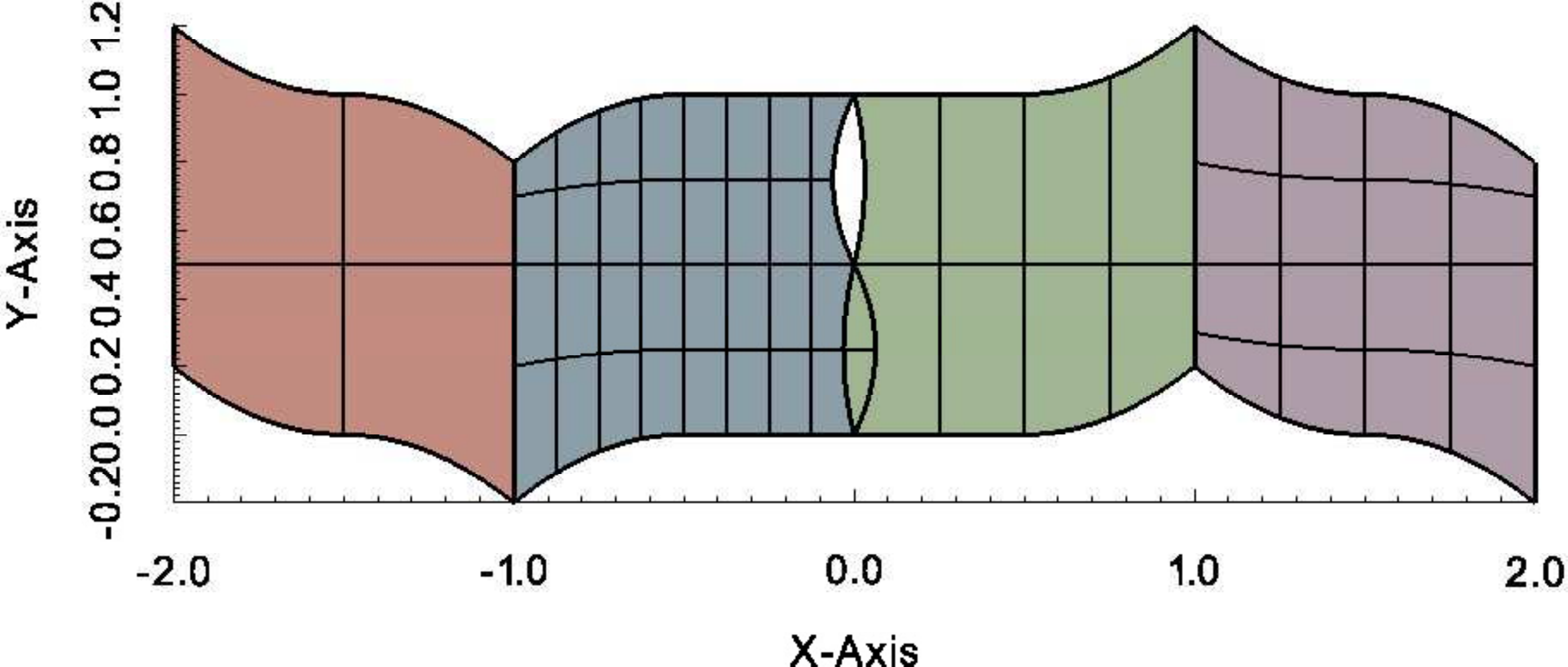}}
    \subfigure[]{\includegraphics[width=5.7cm, height=2.75cm]{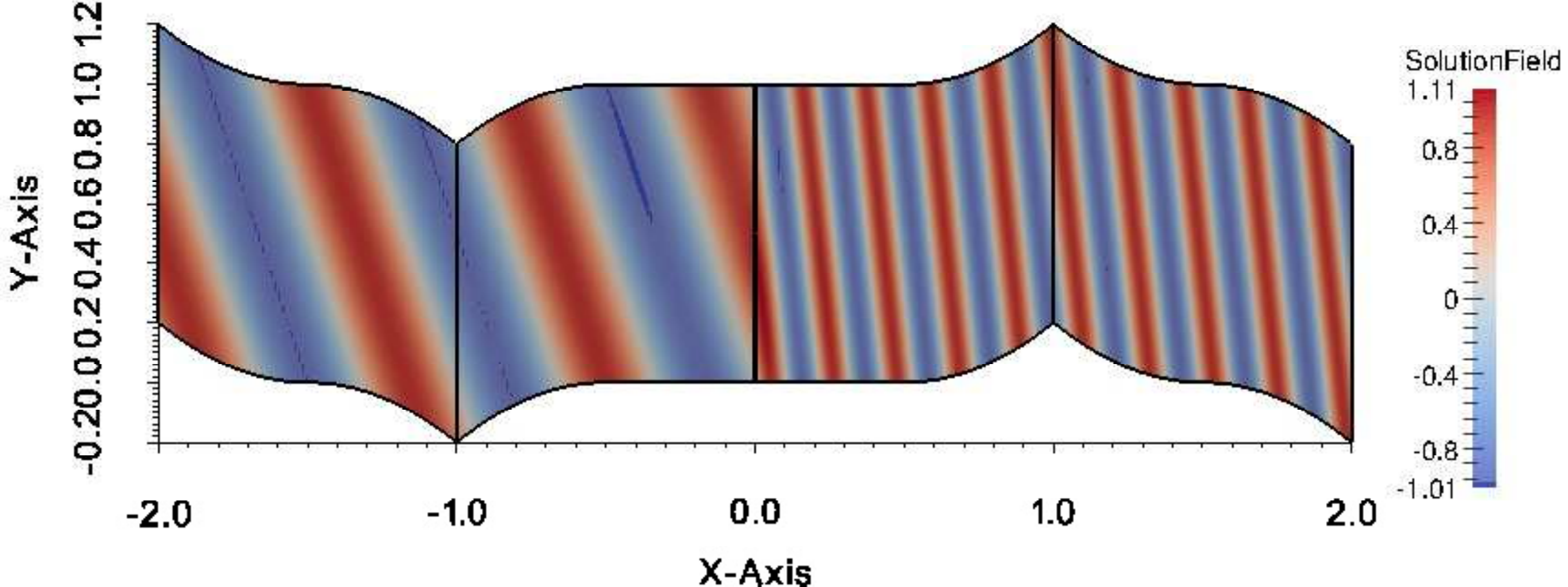}}
       \subfigure[]{\includegraphics[width=5.7cm, height=5.75cm]{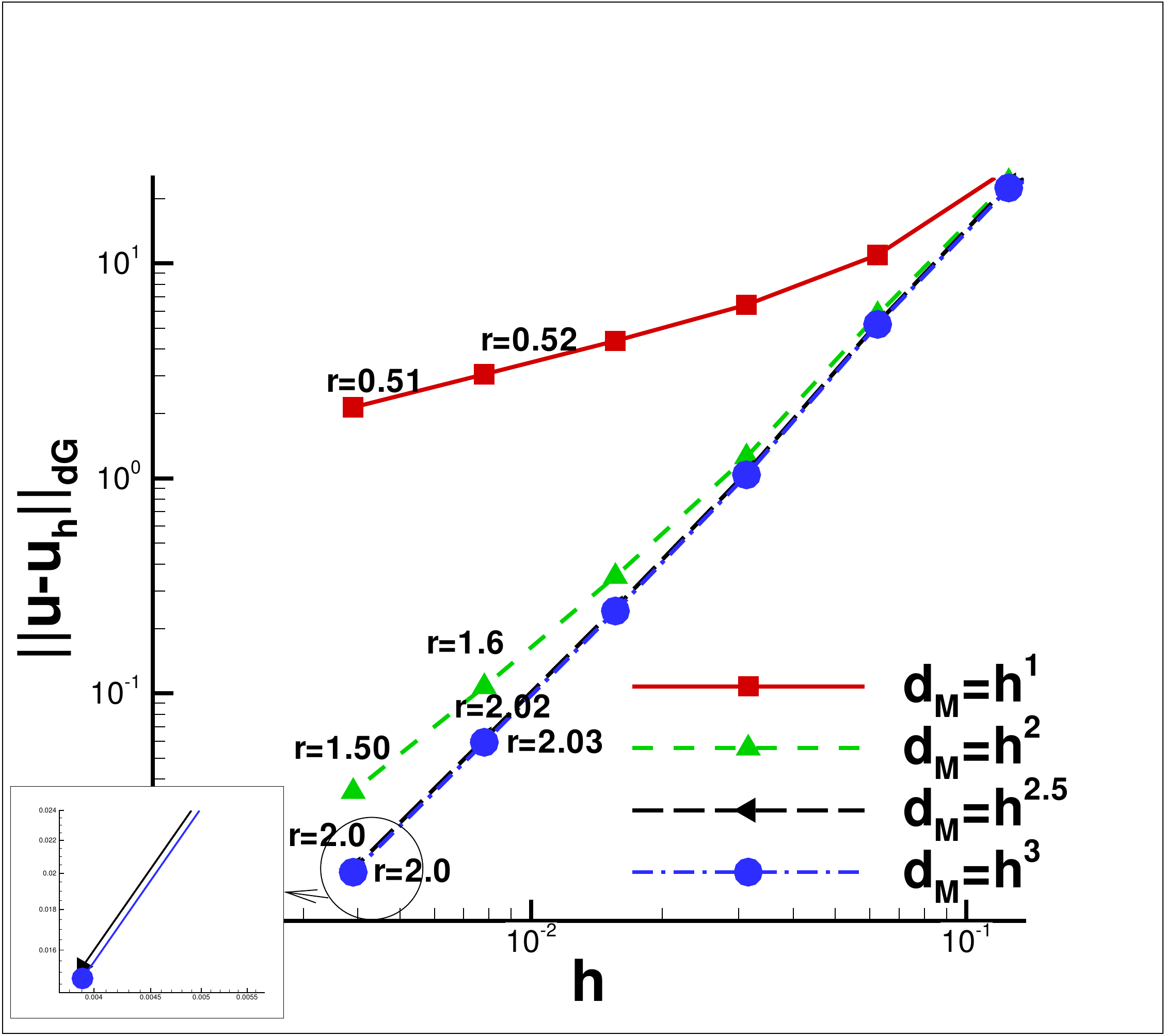}}
  \end{subfigmatrix}
   \caption{Example 3: (a) The 4 patches of the domain with its initial mesh and the gap and overlap region with $d_g = h$.
                       (b)  The contours of $u$ on $\Omega$ on the domain without gap  
                       (c) Convergence rates $r$ for the 4 values of $\lambda$.
}
   \label{Fig3_Test_GapOverlap}
 \end{figure}
 \subsection{Three-dimensional numerical examples}
 In the three-dimensional tests, the domain $\Omega$ 
 has been constructed by  a straight  prolongation to the $z$-direction 
 of the two dimensional domains of Fig.~\ref{Fig1_Test_1Gaps}(a) and Fig.~\ref{Fig3_Test_GapOverlap}(a). 
 However, in contrast to the two dimensional case, we start with matching meshes, as depicted in Fig.~\ref{Fig5_3DTest_1}(a) and Fig.~\ref{Fig6_3DTest}(a). The knot vector in $z$-direction  is simply $\Xi_i^3=\{0 , 0 ,0 , 0.5 ,1, 1, 1\}$ with $i=1,\ldots,N$. The B-spline parametrizations of these domains are constructed by adding a third component to the control points with the following values $\{0,0.5,1\}$. Again, the gap and overlap region is artificially constructed by moving only the interior control points located at the interface into the normal direction $n_F$ of the related interfaces $F$. Due to the fact that the gap and overlap has to be inside of the domain, we have to provide cuts though the domain in order to visualize them, cf. Fig.~\ref{Fig5_3DTest_1}(b) and Fig.~\ref{Fig6_3DTest}(b).
 
  	\paragraph{Example  4: 3d test with   $\rho_i=1$ for $i=1,\ldots,N.$} 
  	The computational domain is chosen as the domain from Example 1, extended to z-dimension as described above. The exact solution is given by 
  	{\color{black} $u(x,y,z)= \sin(\pi (x+0.4))\sin(2\pi(y+0.3))\sin(0.2\pi(z+0.6))$} 
  	with homogeneous diffusion coefficient $\rho_i=1$ for $i=1,\ldots N$.  
  	The set up of the problem is illustrated in Fig.~\ref{Fig5_3DTest_1}. 
  	In Fig.~\ref{Fig5_3DTest_1}(a), we present the domain $\Omega$ with its $N=21$ 
  	patches and the initial mesh. As already mentioned above, we use matching grids across the interface.
  	In Fig.~\ref{Fig5_3DTest_1}(b), we plot the contours of the solution $u_h$ resulting from the solution 
  	of the problem  in case of having a gap and overlapping widths such that $d_M=\max\{d_g,d_o\}=0.5$. 
  	Also, in Fig.~\ref{Fig5_3DTest_1}(b), we see the shape of the gaps  as it appears on  
  	an oblique cut of the domain $\Omega$. 
  	We note that, 
  	on those interfaces with no visible gap in Fig.~\ref{Fig5_3DTest_1}(b), 
  	there are overlapping regions, 
cf. Fig.~\ref{Fig1_Test_1Gaps}(a). We have computed the 
convergence rates for four different values  $\lambda\in\{1,2,2.5,3\}$.  
The results of the computed rates are plotted in Fig.~\ref{Fig5_3DTest_1}(c). 
We observe that the obtained rates are in agreement with the  convergent rates  
predicted by the theory, see estimate (\ref{4.5.d_1}) and Table \ref{table_value_r}.
         As already observed in \cite{HLT:HoferLanger:2016a} and \cite{HLT:HoferLanger:2016b}, 
         the condition number and CG iterations increases when having three dimensional domains. 
         The same behavior can be observed here, although 
         the condition number stays quite small, see Table~\ref{IETI_Ex4}. 
         Moreover, we observe that the stiffness scaling provides better results than the coefficient scaling. 
         For this example it is not beneficial to use additional face averages as primal variable.
         Both algorithms give identical results.
\begin{table}[tbh]
\centering
\begin{tabular}{|r|c|cc|cc|cc|cc|}\hline
 \multicolumn{2}{|c|}{  }& \multicolumn{4}{|c|}{$W^{\mathcal{V}+\mathcal{E}}$}& \multicolumn{4}{|c|}{$W^{\mathcal{V}+\mathcal{E}+\mathcal{F}}$} \\ \hline
 \multicolumn{2}{|c|}{ } &  \multicolumn{2}{|c|}{coefficient scal.} & \multicolumn{2}{|c|}{stiffness scal.} &  \multicolumn{2}{|c|}{coefficient scal.} & \multicolumn{2}{|c|}{stiffness scal.} \\ \hline
	    $\#$ dofs & $H/h$ & $\kappa$ & It. & $\kappa$ & It. &$\kappa$ & It. & $\kappa$&  It. \\ \hline
	      1488   & 2    &  1.2    &  10      &  1.19    &    10   & 1.2   &    10   & 1.19    &   10  \\ \hline
	      5508   & 3    &  7.87   &  26      &  4.45    &    22   & 7.87  &    26   & 4.45    &   22  \\ \hline
	      25908  & 7    &  10.9   &  32      &  6.31    &    26   & 10.9  &    32   & 6.31    &   26  \\ \hline
	      149748 & 15   &  13.5   &  36      &  8.17    &    30   & 13.5  &    36   & 8.17    &   30  \\ \hline
	      998388 & 31   &  16.3   &  40      &  10.2    &    34   & 16.3  &    40   & 10.2    &   34  \\ \hline
  \end{tabular}
  \caption{Example 4, 3D example. Dependence of the condition number ($\kappa$) and the number  CG iterations (It.) on $H/h$ for the preconditioned system with coefficient and stiffness scaling. 
  Choice of primal variables: vertex evaluation and edge averages (left), vertex evaluation, edge and face averages (right).}
\label{IETI_Ex4}
  \end{table}
  \begin{figure}
   \begin{subfigmatrix}{3}
 \subfigure[]{\includegraphics[width=5.0cm, height=4.050cm]{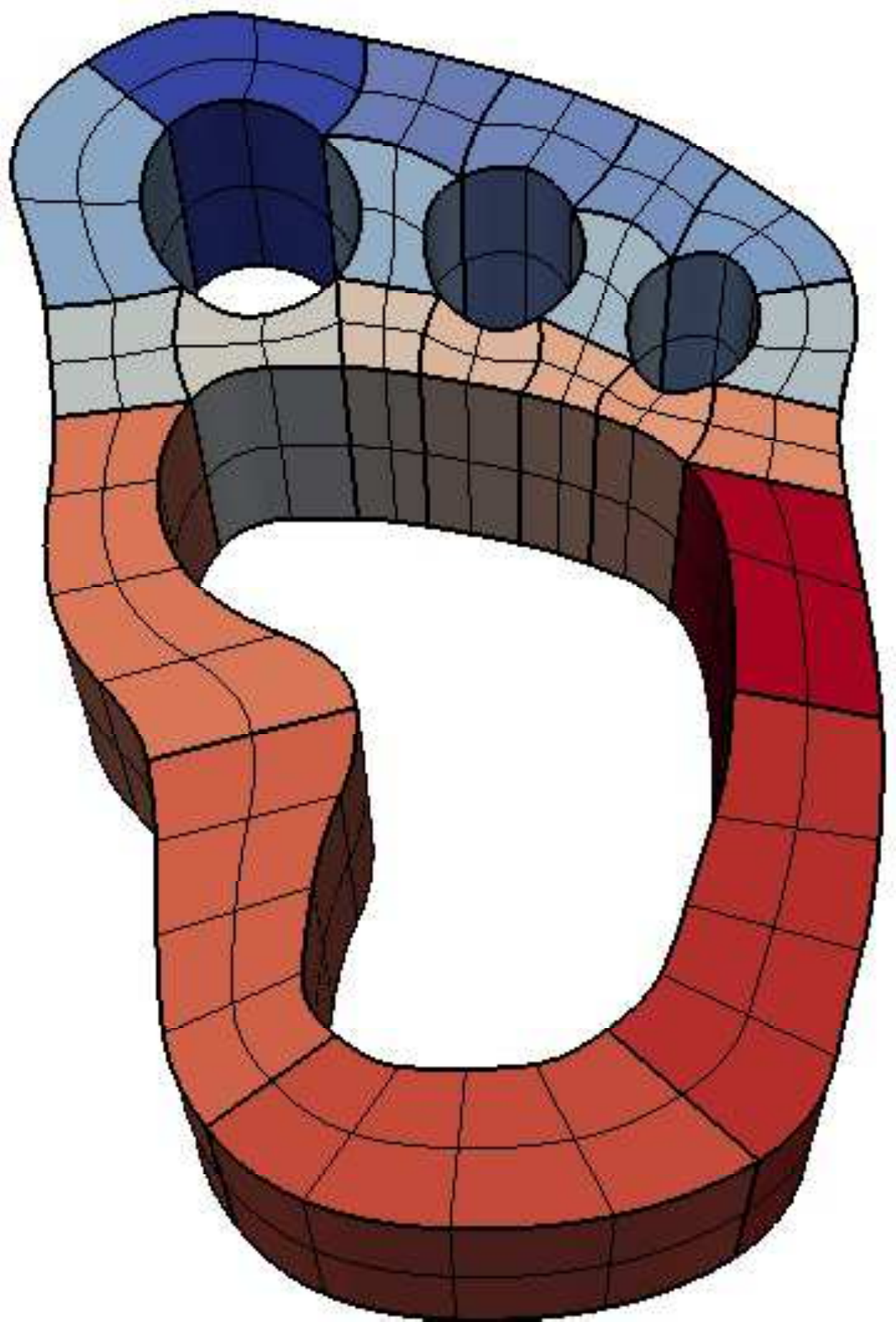}}
  \subfigure[]{\includegraphics[width=5.0cm, height=4.05cm]{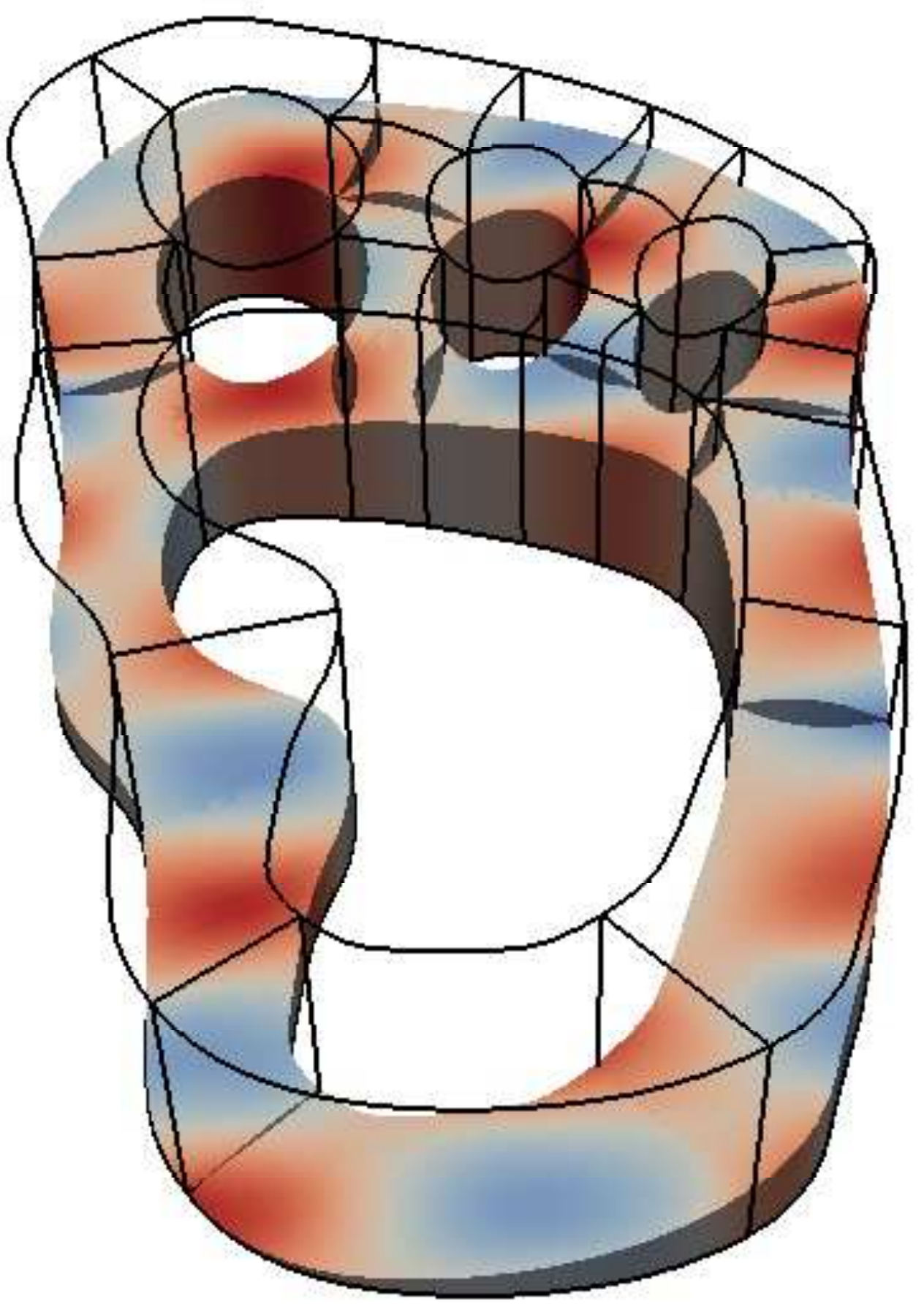}}
 \subfigure[]{\includegraphics[width=5.0cm, height=5.cm]{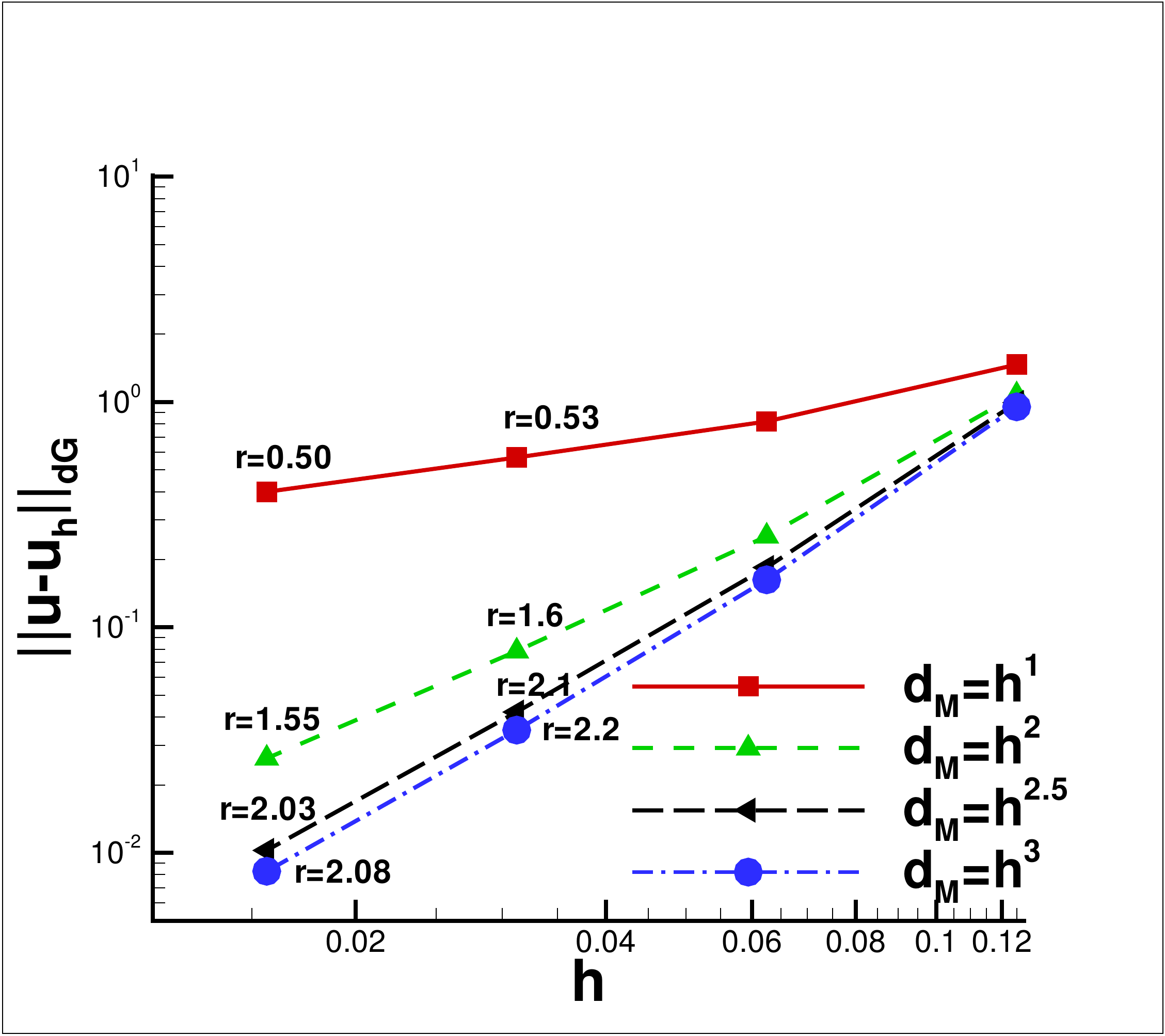}}
  \end{subfigmatrix}
   \caption{Example 4, $\Omega\subset \mathbb{R}^3$: (a) The multipatch system with initial mesh ,
                       (b) The contours of $u_h$ computed on $\Omega\setminus\overline{\Omega}_g$ with $d_g=0.05$, 
                       (c) Convergence rates $r$ for the three values of $\lambda$.}
   \label{Fig5_3DTest_1}
 \end{figure}
   \paragraph{Example  5: 3d gap-overlap region on one interface. } For the second numerical test in three-dimensions,  we choose the domain from Example 3 and extend it towards the $z$-direction  with matching grids. As in Example 3, we have a gap and an overlap simultaneously on an interface. This fact is visualized in Fig.~\ref{Fig3_Test_GapOverlap}(b). We consider a manufactured   problem, where the solution 
               \begin{align}\label{3D_test_2}
                  	u(x,y,z)=
  	                  \begin{cases}
                             u_1:=\sin(\pi (2 x+y+z))      & \text{if } (x,y,z) \in \Omega_1,\Omega_2 \\
                             u_2:= \sin(\pi(3\pi x+y+z)) & \text{if} {\ }(x,y,z)\in \Omega_3,\Omega_4
                           \end{cases}
                  \end{align}
 and the diffusion coefficient is defined to be  $\rho_1=\rho_2=2\pi$ and $\rho_3=\rho_4=3\pi$, as in Example 3. It is easy to see that we have the interface continuity conditions 
 {\color{black}
 $\llbracket  u \rrbracket |_{F} = 
 \llbracket \rho \nabla u\rrbracket |_{F} \cdot n_F =0$. 
 } 
{\color{black}
We solve the problem for $\lambda\in\{1,2,2.5,3\}$ as above.
}
 The contours of the solution $u$  on  $\Omega$ without gaps
 and overlapping regions are shown  in Fig.~\ref{Fig6_3DTest}(a). In Fig.~\ref{Fig6_3DTest}(b),  we plot 
  the contours  of the solution $u_h$ on an oblique cut through the domain $\Omega$, where the gap/overlap width is 
   $d_M=0.025$.
  We have computed the convergence rates $r$ for the four different sizes $d_M$, obtained by the different   values of $\lambda$. The results  are plotted in Fig.~\ref{Fig6_3DTest}(c). We observe that the rates are approaching the expected values that have been  mentioned in Table \ref{table_value_r}. 
 
 In this last example, the dG-IETI-DP method again provides very nice results, 
 summarized in Table~\ref{IETI_Ex5}. 
 Similar to Example 3, 
we only have three interfaces, therefore, only a small number of dofs corresponding to those.
 Hence, we again observe quite small condition numbers and iteration counts. 
 Both choices of primal variables have again identical performance, 
 but in contrast to the previous example, the coefficient scaling clearly provide better results than the stiffness scaling.
             
\begin{table}[h!]
\centering
\begin{tabular}{|r|c|cc|cc|cc|cc|}\hline
 \multicolumn{2}{|c|}{  }& \multicolumn{4}{|c|}{$W^{\mathcal{V}+\mathcal{E}}$}& \multicolumn{4}{|c|}{$W^{\mathcal{V}+\mathcal{E}+\mathcal{F}}$} \\ \hline
 \multicolumn{2}{|c|}{ } &  \multicolumn{2}{|c|}{coefficient scal.} & \multicolumn{2}{|c|}{stiffness scal.} &  \multicolumn{2}{|c|}{coefficient scal.} & \multicolumn{2}{|c|}{stiffness scal.} \\ \hline
	    $\#$ dofs & $H/h$ & $\kappa$ & It. & $\kappa$ & It. &$\kappa$ & It. & $\kappa$&  It. \\ \hline
	      424     &  2    &  1.21   &   7    & 2.09    &    9     &   1.21   &    7     & 2.09      &     9      \\ \hline
	      1416    &  3    &  1.27   &   7    & 3.18    &    12    &   1.27   &    7     & 3.18      &     12     \\ \hline
	      6376    &  7    &  1.36   &   7    & 4.17    &    16    &   1.36   &    7     & 4.17      &     16     \\ \hline
	      36264   &  15   &  1.36   &   7    & 3.65    &    16    &   1.36   &    7     & 3.65      &     16     \\ \hline
	      240424  &  31   &  1.35   &   7    & 3.05    &    14    &   1.35   &    7     & 3.05      &     14     \\ \hline
  \end{tabular}
  \caption{Example 5, 3D example with $p = 2$ and inhomogeneous  diffusion coefficient. Dependence of the condition number ($\kappa$) and the number  CG iterations (It.) on $H/h$ for the preconditioned system with coefficient and stiffness scaling. 
  Choice of primal variables: vertex evaluation and edge averages (left), vertex evaluation, edge and face averages (right).}
\label{IETI_Ex5}
  \end{table}
  
  \begin{figure}
   \begin{subfigmatrix}{3}
    \subfigure[]{\includegraphics[scale=0.15]{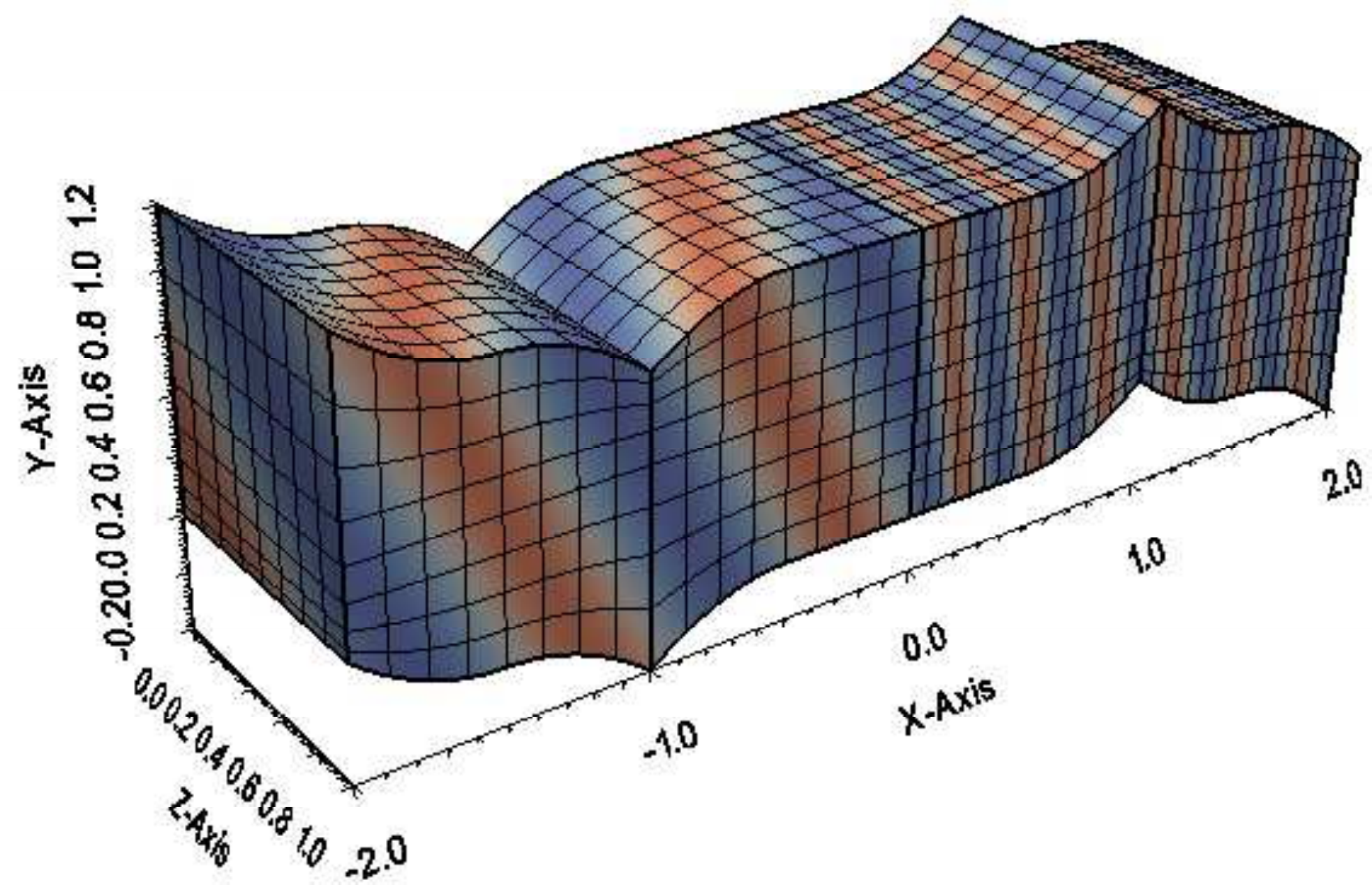}}
 \subfigure[]{\includegraphics[scale=0.15]{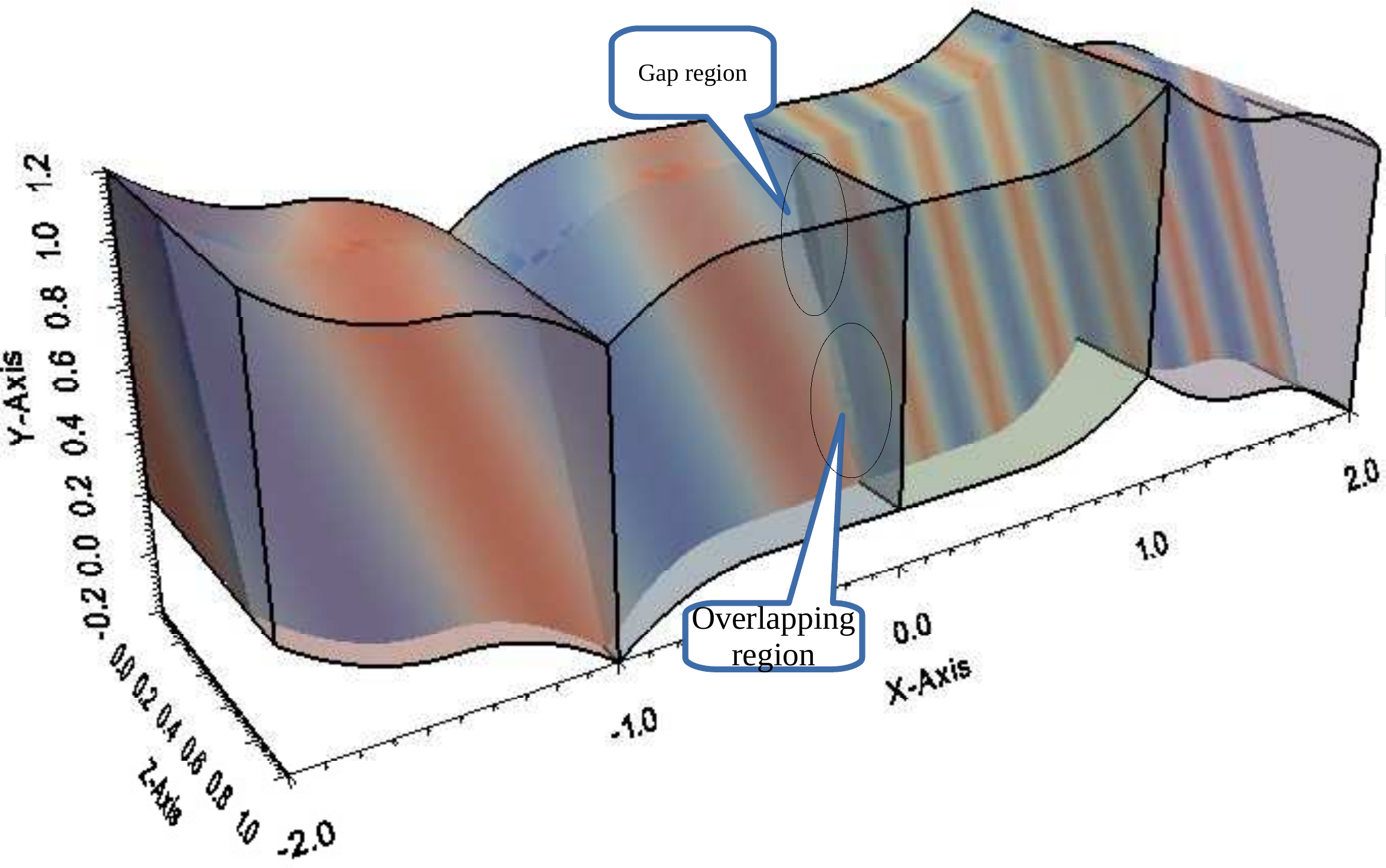}}
  \subfigure[]{\includegraphics[scale=0.015]{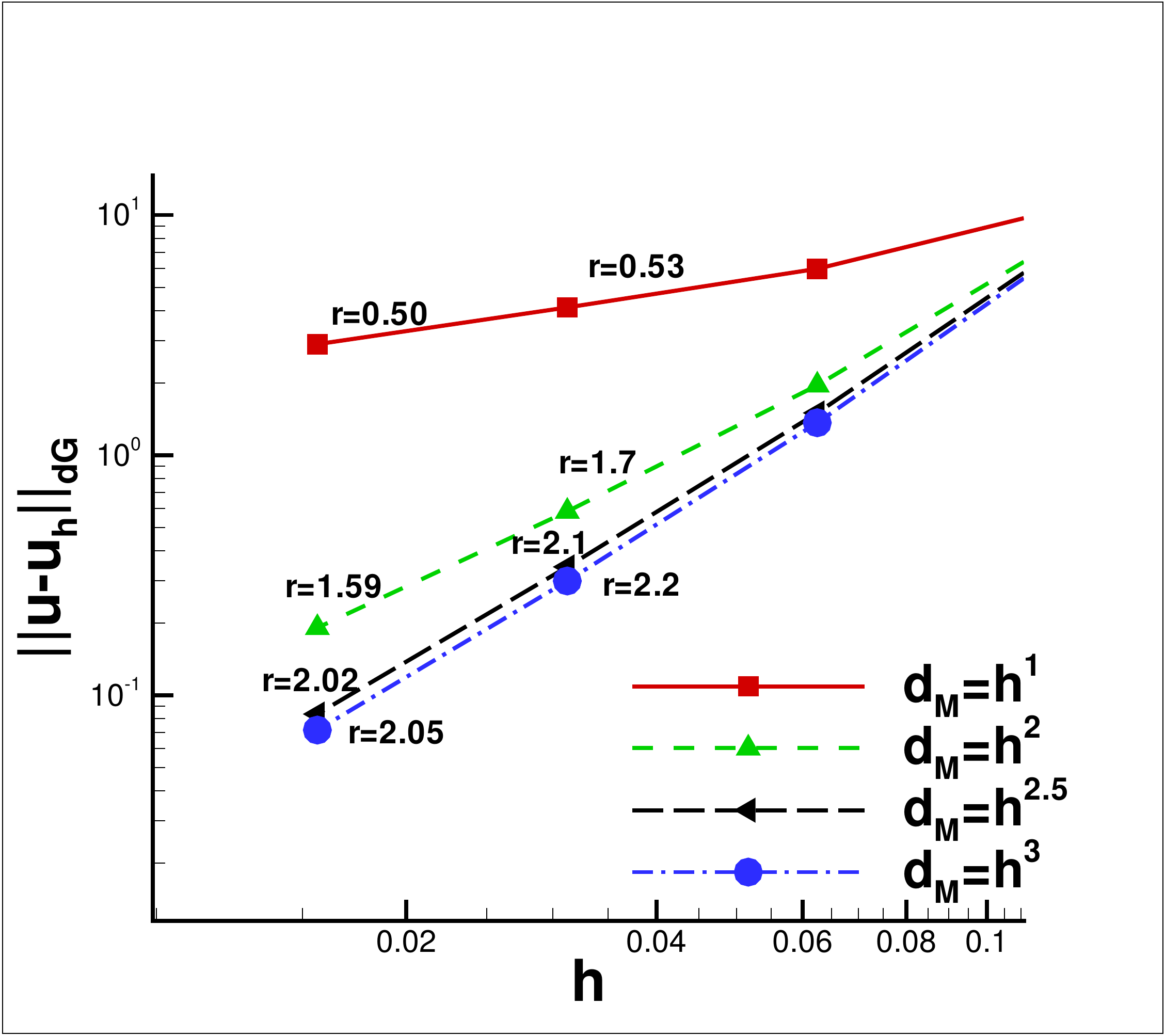}}
  \end{subfigmatrix}
   \caption{Example 5: (a) The contours of $u$ computed on $\Omega$,
                       (b) The contours of $u_h$ computed on $\Omega$ with $d_g=0.025$, 
                       (c) Convergence rates $r$ for the different $d_g$ sizes.}
   \label{Fig6_3DTest}
 \end{figure}     

\section{Conclusions}
In this article,  we have  constructed and analyzed dG IgA methods for discretizing linear,
second-order, scalar elliptic boundary value problems 
on volumetric patch decompositions
with non-matching interface parametrizations, which can include gap and overlapping regions.
Due to the appearance of  such segmentation crimes,
the direct  use of the standard dG numerical fluxes
for coupling the local  patch problems was not possible. 
Thus, the  normal fluxes on the gap and overlapping  boundaries were approximated
via  Taylor expansions using the interior values of the patches. 
The Taylor expansions  were appropriately used in  deriving
the dG-IgA  numerical scheme 
for constructing  the numerical fluxes, and finally helped on the  coupling of the local 
patchwise discrete problems. A priori error estimates in the  dG-norm 
were shown in terms of 
the mesh-size $h$ and the maximum width of the gap/overlapping regions.  
The estimates  were confirmed
by solving several two- and three- dimensional test problems with known exact solutions. 
The method were successfully applied  to the discretization of 
diffusion  problems in cases with     complex  gaps and overlaps 
using non-matching grids. The resulting linear problems were 
solved by means of the dG-IETI-DP method on geometries consisting of several patches.
Since the variational formulation is based on dG-IgA techniques, it is well suited for this method. 
The numerical examples showed that the presence of gap and overlap regions does not affect the solver performance. 
Hence, it is possible to solve these systems efficiently. 
Moreover, the dG-IETI-DP method is robust with respect to
large jumps in the diffusion coefficients across the patch boundaries (interfaces).

\section*{Acknowledgments}
This work was supported by the Austrian Science Fund (FWF) under the grant  S117-03 and  W1214-04.
\bibliographystyle{plain} 
\bibliography{Gaps_Overl_dGIgA}

\begin{thebibliography}{10}

\bibitem{Apostolatos_Schmidt_Wuchner_Bletzinger_IJNUMEng_2014}
A.~Apostolatos, R~Schmidt, R.~W\"uchner, and K.~U. Bletzinger.
\newblock A {N}itsche-type formulation and comparison of the most common domain
  decomposition methods in isogeometric analysis.
\newblock {\em Int. J. Numer. Meth. Engng}, 97:473--504, 2014.

\bibitem{HLT:VeigaChinosiLovadinaPavarino:2010a}
L.~{Beir\~ao da Veiga}, C.~{Chinosi}, C.~{Lovadina}, and L.~F. {Pavarino}.
\newblock {Robust BDDC preconditioners for Reissner-Mindlin plate bending
  problems and MITC elements.}
\newblock {\em {SIAM J. Numer. Anal.}}, 47(6):4214--4238, 2010.

\bibitem{HLT:BeiraoChoPavarinoScacchi:2013a}
L.~{Beir\~ao Da Veiga}, D.~{Cho}, L.F. {Pavarino}, and S.~{Scacchi}.
\newblock {BDDC preconditioners for isogeometric analysis.}
\newblock {\em {Math. Models Methods Appl. Sci.}}, 23(6):1099--1142, 2013.

\bibitem{bercovier2015overlapping}
M.~Bercovier and I.~Soloveichik.
\newblock Overlapping non matching meshes domain decomposition method in
  isogeometric analysis.
\newblock {\em arXiv preprint arXiv:1502.03756}, 2015.

\bibitem{Brivadis_IgAMortar_2015}
E.~Brivadis, A.~Buffa, B.~Wohlmuth, and L.~Wunderlich.
\newblock Isogeometric mortar methods.
\newblock {\em Computer Methods in Applied Mechanics and Engineering},
  284(0):292 -- 319, 2015.

\bibitem{LT:Hughes_IGAbook_2009}
J.~A. Cotrell, T.J.R. Hughes, and Y.~Bazilevs.
\newblock {\em Isogeometric {A}nalysis, Toward {I}ntegration of {CAD} and
  {FEA}}.
\newblock John Wiley and Sons, Sussex, United Kingdom, 2009.

\bibitem{VeigaBuffaSangalli_2014}
L.~Beir\~ao da~Veiga, A.~Buffa, G.~Sangalli, and R.~V\'azquez.
\newblock Mathematical analysis of variational isogeometric methods.
\newblock {\em Acta Numerica}, 23:157--287, 5 2014.

\bibitem{CarlDeBoor_Splines_2001}
C.~De-Boor.
\newblock {\em A {P}ractical {G}uide to {S}plines}, volume~27 of {\em Applied
  Mathematical Science}.
\newblock Springer, New York, 2001.

\bibitem{Maximiliam_Dryga_DG_DD}
M.~Dryja.
\newblock On discontinuous {G}alerkin methods for elliptic problems with
  discontinuous coefficients.
\newblock {\em Comput. Meth. Appl. Math.}, 3(1):76--85, 2003.

\bibitem{HLT:DryjaGalvisSarkis:2007a}
M.~{Dryja}, J.~{Galvis}, and M.~{Sarkis}.
\newblock {BDDC methods for discontinuous Galerkin discretization of elliptic
  problems.}
\newblock {\em {J. Complexity}}, 23(4-6):715--739, 2007.

\bibitem{HLT:DryjaGalvisSarkis:2013a}
M.~{Dryja}, J.~{Galvis}, and M.~{Sarkis}.
\newblock {A {FETI-DP} preconditioner for a composite finite element and
  discontinuous {G}alerkin method.}
\newblock {\em {SIAM J. Numer. Anal.}}, 51(1):400--422, 2013.

\bibitem{HLT:DryjaSarkis:2014a}
M.~Dryja and M.~Sarkis.
\newblock 3-d {FETI-DP} preconditioners for composite finite
  element-discontinuous {G}alerkin methods.
\newblock In {\em Domain Decomposition Methods in Science and Engineering XXI},
  pages 127--140. Springer, 2014.

\bibitem{ERN_FEM_book}
A.~Ern and J.-L. Guermond.
\newblock {\em Theory and Practice of Finite Elements}, volume 159 of {\em
  Applied Mathematical Sciences}.
\newblock Springer-Verlag New York, 2004.

\bibitem{HLT:HoferLanger:2016b}
C.~Hofer and U.~Langer.
\newblock Dual-primal isogeometric tearing and interconnecting methods.
\newblock In P.~Neittanmakki, J.~Periaux, and O.~Pironneau, editors, {\em
  Contributions to PDE for Applications}, Springer-ECCOMAS series
  ''Computational Methods in Applied Sciences''. Springer, Berlin, Heidelberg,
  New York, 2016.
\newblock to appear.

\bibitem{HLT:HoferLanger:2016a}
C.~Hofer and U.~Langer.
\newblock Dual-primal isogeometric tearing and interconnecting solvers for
  multipatch d{G}-{I}g{A} equations.
\newblock {\em Computer Methods in Applied Mechanics and Engineering}, 2016.
\newblock In Press, Accepted Manuscript,
  http://dx.doi.org/10.1016/j.cma.2016.03.031.

\bibitem{HoferLangerToulopoulos_2015a}
C.~Hofer, U.~Langer, and I.~Toulopoulos.
\newblock Discontinuous {G}alerkin {I}sogeometric {A}nalysis of {E}lliptic
  {D}iffusion {P}roblems on {S}egmentations with {G}aps.
\newblock {\em SIAM Journal on Scientific Computing}, 2016.
\newblock Accepted Manuscript.

\bibitem{HoferToulopoulos_IGA_Gaps_2015a}
C.~Hofer and I.~Toulopoulos.
\newblock Discontinuous {G}alerkin {I}sogeometric {A}nalysis of {E}lliptic
  {P}roblems on {S}egmentations with {N}on-matching {I}nterfaces.
\newblock {\em Computers and Mathematics with Applications}, 72:1811--1827,
  2016.

\bibitem{Hoschek_Lasser_CAD_book_1993}
J.~Hoschek and D.~Lasser.
\newblock {\em Fundamentals of {C}omputet {A}ided {G}eometric {D}esign}.
\newblock A K Peters, Wellesley, Massachusetts, 1993.

\bibitem{HLT:JuettlerKaplNguyenPanPauley:2014a}
B.~J\"uttler, M.~Kapl, D.-M. Nguyen, Q.~Pan, and M.~Pauley.
\newblock Isogeometric segmentation: The case of contractible solids without
  non-convex edges.
\newblock {\em Computer-Aided Design}, 57:74--90, 2014.

\bibitem{HLT:JuettlerLangerMantzaflarisMooreZulehner:2014a}
B.~J\"uttler, U.~Langer, A.~Mantzaflaris, S.E. Moore, and W.~Zulehner.
\newblock Geometry + {S}imulation {M}odules: {I}mplementing {I}sogeometric
  {A}nalysis.
\newblock {\em PAMM}, 14(1):961--962, 2014.

\bibitem{LangerMantzaflarisMooreToulopoulos_IGAA_2014a}
U.~Langer, A.~Mantzaflaris, St.~E. Moore, and I.~Toulopoulos.
\newblock {\em Multipatch Discontinuous {G}alerkin Isogeometric Analysis},
  volume 107 of {\em Lecture Notes in Computational Science and Engineering},
  pages 1--32.
\newblock Springer, Heidelberg, 2015.

\bibitem{LT:LangerToulopoulos:2014a}
U.~Langer and I.~Toulopoulos.
\newblock Analysis of {M}ultipatch {D}iscontinuous {G}alerkin {IgA}
  {A}pproximations to {E}lliptic {B}oundary {V}alue {P}roblems.
\newblock {\em Computing and Visualization in Science}, 17(5):217--233, 2016.

\bibitem{HLT:BeiraoPavarinoScacchiWidlundZampini:2014a}
L.{Beir\~ao Da Veiga}, L.F. {Pavarino}, S.~{Scacchi}, O.B. {Widlund}, and
  S.~{Zampini}.
\newblock {Isogeometric BDDC preconditioners with deluxe scaling.}
\newblock {\em {SIAM J. Sci. Comput.}}, 36(3):a1118--a1139, 2014.

\bibitem{gismoweb}
A.~Mantzaflaris, C.~Hofer, et~al.
\newblock {G+SMO} ({G}eometry plus {S}imulation {MO}dules) v0.8.1.
\newblock http://gs.jku.at/gismo, 2015.

\bibitem{Nguyen_Nitche_3Dcoupli_2014}
V.~P. Nguyen, P.~Kerfriden, M.~Brino, S.~P.~A. Bordas, and E.~Bonisoli.
\newblock Nitsche's method for two and three dimensional {NURBS} patch
  coupling.
\newblock {\em Computational Mechanics}, 53(6):1163--1182, 2014.

\bibitem{HLT:PauleyNguyenMayerSpehWeegerJuettler:2015a}
M.~Pauley, D.-M. Nguyen, D.~Mayer, J.~Speh, O.~Weeger, and B.~J\"uttler.
\newblock The isogeometric segmentation pipeline.
\newblock In B.~J\"uttler and B.~Simeon, editors, {\em Isogeometric Analysis
  and Applications IGAA 2014}, volume 107 of {\em Lecture Notes in Computer
  Science}, Heidelberg, 2015. Springer.

\bibitem{HLT:Pechstein:2013a}
C.~{Pechstein}.
\newblock {\em {Finite and boundary element tearing and interconnecting solvers
  for multiscale problems.}}
\newblock Berlin: Springer, 2013.

\bibitem{ERN_DGbook}
D.~A.~Di Pietro and A.~Ern.
\newblock {\em Mathematical {A}spects of {D}iscontinuous {G}alerkin {M}ethods},
  volume~69 of {\em {M}athématiques et {A}pplications}.
\newblock Springer-Verlag, 2010.

\bibitem{Rivierebook}
B.~Riviere.
\newblock {\em Discontinuous {G}alerkin methods for Solving Elliptic and
  Parabolic Equations}.
\newblock SIAM, Society for industrial and Applied Mathematics Philadelphia,
  2008.

\bibitem{Ruess_NitcCoplPtac_2015}
M.~Ruess, D.~Schillinger, A.~I. \"Ozcan, and E.~Rank.
\newblock Weak coupling for isogeometric analysis of non-matching and trimmed
  multi-patch geometries.
\newblock {\em Computer Methods in Applied Mechanics and Engineering},
  269(0):46 -- 71, 2014.

\bibitem{LT:Shumaker_Bspline_book}
L.~L. Schumaker.
\newblock {\em Spline {F}unctions: {B}asic {T}heory}.
\newblock Cambridge, University Press, third edition, 2007.

\bibitem{HLT:ToselliWidlund:2005a}
A.~{Toselli} and O.~B. {Widlund}.
\newblock {\em {D}omain {D}ecomposition {M}ethods -- {A}lgorithms and
  {T}heory.}
\newblock Berlin: Springer, 2005.

\end{thebibliography}

%

\end{document}